\documentclass[11pt]{article}
\usepackage{amsfonts}
\usepackage{amsthm}
\input{epsf.sty}
\usepackage{latexsym,amsmath,amsfonts,amscd}
\usepackage{epsfig}
\usepackage{changebar}
\usepackage{pstricks}
\usepackage{pst-plot}
\usepackage{multirow}
\usepackage{subfigure}
\usepackage{placeins}
\usepackage{amssymb}
\usepackage{mathrsfs}
\usepackage[title]{appendix}
\usepackage{bm}
\usepackage{hyperref}
\usepackage{geometry}
\usepackage{multirow,bigstrut}
\usepackage{enumitem}
\usepackage{booktabs}

\numberwithin{equation}{section}

\theoremstyle{plain}
\newtheorem{thm}{Theorem}[section]

\theoremstyle{definition}
\newtheorem{exam}[thm]{Example}

\newcommand{\brac}[1]{\left(#1\right)}
\newcommand{\abs}[1]{\left\vert#1\right\vert}

\allowdisplaybreaks[4]
\begin{document}
\baselineskip=1.6pc

\vspace{.5in}

\begin{center}

{\large\bf Machine learning moment closure models for the radiative transfer equation I: directly learning a gradient based closure}

\end{center}

\vspace{.1in}

\centerline{
Juntao Huang \footnote{Department of Mathematics,
Michigan State University, East Lansing, MI 48824, USA.
E-mail: huangj75@msu.edu} \qquad
Yingda Cheng  \footnote{Department of Mathematics, Department of Computational Mathematics, Science and Engineering, Michigan State University, East Lansing, MI 48824, USA. E-mail: ycheng@msu.edu. Research is supported by NSF grants    DMS-2011838 and AST-2008004.} \qquad
Andrew J. Christlieb \footnote{Department of Computational Mathematics, Science and Engineering, Michigan State University, East Lansing, Michigan 48824, USA. E-mail: christli@msu.edu. Research is supported by: AFOSR grants FA9550-19-1-0281 and FA9550-17-1-0394; NSF grants DMS-1912183 and AST-2008004; and DoE grant DE-SC0017955.} \qquad
Luke F. Roberts \footnote{National Superconducting Cyclotron Laboratory and Department of Physics and Astronomy, Michigan State University, East Lansing, MI 48824, USA. E-mail: robertsl@nscl.msu.edu. Research is supported by: NSF grant AST-2008004; and DoE grant DE-SC0017955.}
}

\vspace{.4in}

\centerline{\bf Abstract}
\vspace{.1in}

In this paper, we take a data-driven approach and apply machine learning to the moment closure problem for radiative transfer equation in slab geometry. Instead of learning the unclosed high order moment, we propose to directly learn the gradient of the high order moment  using neural networks. This new approach is consistent with the exact closure we derive for the free streaming limit and also provides a natural output normalization. 
A variety of benchmark tests, including the variable scattering problem, the Gaussian source problem with both periodic and reflecting boundaries, and the two-material problem, show both good accuracy and generalizability of our machine learning closure model.

\bigskip

\bigskip


\pagenumbering{arabic}

\section{Introduction}\label{sec:intro}
\setcounter{equation}{0}
\setcounter{figure}{0}
\setcounter{table}{0}

The radiative transfer equation (RTE) describes  particle propagation and interaction with a background medium. It has been widely applied in many fields of science and engineering including astrophysics \cite{pomraning2005equations}, heat transfer \cite{koch2004evaluation}, remote sensing \cite{spurr2001linearized}, and medical imaging \cite{klose2002optical}. The RTE is a high-dimensional integro-differential kinetic equation. Common numerical methods for computing RTE can be classified into two categories: probabilistic methods such as the direct simulation Monte Carlo (DSMC) method \cite{bird1995molecular}, and deterministic schemes including the discrete ordinates method ($S_N$) \cite{larsen2010advances} and the moment method \cite{chandrasekhar1944radiative,levermore1996moment} among others. In general, any mesh based numerical discretization faces formidable computational cost due to the curse of dimensionality.

While there have been advances in solving high dimensional models, such as the RTE, using various approaches in dimensional and model reduction \cite{crockatt2017arbitrary,crockatt2019hybrid,guo2016sparse,guo2017adaptive,buchan2015pod,tencer2016reduced,peng2021reduced}, for a large class of problems, moment methods are the only tractable solution.   Moment methods study the evolution of a finite number of moments of the specific intensity. 
Typically, in this scenario, the evolution of the $p^{th}$ moment depends on the $(p+1)^{th}$ moment, leading to what is known as the moment closure problem.   Hence, one has to introduce suitable closure relations that relates the highest moment with the lower order moments in order to get a closed system of equations.  A given closure relation makes assumptions about micro-physics, which may not be true in all settings. Therefore, the trade off in introducing a closure relation and solving a moment model instead of a kinetic equation, such as the RTE, is generic accuracy verses practical computability.   Many moment closure strategies  have been developed. Some of the best known methods include: the $P_N$ model \cite{chandrasekhar1944radiative}; the variable Eddington factor models \cite{levermore1984relating,murchikova2017analytic}; the entropy-based $M_N$ models \cite{hauck2011high,alldredge2012high,alldredge2014adaptive};  the positive $P_N$ models \cite{hauck2010positive}; the filtered $P_N$ models \cite{mcclarren2010robust,laboure2016implicit,radice2013new}; the $B_2$ models \cite{alldredge2016approximating}; and the $MP_N$ model \cite{fan2020nonlinear,fan2020nonlinear2,li2021direct}. Newly developed theory around the generic closure problem   suggests  an approach to constructing  analytical closure models based on gradients that lead to globally hyperbolic moment models  \cite{li2021direct}.  This approach encompasses many of the well known closure models and offers insight into this challenging problem.

Recently, thanks to the rapid development of machine learning (ML) \cite{lecun2015deep} and data-driven modeling \cite{brunton2016discovering,raissi2019physics,han2018solving,zhang2018deep}, a new approach to solving the moment closure problem has emerged based on ML. In \cite{han2019uniformly}, the authors introduced a framework to construct machine learning moment closure models for kinetic problems. They first learned a set of generalized moments using the auto-encoder to optimally represent the underlying velocity distribution, and then learned the moment closure model for the generalized moments with the aim of best capturing the associated dynamics of the kinetic equation. This framework was further applied to the Williams-Boltzmann equation for polydisperse evaporating sprays in \cite{scoggins2021machine}. In \cite{huang2020learning}, based on the conservation-dissipation formalism \cite{zhu2015conservation} of irreversible thermodynamics, the authors proposed a stable closure model  parametrized by multilayer perceptron (MLP) for the Boltzmann BGK equation.  In \cite{bois2020neural}, a nonlocal closure was proposed for the Vlasov-Poisson system using a convolutional neural network (CNN). In \cite{ma2020machine,wang2020deep}, the authors applied MLP, CNN and a discrete Fourier transform (DFT) network to learn the well-known Hammett–Perkins Landau fluid closure. In \cite{maulik2020neural}, the capability of neural networks to reproduce some known magnetized plasma closures was further investigated. We also note that in addition to the closure problem, ML is being investigated as a method for directly solving high dimensional kinetic equations.  The physics informed neural networks (PINN) was applied to solve forward and inverse problems for kinetic equations including the Boltzmann BGK model \cite{lou2020physics}, the phonon Boltzmann equation \cite{li2021physics} and also the RTE \cite{mishra2020physics}. In \cite{xiao2020using}, the full Boltzmann collision operator was approximated by a neural network with the aim of reducing the computational cost.

In this work, we focus on using ML as a tool for model reduction to address the moment closure problem of the RTE.  To close the moment model deduced from kinetic equations, the conventional approach is to provide an approximation to the unclosed high order moment. In optically thick regimes or intermediate regimes, it is easy to find an accurate closure relation. However, in optically thin regimes (or even the free streaming limit), the kinetic model does not possess intrinsic low dimensional structure, which makes any attempt at model reduction difficult \cite{lee2019deep,lee2020model,kim2020fast,rim2020depth}. To address this problem, we start from investigating the RTE in the free streaming limit and derive the exact closure relations with isotropic initial conditions. Motivated by this closure relation, we propose to directly learn the gradient of the unclosed moment using neural networks for the RTE in slab geometry.  
The advantages of our approach are twofold. First, the functional form of the model is consistent with the exact closure for the free streaming limit. Thus, it is expected to gain better accuracy using this ansatz, especially in the optically thin regime. Second, the unclosed high order moments usually have a wide range of magnitudes and become very small in the optically thick regime. Such a target function makes the neural network difficult to learn, unless an appropriate output normalization is applied \cite{bois2020neural}. Our approach in learning gradient provides a natural output normalization since the magnitude of $\partial_x m_{N+1}$ is close to that of $\partial_x m_N$, see equation \eqref{eq:learn-gradient-antasz} in Section \ref{sec:ml-closure}.
In addition, we incorporate the scale invariance of the closure model into the neural networks by learning the normalized gradient. 
We demonstrate numerically that enforcing scale invariance in the closure model makes 
the model more generalizable, especially when applied to initial data whose dynamic range is outside the training set used to create our ML closure model.

In our numerical tests, the training data is generated by initial conditions consisting of a truncated Fourier series with random coefficients \cite{ma2020machine,bois2020neural} and constant scattering and absorption coefficients. The well-trained model is uniformly accurate in the optically thick regime, intermediate regime and the optically thin regime. Moreover, the accuracy of our model is much better than the approach based on creating a ML closure directly trained to match the  moments, as well as the conventional $P_N$ closure and the filtered $P_N$ closure \cite{radice2013new}.  This is demonstrated on a wide range of 1D test problems,  including the variable scattering problem, the Gaussian source problem, the two-material problem, and the reflective boundary conditions.
An important observation is that our ML closure is able to nearly exactly reproduce the moments for the kinetic equation, even for the two-material problem, with a small number ($N=5$) of moments. The motivation for this choice is that, as shown below, at minimum four degrees of freedom are required to exactly close the moment equations in the free streaming limit with isotropic initial conditions. 
Therefore, it is natural to expect that more degrees of freedom are required in the closure for the variable scattering and absorption setting.  In our numerical tests, we find that giving our ML closure  the freedom to relate  the gradient of the sixth-order moment to the gradient of the first six moments is enough to produce accurate results for a variety of different regimes.

Hyperbolicity is another important property in moment closure models, which is difficult to enforce for traditional closure models \cite{cai2014globally,li2021direct} as well as ML models \cite{huang2020learning}. Our ML closure model is not able to preserve  hyperbolicity. We numerically stabilize the model by adding more numerical diffusion with larger penalty constants in the Lax-Friedrichs numerical flux. How to incorporate hyperbolicity in the ML closure model is certainly an interesting topic.   We also remark that we use ML to learn the gradient to close our system, while the contribution in \cite{li2021direct} enforces the hyperbolicity with knowledge of the gradients, but does not involve  ML. Incorporating the hyperbolicity in the ML closure model is an interesting topic, which is discussed in our subsequent works \cite{huang2021ml2,huang2021ml3}.

The remainder of this paper is organized as follows. In Section \ref{sec:moment-method}, we introduce the moment closure problem for the RTE in slab geometry. We derive the exact closure for the free streaming limit with isotropic initial conditions and propose the approach to directly learn the gradient of the unclosed high order moment using ML. In Section \ref{sec:training}, we present the details in data generation and the training of the neural networks. The effectiveness of our ML closure model is demonstrated through extensive numerical results in Section \ref{sec:numerical-test}. Some concluding remarks are given in Section \ref{sec:conclusion}.

\section{Moment closure for radiative transfer equation}\label{sec:moment-method}

In this section, we motivate a range of possible ML models. We start by  introducing the moment method for the RTE in slab geometry. Next, we derive the exact closure for the free streaming limit with isotropic initial conditions.  The ML closure formulations we produce are inspired by the exact free streaming closure. This includes the approach we propose based on directly learning the gradient of the unclosed high order moments using ML. At the end of this section, we go over the proposed functional forms for the various ML models.

\subsection{Moment method}

The time-dependent RTE for a gray medium in slab geometry has the form:
\begin{equation}\label{eq:rte}
	 \partial_t f + v \partial_x f = \sigma_s \brac{\frac{1}{2} \int_{-1}^1 f dv - f} -  \sigma_a f,
\end{equation}
where $f=f(x,v,t)$ is the specific intensity of radiation. The variable $v\in[-1, 1]$ is the cosine of the angle between the photon velocity and the $x$-axis. $\sigma_s = \sigma_s(x)\ge 0$ and $\sigma_a = \sigma_a(x)\ge 0$ are the scattering and absorption coefficients.

It is common to take moments of the RTE against Legendre polynomials. We denote the $k$-th order Legendre polynomial by $P_k = P_k(x)$ for $k\ge 0$. 
Next, we define the $k$-th order moment of the gray model as
\begin{equation}\label{eq:moment-definition-legendre}
	m_k(x,t) = \frac{1}{2} \int_{-1}^1 f(x,v,t) P_k(v) dv, \quad k\ge0.
\end{equation}
Multiplying \eqref{eq:rte} by $P_k(v)$, then integrating over $v\in[-1, 1]$ and using Bonnet’s recursion formula, we derive the moment equations up to moment $m_N$ as
\begin{equation}\label{eq:moment-eqn}
\begin{aligned}
	 \partial_t m_0 + \partial_x m_1 &= -  \sigma_a m_0, \\
	 \partial_t m_1 + \frac{1}{3} \partial_x m_0 + \frac{2}{3} \partial_x m_2  &= - ( \sigma_s + \sigma_a) m_1, \\
	& \cdots \\
	 \partial_t m_N + \frac{N}{2N+1} \partial_x m_{N-1} + \frac{N+1}{2N+1} \partial_x m_{N+1}  &= - ( \sigma_s +  \sigma_a) m_N. \\
\end{aligned}
\end{equation}
This truncated system is clearly not closed, since in the last equation the evolution of $m_N$ depends on $m_{N+1}$.
There are various ways to close the system, including the classical $P_N$ model \cite{chandrasekhar1944radiative}; the entropy-based $M_N$ models \cite{hauck2011high,alldredge2012high,alldredge2014adaptive};  the variable Eddington factor models \cite{levermore1984relating,murchikova2017analytic}; the positive $P_N$ model \cite{hauck2010positive}; the filtered $P_N$ ($FP_N$) models \cite{mcclarren2010robust,radice2013new,laboure2016implicit}; the $B_2$ models \cite{alldredge2016approximating}; and the $MP_N$ model \cite{fan2020nonlinear,fan2020nonlinear2,li2021direct}.    

{
In the numerical tests in Section \ref{sec:numerical-test}, we will compare our ML closure model with the $P_N$ model \cite{chandrasekhar1944radiative} and the $FP_N$ model \cite{radice2013new}. For the sake of completeness, we present the $P_N$ model and the $FP_N$ model here. The $P_N$ model assumes an ansatz of orthogonal polynomials in velocity space and the closure relation is $m_{N+1} = 0$.
Therefore, the resulting $P_N$ model is
\begin{equation}
	\partial_t \bm{m} + A \partial_x \bm{m} = S \bm{m},
\end{equation}
with $\bm{m} = (m_0, m_1, \cdots, m_N)^T$ and the coefficient matrix $A\in\mathbb{R}^{(N+1)\times(N+1)}$:
\begin{equation}\label{eq:coefficient-matrix-PN}
	A = 
	\begin{pmatrix}
    0 				& 	1 				&	0  			& 0  	&	\dots 	& 	0 	\\
    \frac{1}{3} 	& 	0 				& \frac{2}{3} 	& 0  	&	\dots 	& 	0	\\
    0 & \frac{2}{5} & 	0 				& \frac{3}{5} 	& \dots & 0 \\
    \vdots & \vdots &   \vdots 			& \ddots & \vdots & \vdots \\
    0 & 0 & 	\dots 				& \frac{N-1}{2N-1} 	& 0 & \frac{N}{2N-1} \\
    0 & 0 & \dots & 0 & \frac{N+1}{2N+1} & 0
	\end{pmatrix}
\end{equation}
and
\begin{equation}
	S = \textrm{diag}(-\sigma_a, -(\sigma_s + \sigma_a), \cdots, -(\sigma_s + \sigma_a)).
\end{equation}

The $FP_N$ model proposed in \cite{radice2013new} reads as
\begin{equation}\label{eq:FPN-moment-system}
	\partial_t \bm{m} + A \partial_x \bm{m} = S \bm{m} - \nu L \bm{m},
\end{equation}
with $\bm{m} = (m_0, m_1, \cdots, m_N)^T$. Here, the matrice $A$ and $S$ are the same with the $P_N$ closure. In the additional source term, $\nu > 0$ is a tunable parameter estimating the effective opacity of the filter and
\begin{equation}
	L = \textrm{diag}\brac{ l_0, l_1, l_2, \cdots, l_N }
\end{equation}
with
\begin{equation}
	l_k = \frac{\log\rho(\frac{k}{N+1})}{\log\rho(\frac{N}{N+1})}, \quad k=0,1,\cdots,N
\end{equation}
and $\rho$ is the filter function. In our numerical test, we follow \cite{radice2013new} and take $\rho(\eta) = \frac{1}{1+\eta^4}$.
}

All of these methods relate the $(p+1)^{th}$ moment to the $p^{th}$ moment as a way of closing the system.  Each of them have a set of pros and cons, but none of them do very well in the optically thin limit.  Our proposed ML closures have a functional form that is motivated by a limiting case, which we introduce in the next section.

\subsection{Exact closure for the free streaming limit}\label{sec:free-streaming}

For all moment closure models, it is very challenging to accurately compute the free streaming limit, i.e. when $\sigma_s$ is close to zero in \eqref{eq:rte}.
In such a transport dominated regime, the model does not possess an intrinsic low dimensional structure, which makes any attempt at model reduction difficult. In this part, we aim to find an exact closure in such a challenging case under some suitable assumptions. This will motivate the functional form of our ML model.

We focus on the free-streaming limit in the simplified case of 1D:
\begin{equation}\label{eq:free-streaming}
	 \partial_t f + v \partial_x f = 0,
\end{equation}
with an isotropic initial condition
\begin{equation}\label{eq:free-streaming-init}
	f(x,v,0) = f_0(x),
\end{equation}
which serves as a baseline for our ML closure models. 
Here, we do not pay attention to boundary conditions, and the solution is considered to be either periodic or compactly supported.

The exact solution to \eqref{eq:free-streaming}-\eqref{eq:free-streaming-init} is
\begin{equation}\label{eq:free-streaming-exact}
	f(x,v,t) = f_0(x-vt).
\end{equation}
Here, instead of using $P_k$, we define the $k$-th order moment by the projection with respect to a monomial basis:
\begin{equation}\label{eq:moment-definition-vk}
	n_k(x,t) = \int_{-1}^1 f(x,v,t)v^kdv, \quad k\ge0.
\end{equation}
We note that the monomial basis is equivalent to those defined by Legendre polynomials in \eqref{eq:moment-definition-legendre}. Here, we use the monomial basis moments  since it is easier to derive the exact closure relations.
In the free-streaming limit, these moments satisfy the equations:
\begin{equation}\label{eq:free-streaming-moment-equation}
	\partial_t n_k + \partial_x n_{k+1} = 0, \quad k\ge0.
\end{equation}
Plugging \eqref{eq:free-streaming-exact} into \eqref{eq:moment-definition-vk}, we have
\begin{equation}
	\begin{aligned}
		n_k(x,t) ={}& \int_{-1}^1 f_0(x-vt) v^k dv \\
				={}& \int_{x+t}^{x-t} f_0(w) \brac{\frac{x-w}{t}}^k \brac{-\frac{1}{t}} dw \\
				={}& t^{-(k+1)} \int_{x-t}^{x+t} f_0(w) (x-w)^k  dw,
	\end{aligned}	
\end{equation}
which implies
\begin{equation}
	t^{k+1} n_k(x,t) = \int_{x-t}^{x+t} f_0(w) (x-w)^k  dw.
\end{equation}
Taking a time derivative to the above equation yields
\begin{equation}
	t^{k+1} \partial_t n_k(x,t) + (k+1) t^{k} n_k(x,t) = f_0(x+t)(-t)^k + f_0(x-t) t^k.
\end{equation}
We immediately obtain
\begin{equation}
	t \partial_t n_k(x,t) + (k+1) n_k(x,t) = f_0(x+t)(-1)^k + f_0(x-t).
\end{equation}
We notice that the right-hand-side of the above equation only depends on $f_0(x\pm t)$ and the fact that $k$ is even or odd. Thus, we have that: for even $k$,
\begin{equation}
	t \partial_t n_k + (k+1) n_k = t \partial_t n_0 + n_0,
\end{equation}
and for odd $k$
\begin{equation}
	t \partial_t n_k + (k+1) n_k = t \partial_t n_1 + 2n_1,
\end{equation}
i.e. for any $k\ge0$,
\begin{equation}
	t \partial_t n_k + (k+1) n_k = \frac{1+(-1)^k}{2}(t \partial_t n_0 + n_0) + \frac{1-(-1)^k}{2}(t \partial_t n_1 + 2n_1).
\end{equation}

Next, in the above equation, we replace the time derivatives by the spatial derivatives using \eqref{eq:free-streaming-moment-equation}:
\begin{equation}
	- t \partial_x n_{k+1} + (k+1) n_k = \frac{1+(-1)^k}{2}(-t \partial_x n_1 + n_0) + \frac{1-(-1)^k}{2}(-t \partial_x n_2 + 2n_1),
\end{equation}
and then derive
\begin{equation}\label{eq:free-streaming-exact-t}
\begin{aligned}
	\partial_x n_{k+1} ={}& \frac{1+(-1)^k}{2} \partial_x n_1 + \frac{1-(-1)^k}{2} \partial_x n_2 \\
	& + \frac{1}{t}\brac{(k+1) n_k - \frac{1+(-1)^k}{2} n_0 - {(1-(-1)^k)}n_1}.
\end{aligned}
\end{equation}
Notice that the above relation provides an exact closure for $\partial_x n_{k+1}$. However, the closure has dependence on $t$. Next, we remove the dependence on $t$. Taking $k=2$ in \eqref{eq:free-streaming-exact-t} yields
\begin{equation}
	\partial_x n_3 = \partial_x n_1 + \frac{1}{t} (3n_2 - n_0),
\end{equation}
then eliminating $t$ in \eqref{eq:free-streaming-exact-t}
\begin{equation}\label{eq:free-streaming-exact-closure}
\begin{aligned}
	\partial_x n_{k+1} ={}& \frac{1+(-1)^k}{2} \partial_x n_1 + \frac{1-(-1)^k}{2} \partial_x n_2 \\
	& + \frac{(k+1) n_k - \frac{1+(-1)^k}{2} n_0 - {(1-(-1)^k)}n_1}{3n_2 - n_0} \brac{\partial_x n_3 - \partial_x n_1}
\end{aligned}	
\end{equation}
This is an exact closure for $\partial_x n_{k+1}$. We remark that this closure holds for any $k\ge 0$ but it reduces to the trivial case for $k=0,1,2$ and only make senses for $k\ge3$.
Since the two sets of moments defined in \eqref{eq:moment-definition-legendre} and \eqref{eq:moment-definition-vk} are equivalent, it is easy to   derive similar closure relations for $\partial_x m_{k+1}$.

{
We also point out that the exact closure \eqref{eq:free-streaming-exact-closure} cannot be written into a conservative form:
\begin{thm}
The exact closure \eqref{eq:free-streaming-exact-closure} for $k\ge3$ cannot be written into a conservative form. To be more precise, there exists no smooth function $F$ such that \eqref{eq:free-streaming-exact-closure} can be written as
\begin{equation}\label{eq:conservation-exact-closure}
    n_{k+1} = F(n_0, n_1, \cdots, n_k).
\end{equation}
\end{thm}
\begin{proof}
We prove by contradiction. Assume that there exist a smooth function $F$ such that \eqref{eq:conservation-exact-closure} holds true. Using the equality of mixed partial derivatives, we have
\begin{equation}
    \frac{\partial }{\partial n_3}\brac{ \frac{1-(-1)^k}{2}} = \frac{\partial }{\partial n_2}\brac{\frac{(k+1) n_k - \frac{1+(-1)^k}{2} n_0 - {(1-(-1)^k)}n_1}{3n_2 - n_0}}
\end{equation}
This is equivalent to
\begin{equation}
    (k+1) n_k - \frac{1+(-1)^k}{2} n_0 - {(1-(-1)^k)}n_1 = 0
\end{equation}
Plugging \eqref{eq:free-streaming-exact}-\eqref{eq:moment-definition-vk} into the above equality, we have
\begin{equation*}
\begin{aligned}
0 ={}& (k+1) \int_{-1}^1 f_0(x-vt) v^k dv - \frac{1+(-1)^k}{2} \int_{-1}^1 f_0(x-vt)  dv - {(1-(-1)^k)} \int_{-1}^1 f_0(x-vt) v dv \\
={}& \int_{-1}^1 f_0(x-vt) \brac{(k+1)v^k - \frac{1+(-1)^k}{2} - {(1-(-1)^k)} v }dv.
\end{aligned}
\end{equation*}
For any fixed $k\ge3$, we can always find $f_0$ such that the above relation does not hold true.
\end{proof}
}

\subsection{Machine learning closure model}\label{sec:ml-closure}

In this section, armed with what we just learned about the functional form of moment closures in the free-streaming limit in 1D, we introduce several  models we will explore as possible closure models. As part of this work on closing moment based models of the RTE, we explore   the effectiveness of each of these models. 

A standard approach for moment closures is to find a relation between $m_{N+1}$ and the lower order moments:
\begin{equation}\label{eq:learn-moment-antasz}
	m_{N+1} = \mathcal{N}(m_0,m_1,\cdots,m_N).
\end{equation}
The ML based approach will find $\mathcal{N}:\mathbb{R}^{N+1}\rightarrow\mathbb{R}$ that is represented by a neural network and trained from data. This is the regression in supervised learning and also a part of the end-to-end learning procedure in \cite{han2019uniformly}. We call this approach Learning the Moment (LM).
However, in our numerical implementation, we find that the training process usually gets stuck in local minimum when using this approach.  Hence, the model has difficulty fitting the data well, see the detailed discussions of Figure \ref{fig:train-loss} in Section \ref{sec:training}. 
\begin{figure}
    \centering
    \includegraphics[width=0.5\textwidth]{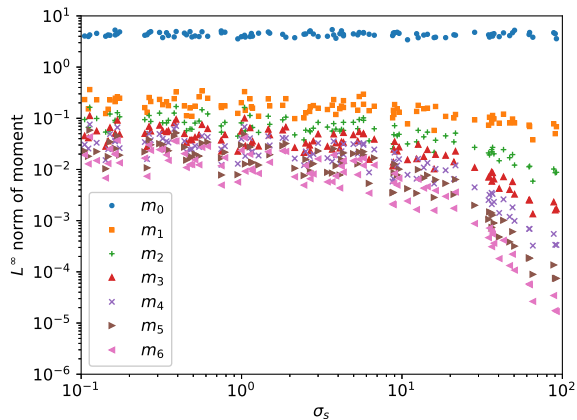}
    \caption{Magnitudes of moments $m_k$ for $k=0,\cdots,6$ with different scattering coefficients in the training dataset.}
    \label{fig:moment-magnitude}
\end{figure}

We also experiment with another approach, which we refer to as the weighted moment model.  The weighted moment model recognizes the fact that equation \eqref{eq:free-streaming-exact-closure} could loosely be viewed as relating many weighted moments as a closure.  The Learning Weighted Moment (LWM) model takes the following form:
\begin{equation}\label{eq:learn-moment-weighted-antasz}
	m_{N+1} = \sum_{k=0}^N \mathcal{N}_k(m_0,m_1,\ldots,m_N) m_k.
\end{equation}
Here $\mathcal{N}=(\mathcal{N}_0, \ldots, \mathcal{N}_N): \mathbb{R}^{N+1}\rightarrow\mathbb{R}^{N+1}$ will be represented by a neural network.
This formulation can deal with normalization issues. However, as shown in Figure \ref{fig:train-loss} in Section \ref{sec:training}, its behaviour is similar to that of the simpler Learn Moment model,  in \eqref{eq:learn-moment-antasz}.

Motivated by the exact closure \eqref{eq:free-streaming-exact-closure} for the free streaming limit, we propose to directly learn the gradient of the unclosed moment. Specifically, we assume that $\partial_x m_{N+1}$ depends linearly on the gradients of the lower order moments with the coefficients being functions of the lower order moments:
\begin{equation}\label{eq:learn-gradient-antasz}
	\partial_x m_{N+1} = \sum_{k=0}^N \mathcal{N}_k(m_0,m_1,\ldots,m_N)\partial_x m_k.
\end{equation}
Here $\mathcal{N}=(\mathcal{N}_0, \ldots, \mathcal{N}_N): \mathbb{R}^{N+1}\rightarrow\mathbb{R}^{N+1}$ will be represented by a neural network and trained from data. We call this approach learning the gradient (LG).

The advantages of directly learning the gradient in \eqref{eq:learn-gradient-antasz} are twofold. First, this ansatz \eqref{eq:learn-gradient-antasz} is consistent with the exact closure for the free streaming limit \eqref{eq:free-streaming-exact-closure}. Therefore, it is expected to have better accuracy, especially in the optically thin regime. Second, the unclosed high order moments usually have a wide range of magnitudes and become very small in the optically thick regime. In Figure \ref{fig:moment-magnitude}, we show the $L^{\infty}$-norm (in space and time) of moments $m_k$ for $k=0,\cdots,6$ with different scattering coefficients. It is observed that the magnitude of $m_6$ ranges from $10^{-5}$ to $10^{-2}$. Such a target function makes the neural network difficult to learn, unless appropriate output normalization is applied. Similar problems were also noticed in closing the Vlasov-Poisson equation in \cite{bois2020neural}, where the output normalization technique was applied to the heat flux with its estimation given by the Navier-Stokes approximation. Our antasz \eqref{eq:learn-gradient-antasz} provides a natural output normalization since the magnitude of $\partial_x m_{N+1}$ is similar to that of $\partial_x m_N$.

We further incorporate the scale invariance of the closure model into the neural networks and {learn the gradient with normalized moments (LGNM)}:
\begin{equation}\label{eq:learn-gradient-linear-antasz}
	\partial_x m_{N+1} = \sum_{k=0}^N \mathcal{N}_k\brac{\frac{m_1}{m_0},\frac{m_2}{m_0},\ldots,\frac{m_N}{m_0}}\partial_x m_k.
\end{equation}
Here $\mathcal{N}=(\mathcal{N}_0, \ldots, \mathcal{N}_N):\mathbb{R}^{N}\rightarrow\mathbb{R}^{N+1}$ will be replaced by a neural network and trained from data. {Given the linearity of the RTE in \eqref{eq:rte}, LGNM in \eqref{eq:learn-gradient-linear-antasz} is expected to have better performance than LG in \eqref{eq:learn-gradient-antasz}: (1) If $f$ is scaled by a constant in \eqref{eq:rte}, the antasz \eqref{eq:learn-gradient-linear-antasz} can provide exactly the same prediction; (2) LGNM in \eqref{eq:learn-gradient-linear-antasz} removes the linear redundancy of the training data and should be more data efficient; (3) LGNM in \eqref{eq:learn-gradient-linear-antasz} does not sacrifice any expressive ability.} This results in a ML model with better generalization performance, especially for testing data with a totally different magnitude. This will be further investigated numerically, see Figure \ref{fig:gauss-source-scale} in Example \ref{ex:gauss-source} in Section \ref{sec:numerical-test}.

We also point out a disadvantage of the ML closure model with LG in \eqref{eq:learn-gradient-antasz} or LGNM in \eqref{eq:learn-gradient-linear-antasz}. This breaks the conservation property of the last equation in \eqref{eq:moment-eqn} and generates a partially-conservative system. This might lead to problems in numerical implementations, especially when shocks exists \cite{koellermeier2017numerical,koellermeier2021high}.

\section{Training of the neural network}\label{sec:training}

In this section, we present the details of generating the training data and the training of the neural network.  These details are provided for reproduce-ability of the results.

\subsection{Data preparation}
In this section, we outline how we generate the data we use for training our ML model.  The data comes from simulating a simple 1D slab RTE problem over a range of initial conditions and scattering and absorption coefficients.   The data is curated and used to create $M$ time  snapshots of the $N$ moments, $m_{N+1}^{true}$, of the kinetic solution over a fixed time window.  This data is used in training our ML model (see section \ref{sec:subsec-training}).   We now go over the details for the creation of this data.

We consider the unit interval $[0, 1]$ in the physical domain with periodic boundary conditions. Following \cite{ma2020machine,bois2020neural}, we take the initial conditions to be an isotropic distribution in the form of a truncated Fourier series:
\begin{equation}\label{eq:init-fourier-series}
	f_0(x,v) = a_0 + \sum_{k=1}^{k_{\max}} a_k\sin(2k\pi x + \phi_k).
\end{equation}
Here, we take $k_{\max}=10$ in our dataset. For $k\ge1$, $a_k$ and $\phi_k$ are random variables sampled from the uniform distributions on $[-\frac{1}{k}, \frac{1}{k}]$ and $[0, 2\pi]$, respectively. We take $a_0=c+\sum_{k=1}^{k_{\max}}\frac{1}{k}$ with $c$ a random variable sampled from the uniform distributions on $[0,1]$. This guarantees the positivity of the distribution function. Both $\sigma_s$ and $\sigma_a$ are constants over the domain for each run. The scattering coefficient $\sigma_s$ is sampled from a log-uniform distribution on $[0.1, 100]$. The absorption coefficient $\sigma_a$ are randomly sampled from $[0, 10]$. A possible function generated with \eqref{eq:init-fourier-series} are shown in Figure \ref{fig:init-profile}. In the current work, we train with 100 different initial data.
\begin{figure}
    \centering
    \includegraphics[width=0.5\textwidth]{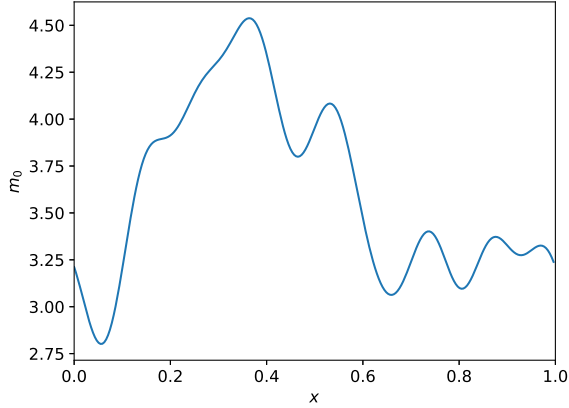}
    \caption{A possible initial condition generated from the truncated Fourier series in the training dataset.}
    \label{fig:init-profile}
\end{figure}

The space time discontinuous Galerkin (DG) method  \cite{crockatt2017arbitrary,crockatt2019hybrid} is applied to solve the RTE in slab geometry \eqref{eq:rte}. Piecewise polynomials of degree 2 in space and degree 1 in time are applied. We take 64 Gauss-Legendre quadrature points to discretize the velocity space. We take the number of grid points in space to be $N_x=512$. The CFL condition is taken to be $\Delta t=8\Delta x$ and the final time is $t=1$.  In this work, we use the $N$ moments of every time step taken for each of the 100 initial conditions to form our ground truth data set.

\subsection{Training}\label{sec:subsec-training}
For our neural network model, we use a fully-connected neural network and choose to use the hyperbolic tangent for our  activation function. The number of layers is taken to be 6 and the number of nodes in each hidden layer is taken to be 256, unless otherwise stated. {The input normalization is applied: each component of the input is linearly scaled to have zero mean value and unit variance.} For the training of the neural network, we take 1000 total epochs (the number of iterations in the optimization process). The learning rate is set to be $10^{-3}$ in the initial epoch and decays by 0.35 every 100 epochs. $L^2$ regularization is applied with weight $10^{-7}$. The batch size is taken to be 1024. The training is implemented within the PyTorch framework \cite{paszke2019pytorch}. We refer readers to \cite{higham2019deep} for the details on the basic concepts of ML.

For the moment model and  weighted moment model, \eqref{eq:learn-moment-antasz} and \eqref{eq:learn-moment-weighted-antasz}, the loss function is taken to be the mean squared error (MSE):
\begin{equation}
    \mathcal{L} = \frac{1}{N_{\textrm{data}}} \sum_{j,n}\abs{{m^{\textrm{true}}_{N+1}(x_j, t_n) - m^{\textrm{appx}}_{N+1}(x_j, t_n)}}^2.
\end{equation}
Here, $m^{\textrm{true}}_{N+1}(x_j, t_n)$ denote the $(N+1)$-th order moment at $x=x_j$ and $t=t_n$ computed from the kinetic solver and $m^{\textrm{appx}}_{N+1}(x_j, t_n)$ is the neural network given in equations \eqref{eq:learn-moment-antasz} or \eqref{eq:learn-moment-weighted-antasz}, the moment model and weighted moment model respectively.  Here, $N_{\textrm{data}}$ denotes the total count of the data used in  training the neural network. For the gradient model and  normalized gradient model, \eqref{eq:learn-gradient-antasz} and \eqref{eq:learn-gradient-linear-antasz}, the loss function is taken to be:
\begin{equation}
    \mathcal{L} = \frac{1}{N_{\textrm{data}}}\sum_{j,n}\abs{{\partial_x m^{\textrm{true}}_{N+1}(x_j, t_n) - \partial_x m^{\textrm{appx}}_{N+1}(x_j, t_n)}}^2.
\end{equation}
Here, $\partial_x m^{\textrm{true}}_{N+1}(x_j, t_n)$ denotes  the spatial derivative of $(N+1)$-th order moment at $x=x_j$ and $t=t_n$ computed from the kinetic solver and $\partial_x m^{\textrm{appx}}_{N+1}(x_j, t_n)$ comes from the evaluation of the neural network using \eqref{eq:learn-gradient-antasz} or \eqref{eq:learn-gradient-linear-antasz}.

In the training process, we find that the approach to learn the moment in \eqref{eq:learn-moment-antasz} and learn the weighted moment in \eqref{eq:learn-moment-weighted-antasz} usually get stuck in a local minimum and does not fit the data well, no matter how we tune the hyperparameters. As a comparison, the approach to learn the gradient in \eqref{eq:learn-gradient-antasz} and \eqref{eq:learn-gradient-linear-antasz} has much smaller relative errors, see Figure \ref{fig:train-loss}. Moreover, we observe that increasing the number of moments will result in the smaller training error in Figure \ref{fig:train-loss}. The relative $L^2$ error in Figure \ref{fig:train-loss} is defined to be the relative error between the target function and the approximated function. Specifically, for  the moment model, \eqref{eq:learn-moment-antasz} and \eqref{eq:learn-moment-weighted-antasz}, the relative $L^2$ error for the moment models is defined to be
\begin{equation}
	E_2 = \sqrt{ \frac{\sum_{j,n}({m^{\textrm{true}}_{N+1}(x_j, t_n) - m^{\textrm{appx}}_{N+1}(x_j, t_n)})^2}{\sum_{j,n}({m^{\textrm{true}}_{N+1}(x_j, t_n)})^2} }.
\end{equation}
Likewise, for the gradient model, \eqref{eq:learn-gradient-antasz} and \eqref{eq:learn-gradient-linear-antasz}, the relative $L^2$ error for the gradient model is defined to be
\begin{equation}
	E_2 = \sqrt{ \frac{\sum_{j,n}({\partial_x m^{\textrm{true}}_{N+1}(x_j, t_n) - \partial_x m^{\textrm{appx}}_{N+1}(x_j, t_n)})^2}{\sum_{j,n}({\partial_x m^{\textrm{true}}_{N+1}(x_j, t_n)})^2} }.
\end{equation}

\begin{figure}
    \centering
    \subfigure[$N=1$]{
    \begin{minipage}[b]{0.46\textwidth}
    \includegraphics[width=1\textwidth]{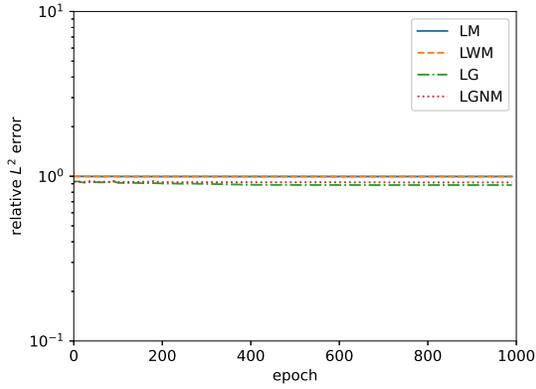}
    \end{minipage}
    }
    \subfigure[$N=3$]{
    \begin{minipage}[b]{0.46\textwidth}    
    \includegraphics[width=1\textwidth]{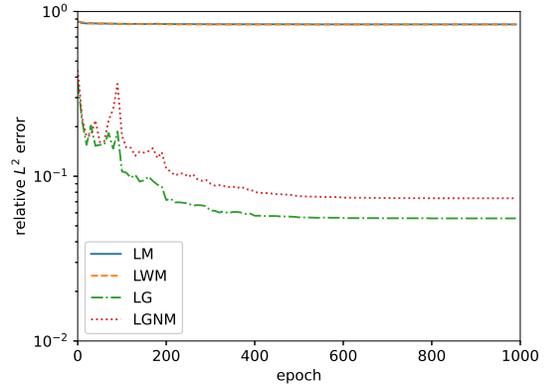}
    \end{minipage}
    }
    \bigskip	    
    \subfigure[$N=5$]{
    \begin{minipage}[b]{0.46\textwidth}
    \includegraphics[width=1\textwidth]{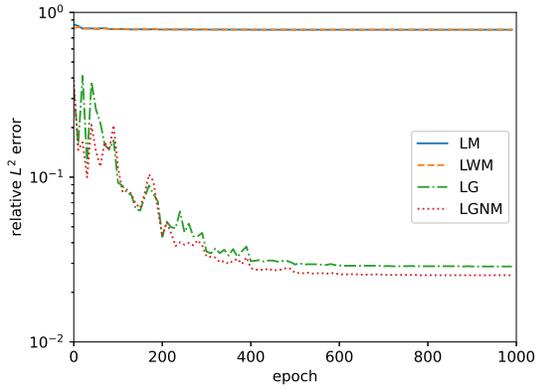}
    \end{minipage}
    }
    \subfigure[$N=7$]{
    \begin{minipage}[b]{0.46\textwidth}    
    \includegraphics[width=1\textwidth]{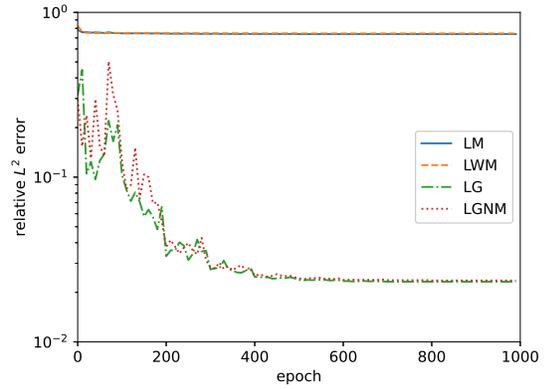}
    \end{minipage}
    }	    
    \caption{Relative $L^2$ error of the target function during the training process. Four approaches are compared here: learning the moment (LM) in \eqref{eq:learn-moment-antasz}; learning the weighted moment (LWM) in \eqref{eq:learn-moment-weighted-antasz}; learning the gradient (LG) in \eqref{eq:learn-gradient-antasz}; learning the gradient with normalized moments (LGNM) in \eqref{eq:learn-gradient-linear-antasz}.}
    \label{fig:train-loss}
\end{figure}

The depth and width of neural networks (i.e., the number of hidden layers and the number of nodes in the hidden layers) are also crucial hyperparameters in a neural network. Here, we test the number of layers to be $2,3,\cdots,7$ and number of nodes in hidden layers to be $8,16,\cdots,1024$. The results are shown in Figure \ref{fig:nlayer-training}. The error decreases when we increase the number of layers and nodes in hidden layers and saturate when they reach a certain level. These tests indicate that taking number of layers to be 6 and number of nodes to be 256 are good hyperparameters for our neural network.  As such these are the values used in this work unless otherwise stated.  
\begin{figure}
    \centering
    \includegraphics[width=0.5\textwidth]{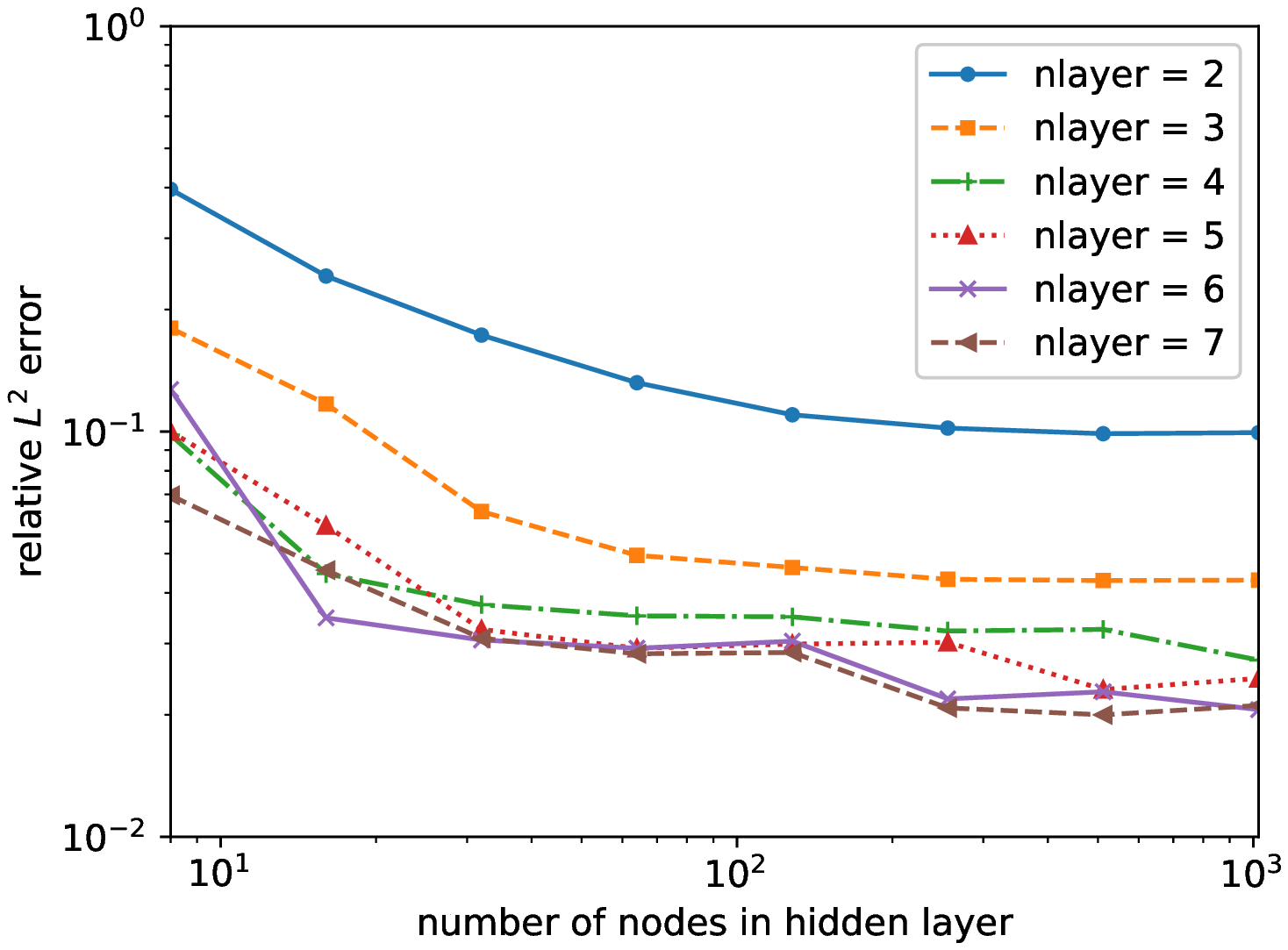}
    \caption{Relative $L^2$ error in the training data with different depths and widths of the neural networks. The number of layers: $2, 3, \dots,7$; the number of nodes in the hidden layers: $8, 16, \dots, 1024$. The number of moments in the system is $N=5$.}
    \label{fig:nlayer-training}
\end{figure}

\section{Numerical tests}\label{sec:numerical-test}

In this section, we show the performance of our ML closure model on a variety of benchmark tests, including problems with constant scattering and absorption coefficients, variable scattering  problems, Gaussian source problems and two-material problems. In all the numerical examples, we take the physical domain to be the unit interval $[0,1]$. We consider periodic boundary conditions and reflective boundary conditions.

To numerically solve the ML moment closure system, we apply the fifth-order finite difference WENO scheme \cite{jiang1996efficient} with a Lax–Friedrichs flux splitting for the spatial discretization. For the time discretization, we employ the third order strong-stability-preserving Runge-Kutta (RK) scheme \cite{shu1988efficient}. We take the grid number in space to be $N_x=256$. The CFL condition is taken to be $\Delta t=0.1\Delta x$. The penalty constant in the Lax-Friedrichs numerical flux is taken to be $\alpha_{\textrm{LF}}=5$, unless otherwise stated.

{We mainly focus on the comparison of the ML closure based on Learning the Gradient with Normalized Moments (LGNM) \eqref{eq:learn-gradient-linear-antasz}, ML closure based on  Learning the Moments (LM) \eqref{eq:learn-moment-antasz}, and  the $P_N$ closure \cite{chandrasekhar1944radiative}. In addition, we consider the $FP_N$ closure \cite{radice2013new} and make comparisons with the other three models in the context of  modeling transport on bounded domains with reflecting walls in  Example \ref{ex:gauss-source} and in the study of the two-material problem in Example \ref{ex:two-material}. We note that, in most cases, the ML model based on LGNM \eqref{eq:learn-gradient-linear-antasz} has better accuracy than the Learn Gradient (LG) model \eqref{eq:learn-gradient-antasz}. In particular, the benefit of using \eqref{eq:learn-gradient-linear-antasz} will be illustrated using the Gaussian source problem in Example \ref{ex:gauss-source}. 
}

An important observation is that our ML closure is able to nearly exactly reproduce the moments for the kinetic equation, even for the two-material problem, with a small number ($N=5$) of moments. 
In our numerical tests, we find that giving our ML closure  the freedom to relate  the gradient of the sixth-order moment to the gradient of the $0^{th}$ through $5^{th}$  moments is enough to produce accurate results for different regimes. Moreover, we numerically observe that if we use fewer than $N=5$ moments in our model, where the ML closure would relate the gradient of the highest moment to the gradient of the lower moments,  the method can still produce accurate results in the intermediate  regime,  but struggles to describe the solution near the optically thin regime.

\begin{exam}[constant scattering and absorption coefficients]\label{ex:const}
The setup of this example is the same as the data preparation. The scattering and absorption coefficients are taken to be constants over the domain. 

We test three different regimes: the optically thick regime ($\sigma_s=\sigma_t=100$); the intermediate regime ($\sigma_s=\sigma_t=10$); and the optically thin regime ($\sigma_s=\sigma_t=1$). All  closures work well in the optically thick regime and thus we omit the results and only focus on the intermediate and optically thin regimes. 

In Figure \ref{fig:profile-compare-N1}, we show the numerical solutions of $m_0$ and $m_1$ at $t=0.5$ with two moments ($N=1$) in the system and in the closure. We observe that, in the intermediate regime ($\sigma_s=\sigma_t=10$), the solution generated by the closure model with LGNM agrees well with the solution to the kinetic model. As a comparison, there exist some deviations for the other two closure models, the LM model and the $P_N$ model. However, in the optically thin regime ($\sigma_s=\sigma_t=1$), all  closures fail to capture the correct physical phenomenon for the kinetic model. This indicates that taking only two moments is not enough to close the kinetic equation in this regime.  This is also consistent with what we discovered in Section \ref{sec:free-streaming}, that is, as we move to the free streaming limit, the closure is related to the gradients of many moments.
\begin{figure}
    \centering
    \subfigure[$m_0$, $\sigma_s = \sigma_t = 10$]{
    \begin{minipage}[b]{0.46\textwidth}
    \includegraphics[width=1\textwidth]{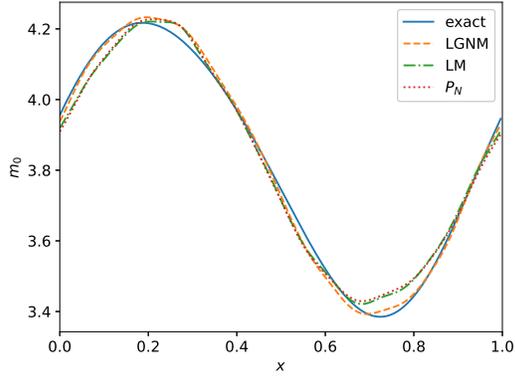}
    \end{minipage}
    }
    \subfigure[$m_1$, $\sigma_s = \sigma_t = 10$]{
    \begin{minipage}[b]{0.46\textwidth}    
    \includegraphics[width=1\textwidth]{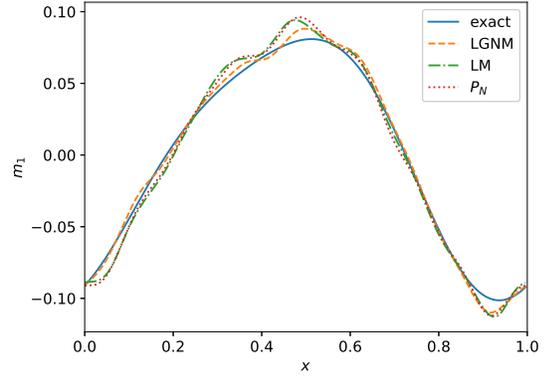}
    \end{minipage}
    }
    \bigskip
    \subfigure[$m_0$, $\sigma_s = \sigma_t = 1$]{
    \begin{minipage}[b]{0.46\textwidth}
    \includegraphics[width=1\textwidth]{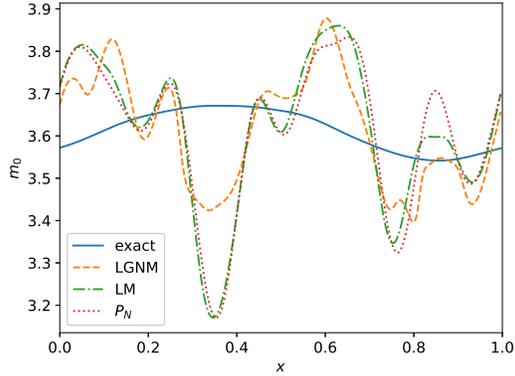}
    \end{minipage}
    }
    \subfigure[$m_1$, $\sigma_s = \sigma_t = 1$]{
    \begin{minipage}[b]{0.46\textwidth}    
    \includegraphics[width=1\textwidth]{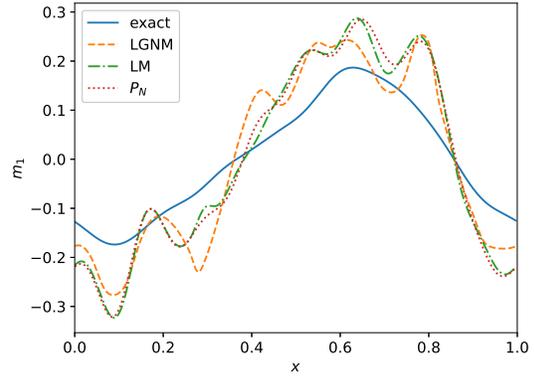}
    \end{minipage}
    }
    \caption{Example \ref{ex:const}: constant scattering and absorption coefficients. Here we are plotting the numerical solutions of $m_0$ and $m_1$ for three moment closures, including $P_N$, Learning Moment (LM) and Learning Gradient with Normalized Moments (LGNM), at $t=0.5$ with $N=1$ in the intermediate regime ($\sigma_s=\sigma_t=10$) and the optically thin regime ($\sigma_s=\sigma_t=1$). We note that the closure based on the LGNM performs slightly better than the other two methods for the $N=1$ case.}
    \label{fig:profile-compare-N1}
\end{figure}

In Figure \ref{fig:profile-compare-N5}, we show the numerical solutions of $m_0$ and $m_1$ at $t=0.5$ with $N=5$. In the intermediate regime, all the closures predict the solution quite well. Moreover, it is observed that the ML closure model based on the LGNM formulation has the smallest error (see the zoomed-in figure). In the optically thin regime, only the closure model based on the LGNM formulation agrees well with the kinetic equation, while the other two closures have large deviations in the moments.
\begin{figure}
    \centering
    \subfigure[$m_0$, $\sigma_s = \sigma_t = 10$]{
    \begin{minipage}[b]{0.46\textwidth}
    \includegraphics[width=1\textwidth]{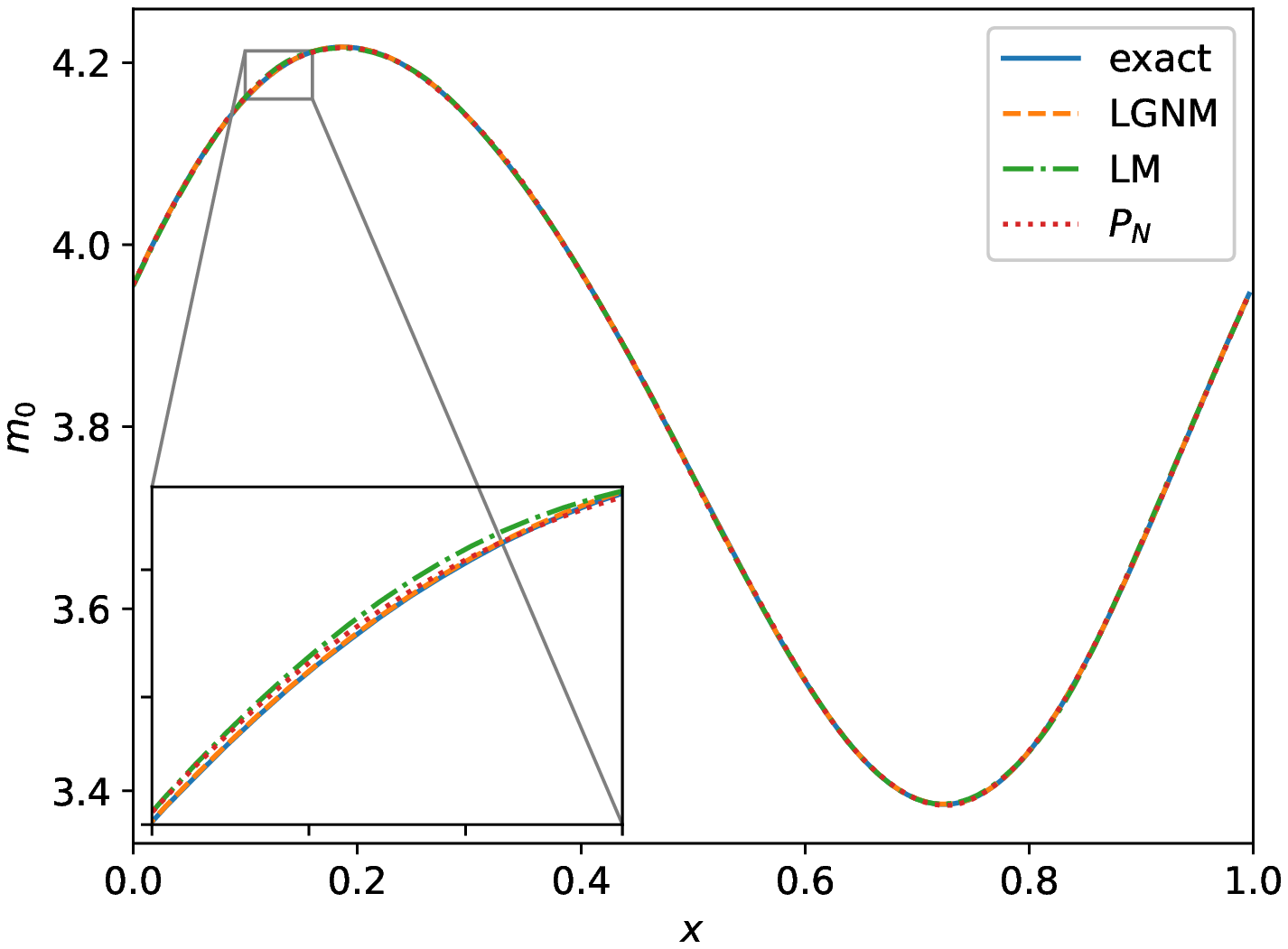}
    \end{minipage}
    }
    \subfigure[$m_1$, $\sigma_s = \sigma_t = 10$]{
    \begin{minipage}[b]{0.46\textwidth}    
    \includegraphics[width=1\textwidth]{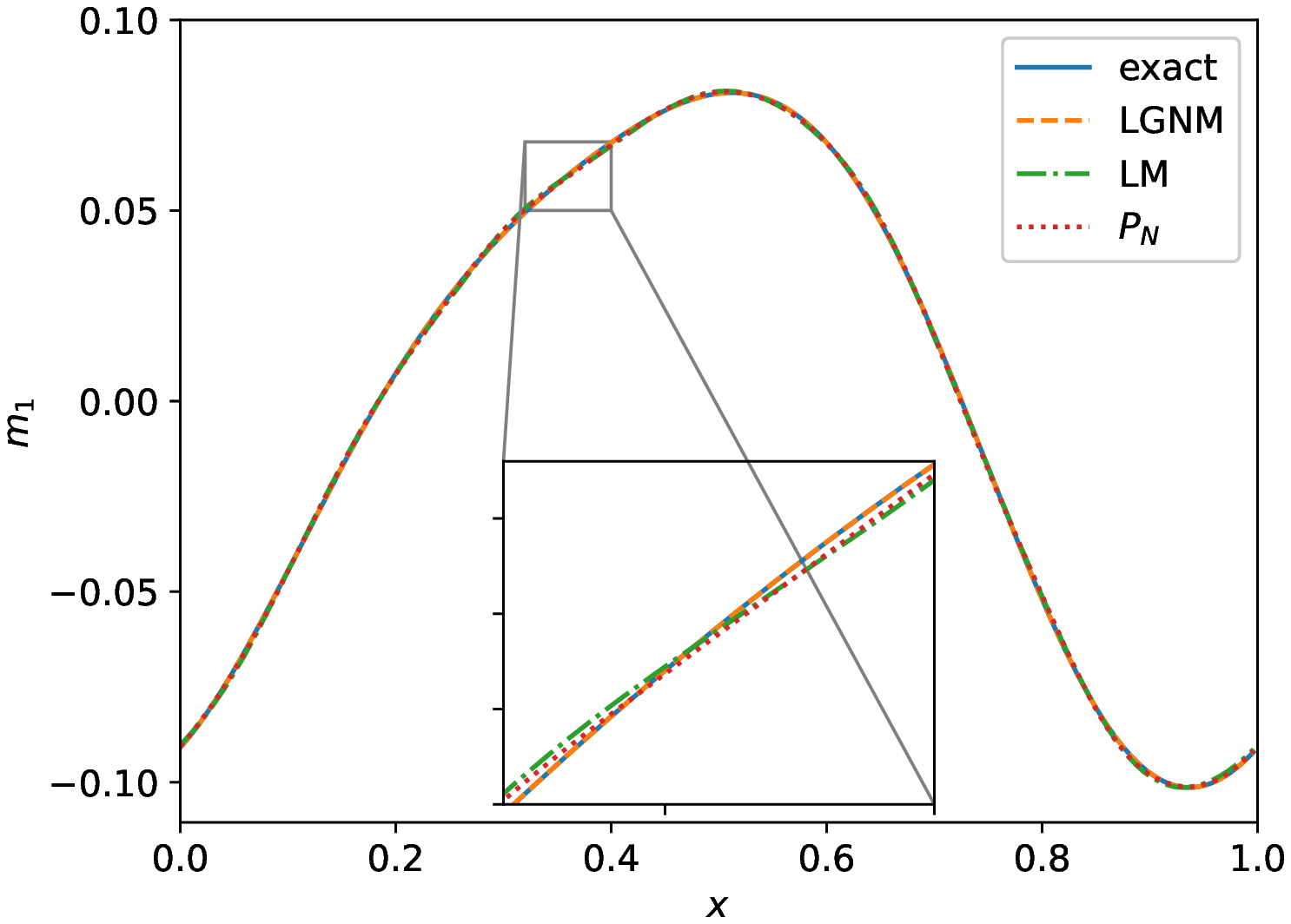}
    \end{minipage}
    }
    \bigskip
    \subfigure[$m_0$, $\sigma_s = \sigma_t = 1$]{
    \begin{minipage}[b]{0.46\textwidth}
    \includegraphics[width=1\textwidth]{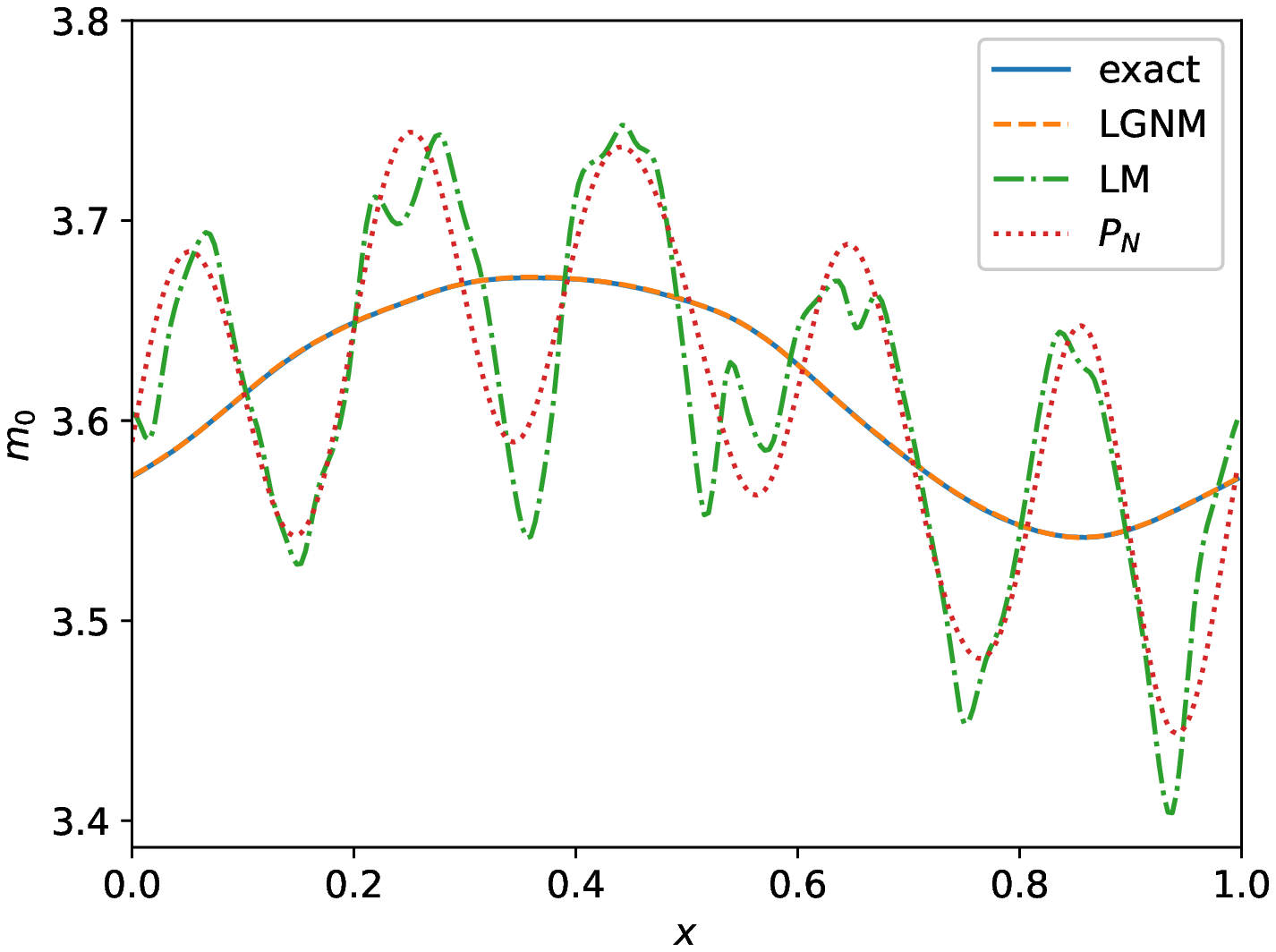}
    \end{minipage}
    }
    \subfigure[$m_1$, $\sigma_s = \sigma_t = 1$]{
    \begin{minipage}[b]{0.46\textwidth}    
    \includegraphics[width=1\textwidth]{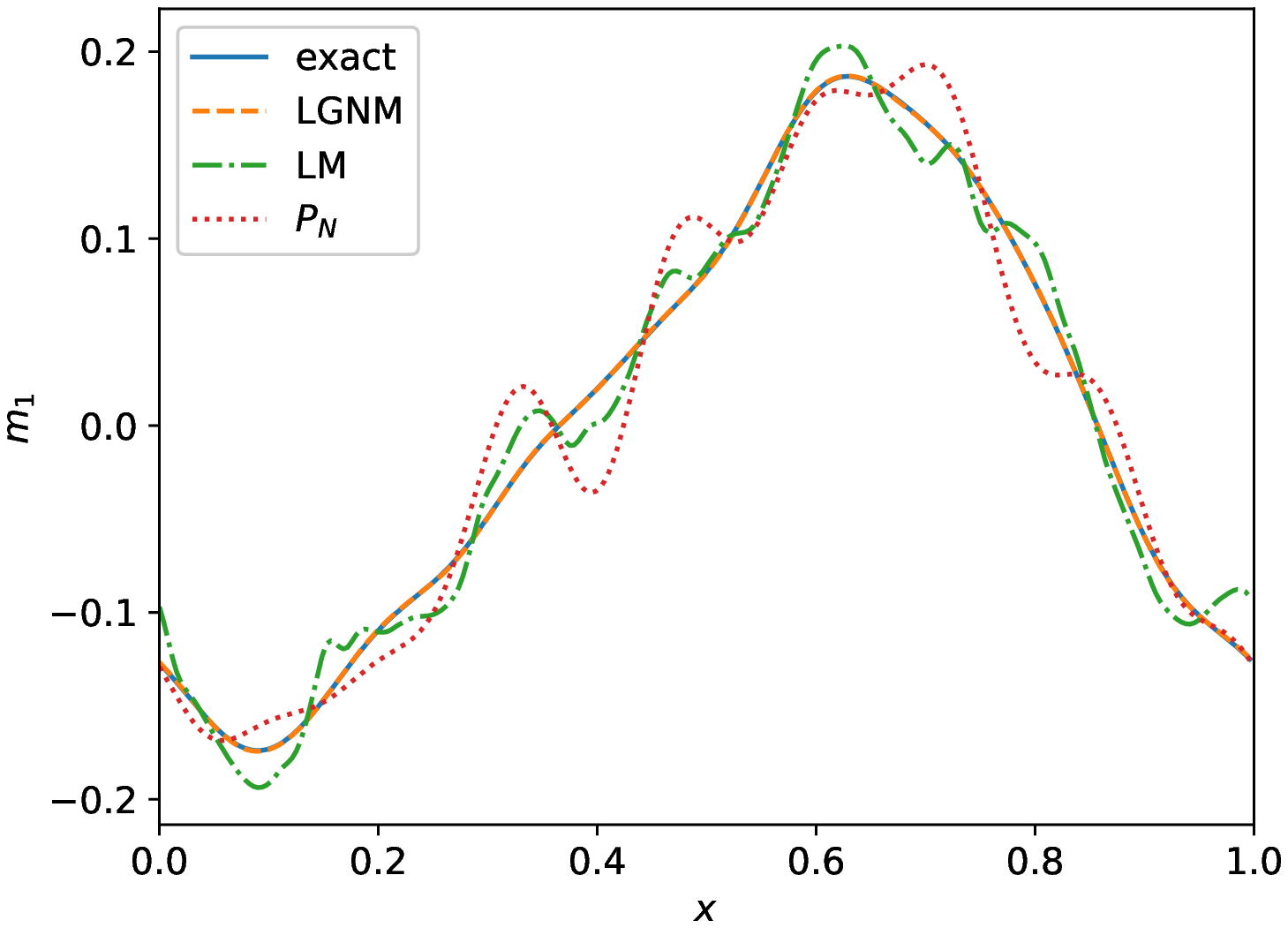}
    \end{minipage}
    }
    \caption{Example \ref{ex:const}: constant scattering and absorption coefficients.
    Here we are plotting the numerical solutions of $m_0$ and $m_1$ for three moment closures, including $P_N$, Learning Moment (LM) and Learning Gradient with Normalized Moments (LGNM), at $t=0.5$  with $N=5$ in the intermediate regime ($\sigma_s=\sigma_t=10$) and the optically thin regime ($\sigma_s=\sigma_t=1$). We note that the closure based on the LGNM performs better than the other two methods for the $N=5$ case.}
    \label{fig:profile-compare-N5}
\end{figure}

In Figure \ref{fig:error-compare-N1} and Figure \ref{fig:error-compare-N5}, we display the log–log scatter plots of the relative $L^2$ error versus the scattering coefficient for $N=1$ and $N=5$, respectively. In the case of $N=1$, we observe that building a closure model based on the LG and LGNM formulations does not always result in smaller errors than the $P_N$ closure, see Figure \ref{fig:error-compare-N1}. This indicates that it seems impossible to find a model with only two moments that will be able to  approximate the  moments  accurately in the optically thin case. In the case of $N=5$, the LG and LGNM models have a much smaller error than the LM model and the $P_N$ model, especially for smaller scattering coefficients, see  Figure \ref{fig:error-compare-N5}. Moreover, we notice that, in most cases, the model based on LGNM has better accuracy than the LG model. 
\begin{figure}
    \centering
    \subfigure[$m_0$]{
    \begin{minipage}[b]{0.46\textwidth}
    \includegraphics[width=1\textwidth]{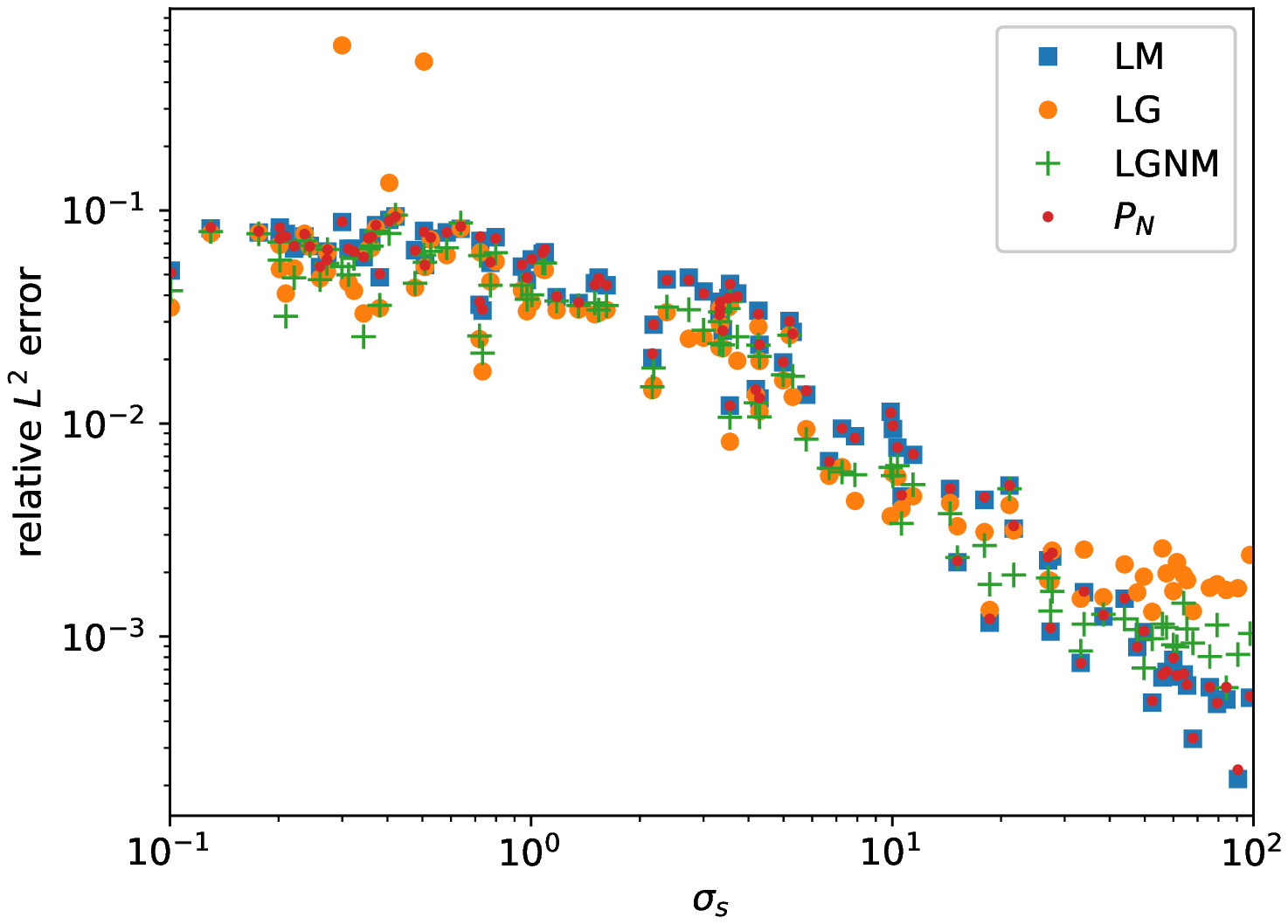}
    \end{minipage}
    }
    \subfigure[$m_1$]{
    \begin{minipage}[b]{0.46\textwidth}    
    \includegraphics[width=1\textwidth]{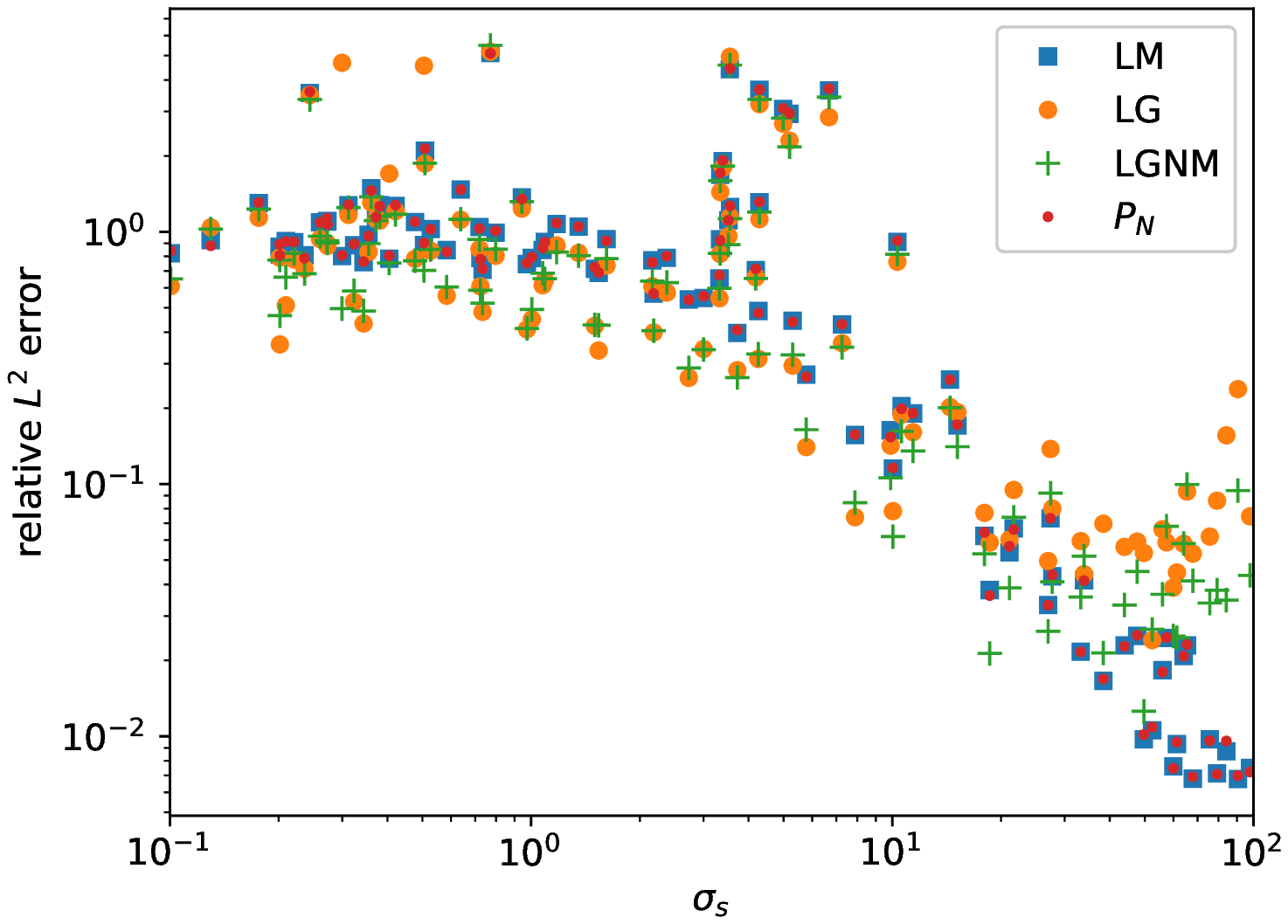}
    \end{minipage}
    }
    \caption{Example \ref{ex:const}: constant scattering and absorption coefficients. In this figure we are plotting the relative $L^2$ error of $m_0$ and $m_1$ with different scattering coefficient at $t=0.5$ with $N=1$.  All methods perform about the same for $N=1$.}
    \label{fig:error-compare-N1}
\end{figure}

\begin{figure}
    \centering
    \subfigure[$m_0$]{
    \begin{minipage}[b]{0.46\textwidth}
    \includegraphics[width=1\textwidth]{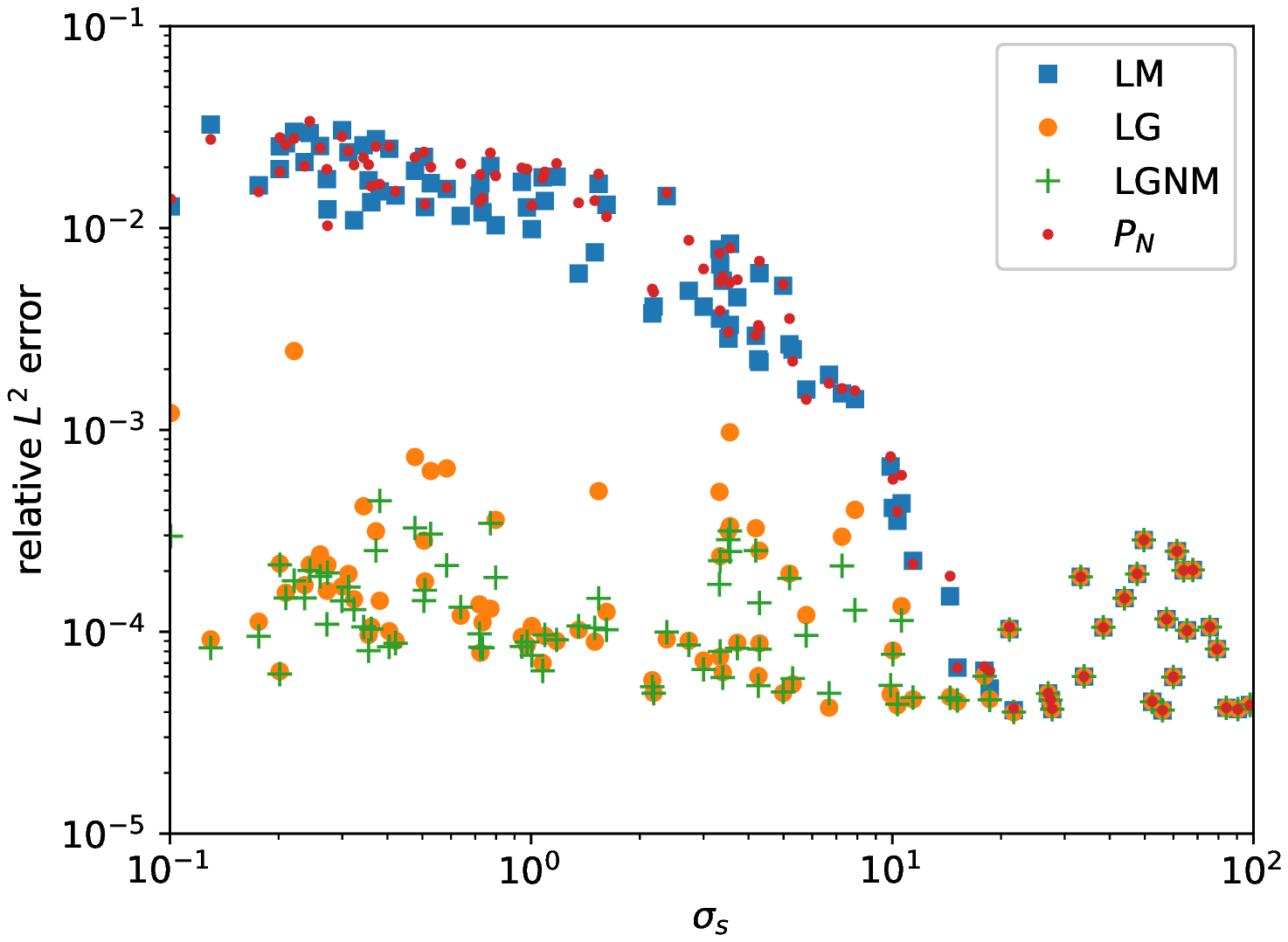}
    \end{minipage}
    }
    \subfigure[$m_1$]{
    \begin{minipage}[b]{0.46\textwidth}    
    \includegraphics[width=1\textwidth]{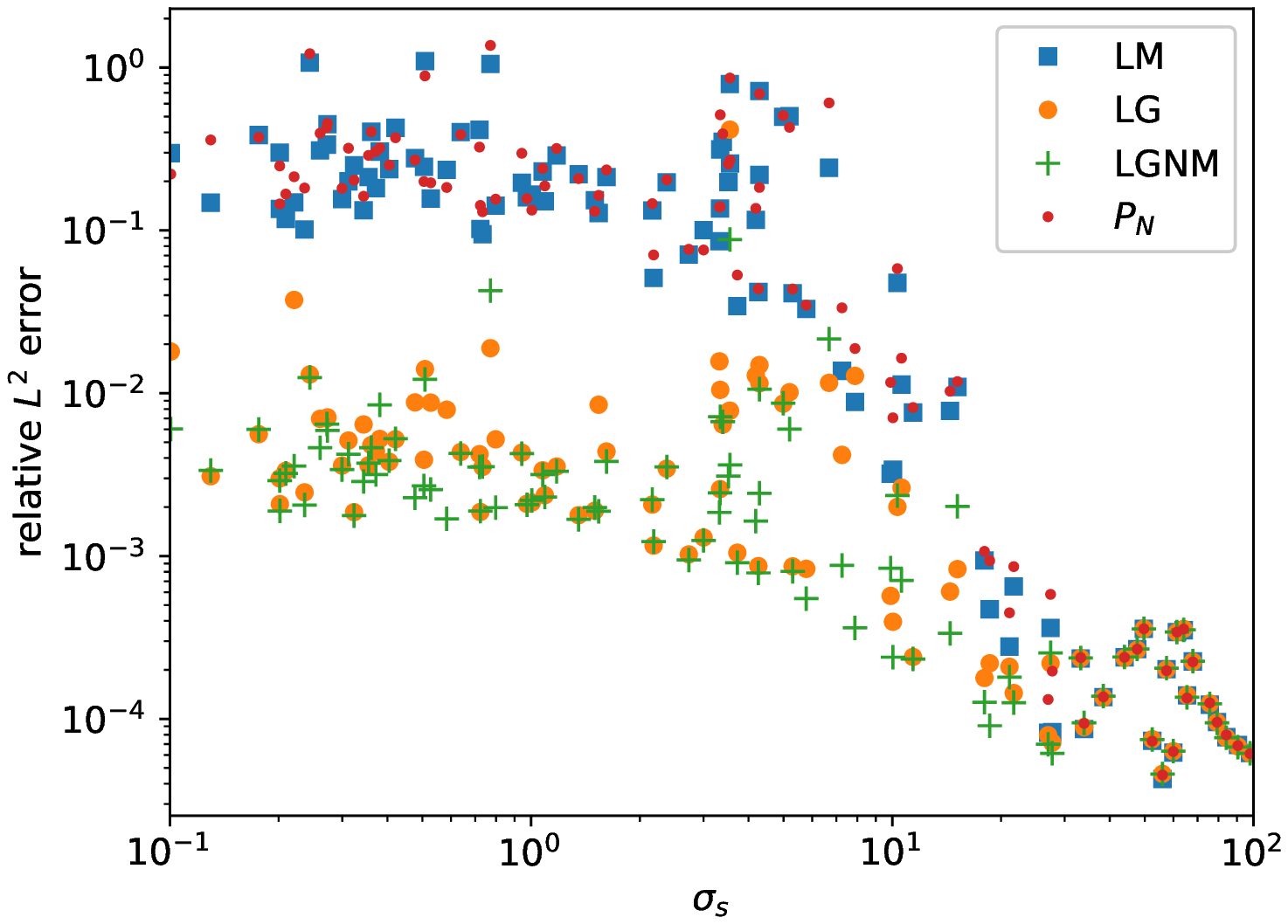}
    \end{minipage}
    }
    \caption{Example \ref{ex:const}: constant scattering and absorption coefficients. The figure plots the relative $L^2$ error of $m_0$ and $m_1$ with different scattering coefficient at $t=0.5$ with $N=5$. Here Learning Gradient (LG) and Learn Gradient with Normalized Moments (LGNM) methods perform better than Learning Moment (LM) and the $P_N$ closure.}
    \label{fig:error-compare-N5}
\end{figure}

Similar to many analytical closure models, our ML closure models do not guarantee  hyperbolicity. We now use this example to investigate the numerical stability of our closure model. We  take the penalty constant in the Lax-Friedrichs flux to be $\alpha_{\textrm{LF}}=2$. In Figure \ref{fig:eigenvalue-lax2}, we output the number of grid points with imaginary eigenvalues and $L^{\infty}$-norm of the numerical solution during the time evolution. We find that this model does not preserve the hyperbolicity property and the numerical solution starts to blow up at around $t=0.27$. We then increase the penalty constant to $\alpha_{\textrm{LF}}=5$. As shown in Figure \ref{fig:eigenvalue-lax5}, this helps stabilize the model numerically. Even though there are still a large number of grid points with imaginary eigenvalues during the time evolution, the numerical solution does not blow up.
We conclude that  incorporating hyperbolicity in the ML closure model is an important topic, which is discussed in a subsequent work \cite{huang2021ml2,huang2021ml3}.
\begin{figure}
    \centering
    \subfigure[number of grid points with imaginary eigenvalues]{
    \begin{minipage}[b]{0.46\textwidth}
    \includegraphics[width=1\textwidth]{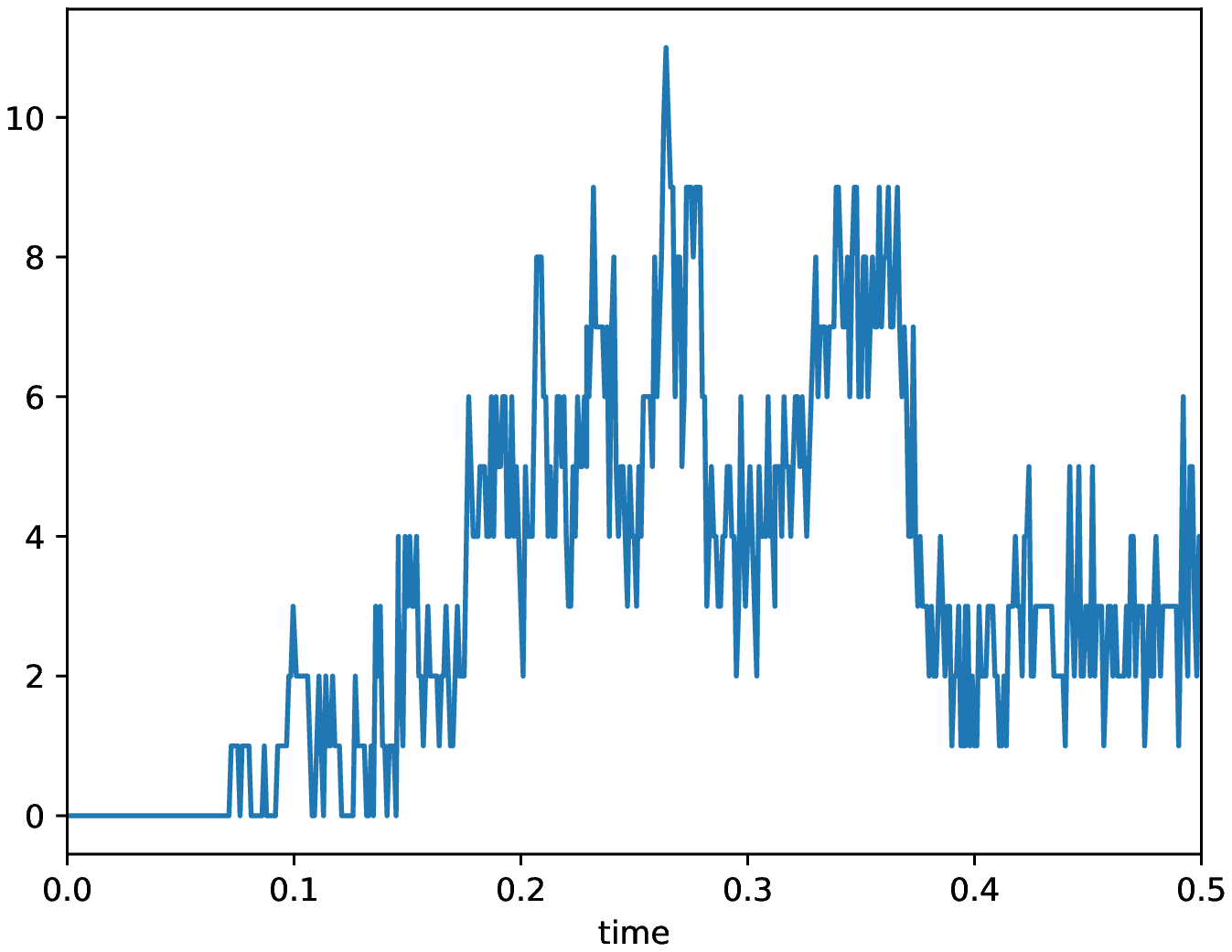}
    \end{minipage}
    }
    \subfigure[$L^{\infty}$ norm of numerical solutions]{
    \begin{minipage}[b]{0.46\textwidth}    
    \includegraphics[width=1\textwidth]{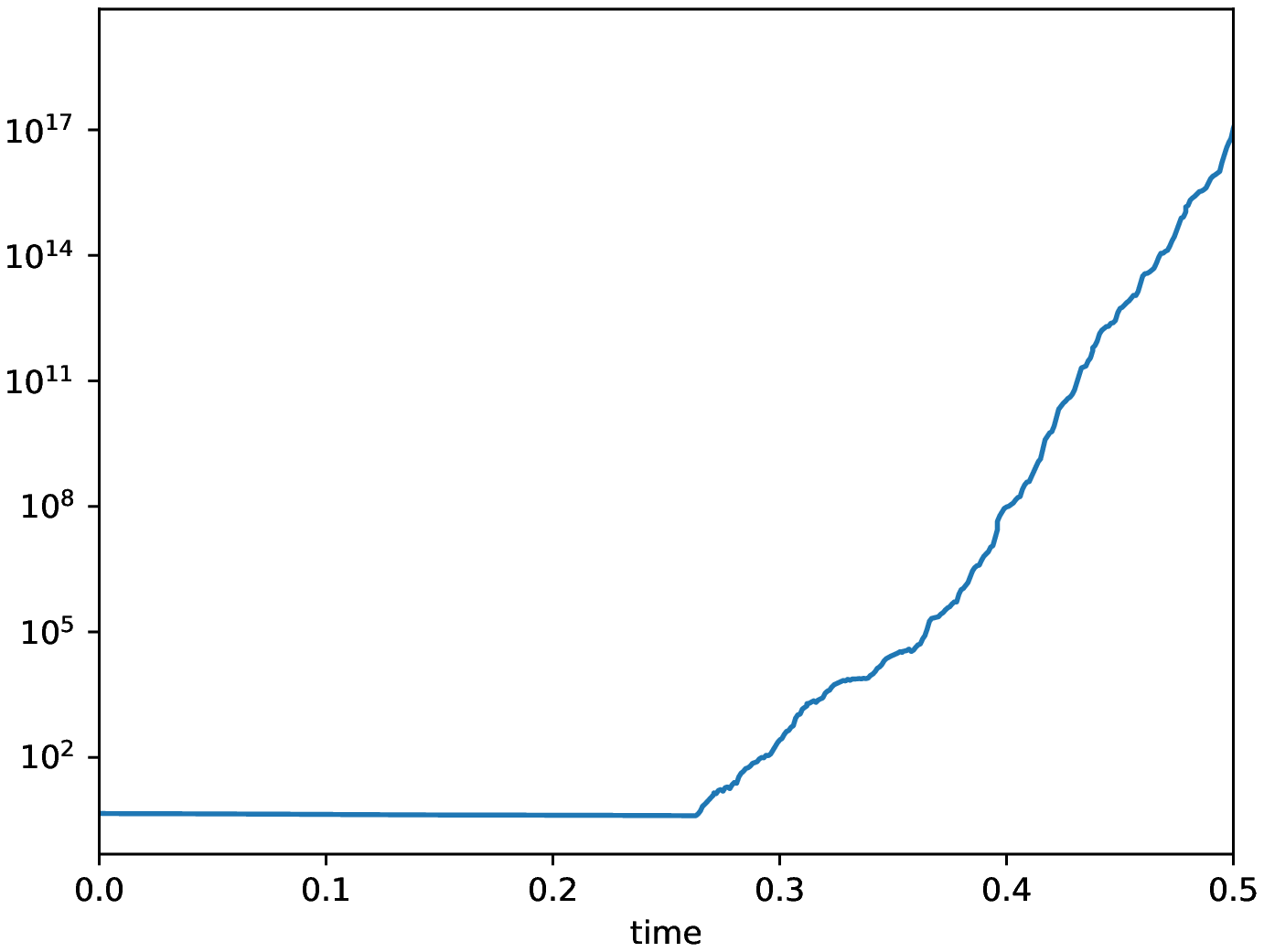}
    \end{minipage}
    }
    \caption{Example \ref{ex:const}: constant scattering and absorption coefficients. The number of grid points with imaginary eigenvalues and $L^{\infty}$ norm of numerical solutions during the time evolution in the optically thin regime ($\sigma_s = \sigma_t = 1$) with $N=5$ and $\alpha_{\textrm{LF}}=2$.}
    \label{fig:eigenvalue-lax2}
\end{figure}

\begin{figure}
    \centering
    \subfigure[number of grid points with imaginary eigenvalues]{
    \begin{minipage}[b]{0.46\textwidth}
    \includegraphics[width=1\textwidth]{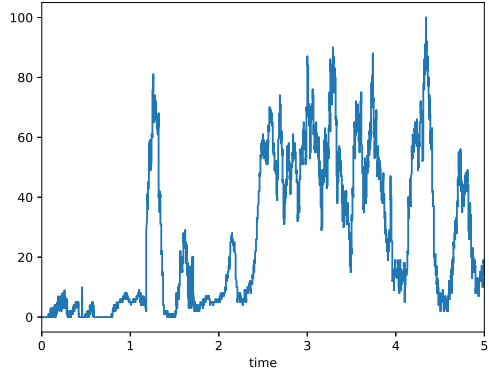}
    \end{minipage}
    }
    \subfigure[$L^{\infty}$ norm of numerical solutions]{
    \begin{minipage}[b]{0.46\textwidth}    
    \includegraphics[width=1\textwidth]{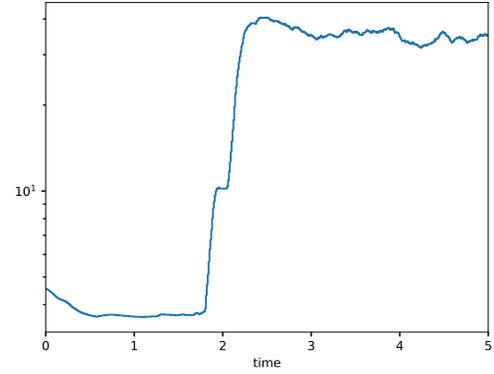}
    \end{minipage}
    }
    \caption{Example \ref{ex:const}: constant scattering and absorption coefficients. The number of grid points with imaginary eigenvalues and $L^{\infty}$ norm of numerical solutions during the time evolution in the optically thin regime ($\sigma_s = \sigma_t = 1$) with $N=5$ and $\alpha_{\textrm{LF}}=5$.}
    \label{fig:eigenvalue-lax5}
\end{figure}

\end{exam}

\begin{exam}[variable scattering  problem]\label{ex:scatter-vary}
	In this example, we investigate the performance and generalizability  of our closure models by testing them on problems that have spatially varying scattering coefficients. In these tests, the scattering coefficient is taken to have the following   form
	\begin{equation}\label{eq:vary-scatter-function}
		\sigma_s(x) = c_1 (\tanh(1 + c_2 (x - x_0)) + \tanh(1 - c_2 (x - x_0))) + \sigma_{s,\textrm{base}},
	\end{equation}
	with $c_1$, $c_2$, $\sigma_{s,\textrm{base}}$ and $x_0$ being constants. Here, the parameters are taken to be $c_1=c_2=15$, $x_0=0.5$ and $\sigma_a=1$. We test two cases with $\sigma_{s,\textrm{base}}=1$ and $\sigma_{s,\textrm{base}}=10$. The profiles of scattering coefficients with $\sigma_{s,\textrm{base}}=1$ and $\sigma_{s,\textrm{base}}=10$ are shown in Figure \ref{fig:scattering-profile}. It is observed that, in the case of $\sigma_{s,\textrm{base}}=1$, the middle part of the domain is in the intermediate regime and the domain near the boundary is in the optically thin regime. On the other hand, the whole domain is in the intermediate regime when $\sigma_{s,\textrm{base}}=10$. 
	\begin{figure}
	    \centering
	    \subfigure[$\sigma_{s,\textrm{base}}=1$]{
	    \begin{minipage}[b]{0.46\textwidth}
	    \includegraphics[width=1\textwidth]{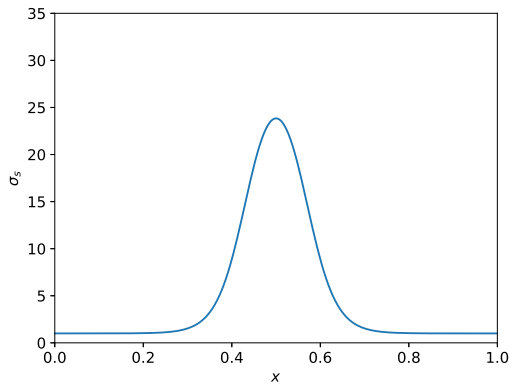}
	    \end{minipage}
	    }
	    \subfigure[$\sigma_{s,\textrm{base}}=10$]{
	    \begin{minipage}[b]{0.46\textwidth}    
	    \includegraphics[width=1\textwidth]{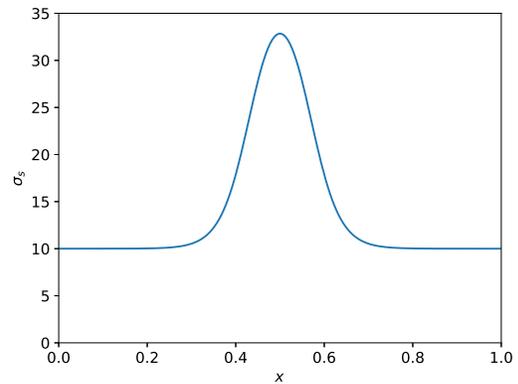}
	    \end{minipage}
	    }
	    \caption{Example \ref{ex:scatter-vary}: variable scattering problem. Plots of the profiles of scattering coefficient functions given in equation \eqref{eq:vary-scatter-function}, with $\sigma_{s,\textrm{base}}=1$ and $\sigma_{s,\textrm{base}}=10$ respectively.}
	    \label{fig:scattering-profile}
	\end{figure}

	The profiles of numerical solutions with $\sigma_{s,\textrm{base}}=1$ and $\sigma_{s,\textrm{base}}=10$ are presented in Figure \ref{fig:scatter-vary-compare-sigmas1} and Figure \ref{fig:scatter-vary-compare-sigmas10}, respectively. In the test case that spans the optically thin to intermediate regime, we observe good agreement between  our LGNM closure model and the kinetic model, see Figure \ref{fig:scatter-vary-compare-sigmas1}, while the other two closure models do not obtain satisfactory solution profiles. In Figure \ref{fig:scatter-vary-compare-sigmas10}, we plot the results when $\sigma_{s,\textrm{base}}=10$.   In this  case, all closure models work well. We note that our LGNM model still has the smallest error.
	\begin{figure}
	    \centering
	    \subfigure[$m_0$]{
	    \begin{minipage}[b]{0.46\textwidth}
	    \includegraphics[width=1\textwidth]{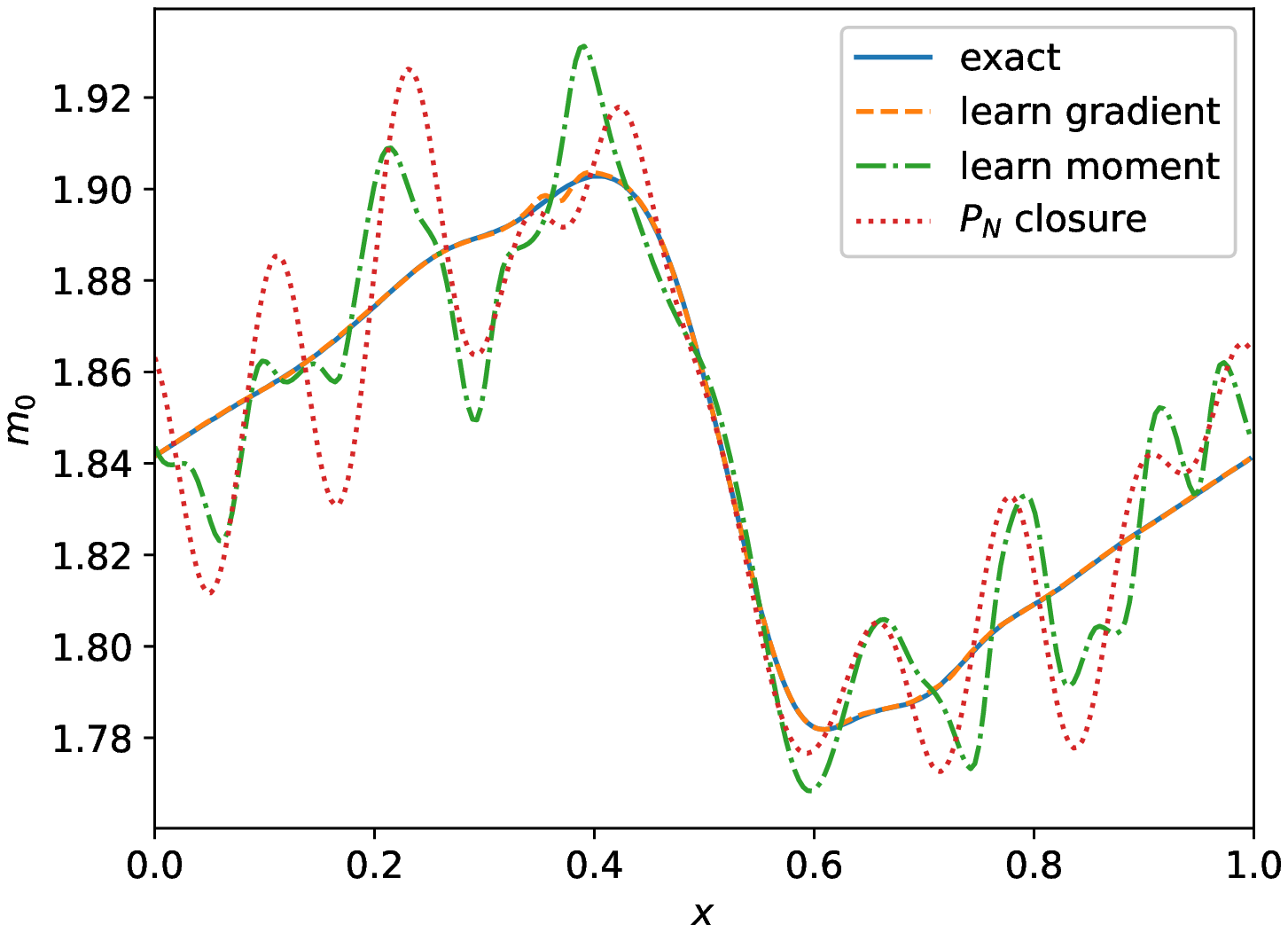}
	    \end{minipage}
	    }
	    \subfigure[$m_1$]{
	    \begin{minipage}[b]{0.46\textwidth}    
	    \includegraphics[width=1\textwidth]{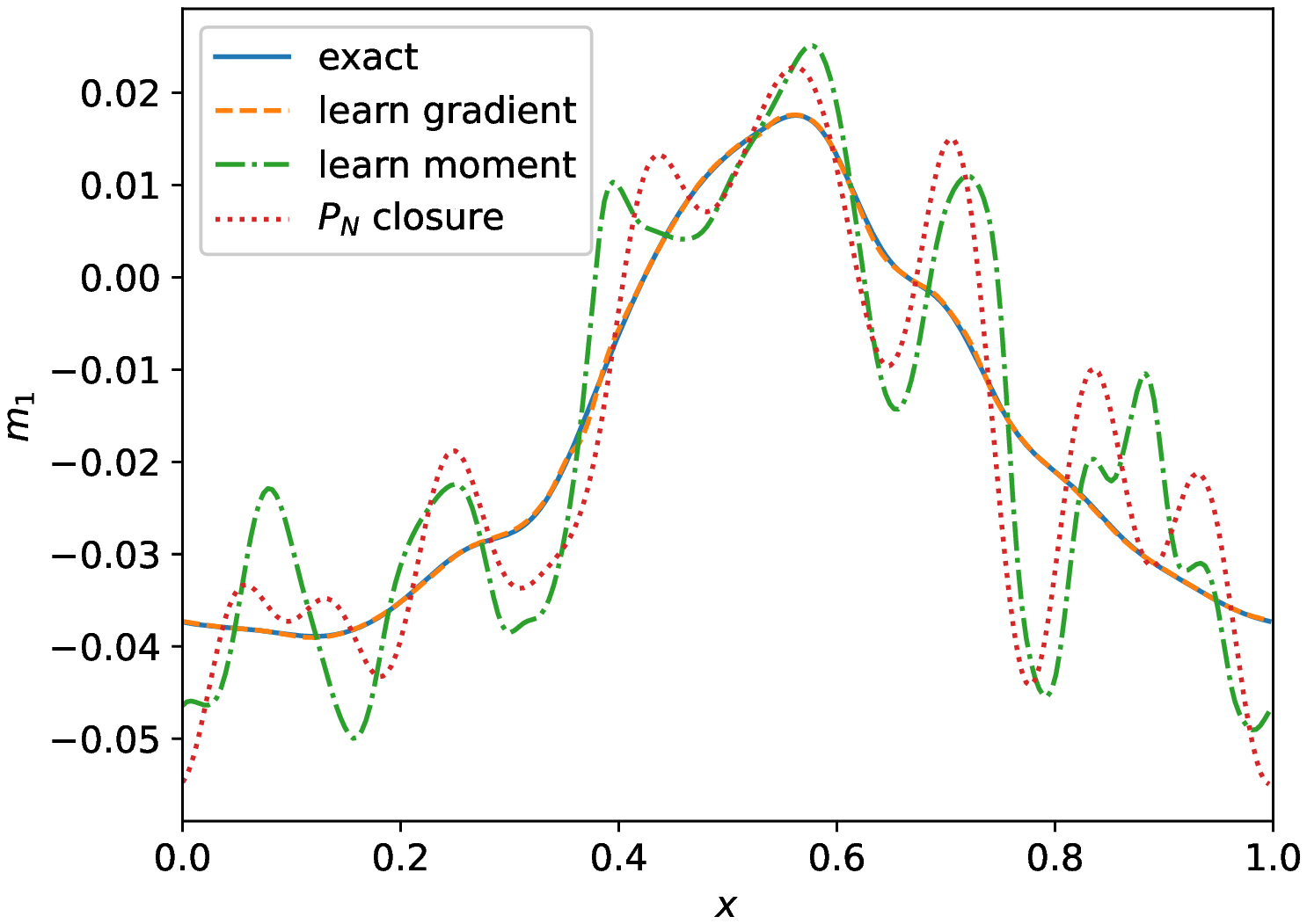}
	    \end{minipage}
	    }
	    \caption{Example \ref{ex:scatter-vary}: variable scattering problem. Numerical solutions of $m_0$ and $m_1$ at $t=0.5$ with $N=5$ and $\sigma_{s,\textrm{base}}=1$.}
	    \label{fig:scatter-vary-compare-sigmas1}
	\end{figure}
	\begin{figure}
	    \centering
	    \subfigure[$m_0$]{
	    \begin{minipage}[b]{0.46\textwidth}
	    \includegraphics[width=1\textwidth]{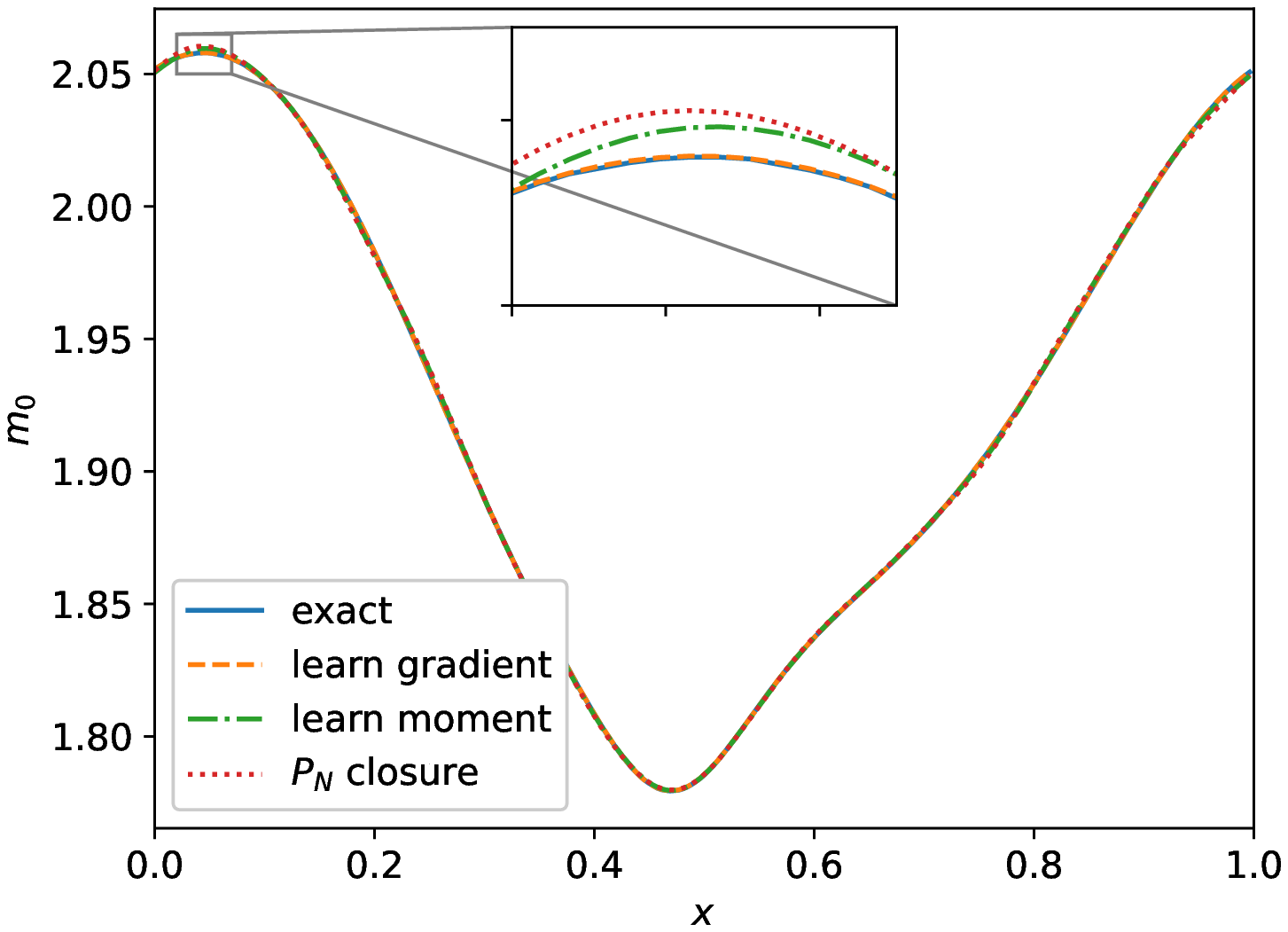}
	    \end{minipage}
	    }
	    \subfigure[$m_1$]{
	    \begin{minipage}[b]{0.46\textwidth}    
	    \includegraphics[width=1\textwidth]{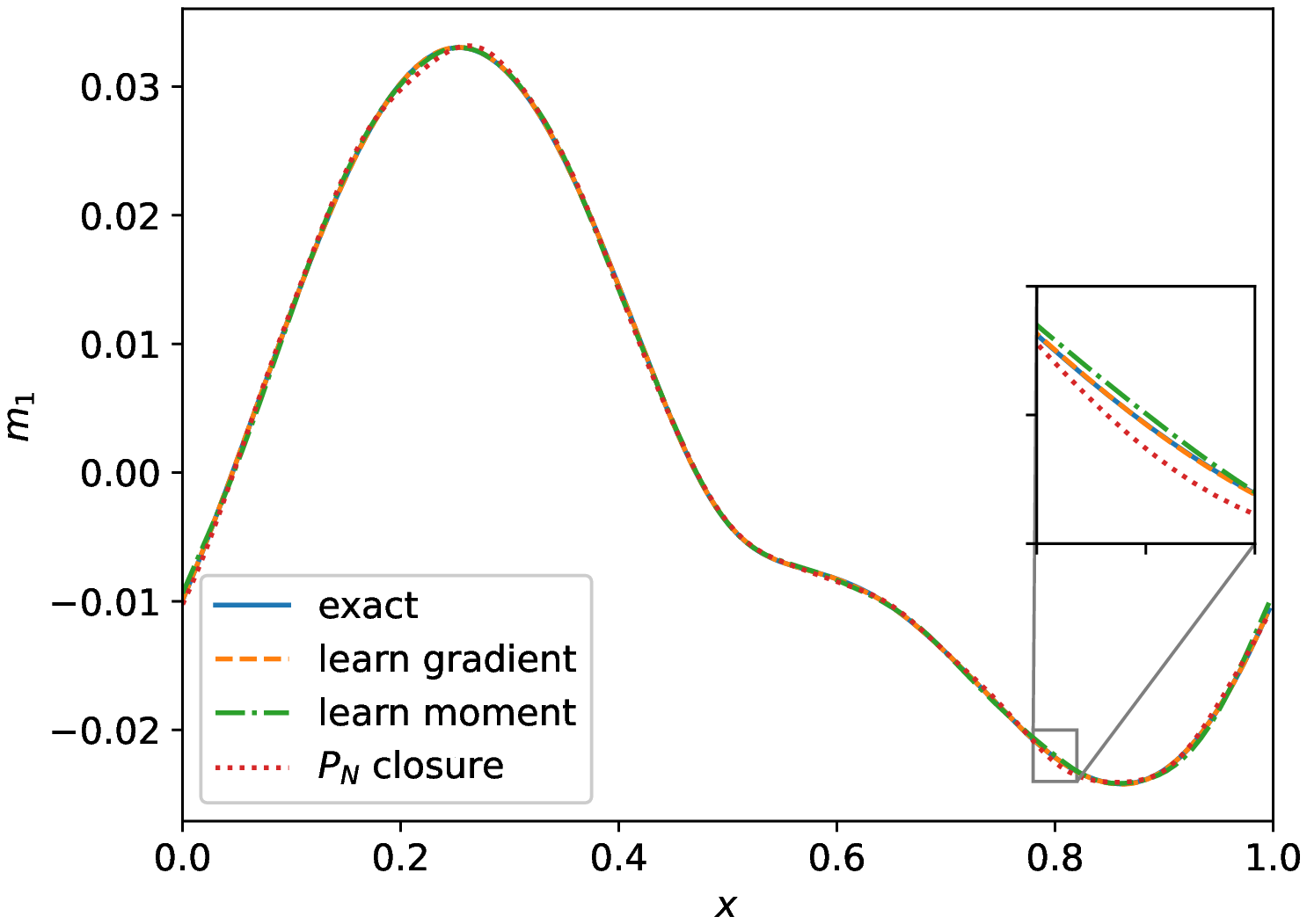}
	    \end{minipage}
	    }
	    \caption{Example \ref{ex:scatter-vary}: variable scattering  problem.
	    Numerical solutions of $m_0$ and $m_1$ at $t=0.5$ with $N=5$ and $\sigma_{s,\textrm{base}}=10$.}
	    \label{fig:scatter-vary-compare-sigmas10}
	\end{figure}	
\end{exam}

\begin{exam}[Gaussian source problem]\label{ex:gauss-source}
	In this example, we investigate the Gaussian source problem, which simulates particles with an initial   intensity that is a Gaussian distribution in space \cite{frank2012perturbed,fan2020nonlinear}
	\begin{equation}\label{eq:gauss-source-init}
		f_0(x,v) = \frac{c_1}{(2 \pi \theta)^{1/2}} \exp\brac{-\frac{(x - x_0)^2}{2 \theta}} + c_2.
	\end{equation}
	In this test, we take $c_1=0.5$, $c_2=2.5$, $x_0=0.5$ and $\theta=0.01$. 

	In Figure \ref{fig:gauss-source-compare}, we present the results obtained using various closure models. We observe good agreement between the LGNM closure model and the kinetic model, while the other two closure models have large deviations from the kinetic model. This illustrates that the ML closure model exhibits good generalization to other types of initial conditions beyond the training data.
	\begin{figure}
	    \centering
	    \subfigure[$m_0$]{
	    \begin{minipage}[b]{0.46\textwidth}
	    \includegraphics[width=1\textwidth]{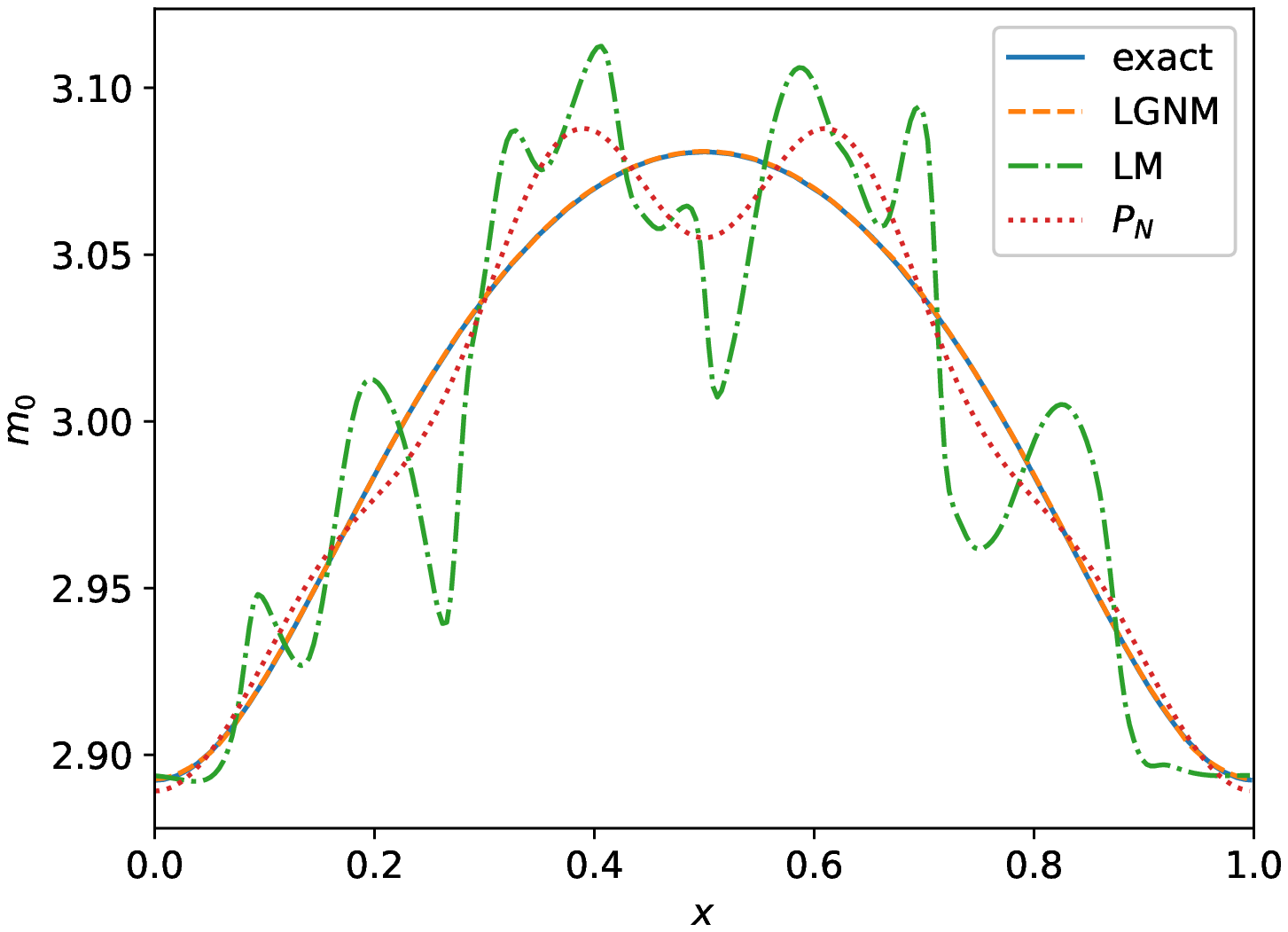}
	    \end{minipage}
	    }
	    \subfigure[$m_1$]{
	    \begin{minipage}[b]{0.46\textwidth}    
	    \includegraphics[width=1\textwidth]{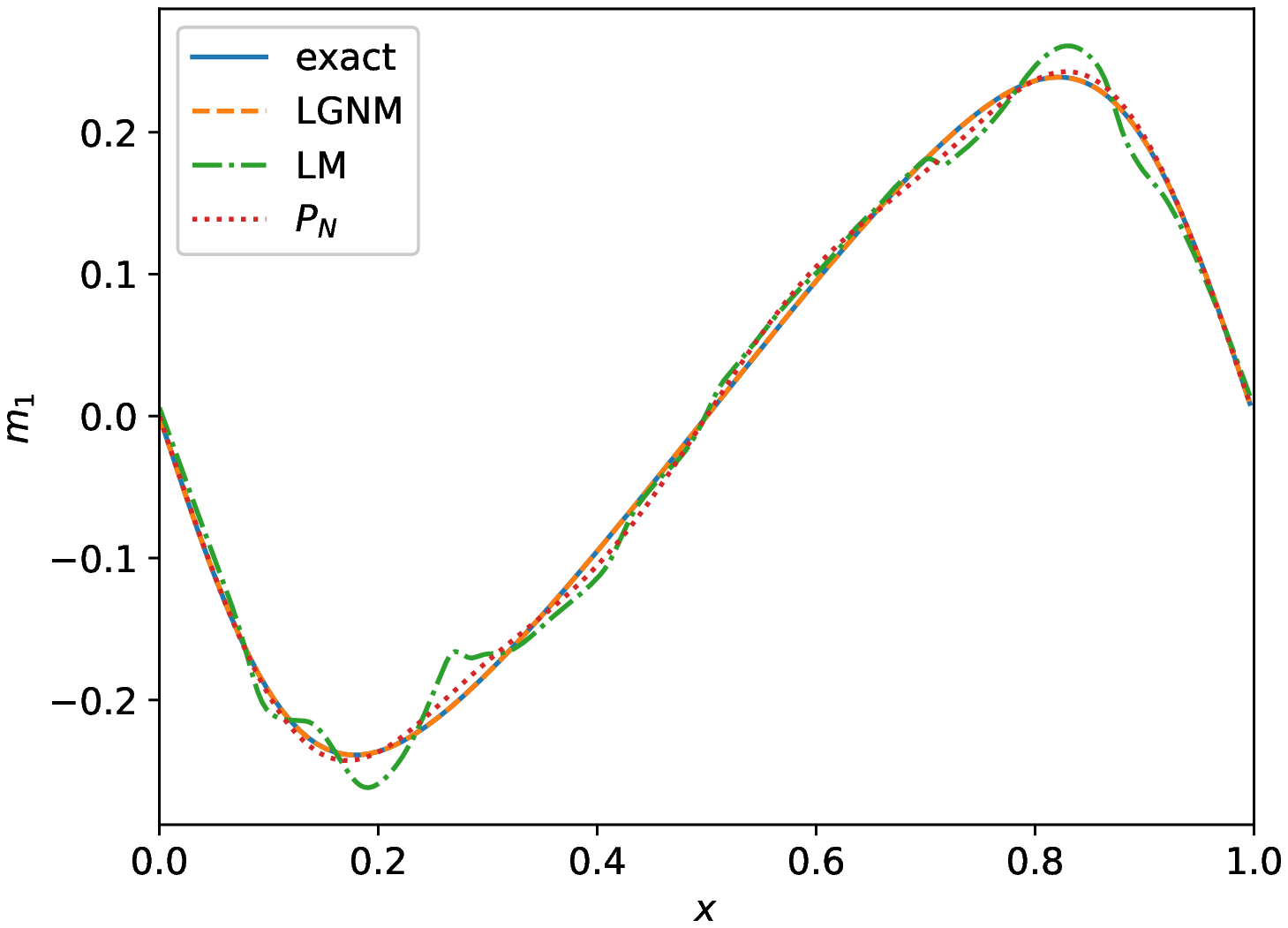}
	    \end{minipage}
	    }
	    \caption{Example \ref{ex:gauss-source}: Gaussian source problem. Numerical solutions of $m_0$ and $m_1$ at $t=0.5$ with $N=5$ in the optically thin regime ($\sigma_s=1$).}
	    \label{fig:gauss-source-compare}
	\end{figure}

	From \eqref{eq:init-fourier-series}, it is easy to see that the upper bound of the initial conditions in our training dataset is $1 + 2\sum_{k=1}^{k_{\max}}\frac{1}{k} \approx 6.86$.
	To investigate the generalization  of our model to problems that have a magnitude that is outside of the  range of training data, we amplify the initial condition \eqref{eq:gauss-source-init} by 1000 times.  That is, in this test we take the initial condition  to be 
	\begin{equation}\label{eq:gauss-source-init-amplify}
		f_0(x,v) = 1000\brac{\frac{c_1}{(2 \pi \theta)^{1/2}} \exp\brac{-\frac{(x - x_0)^2}{2 \theta}} + c_2}
	\end{equation}
	with the same parameters $c_1$, $c_2$, $x_0$ and $\theta$ as in \eqref{eq:gauss-source-init}.
	For the original initial condition \eqref{eq:gauss-source-init}, the ML closure model based on the LG and LGNM both predict the solution well, see Figure \ref{fig:gauss-source-scale} (a). The relative $L^2$ errors are $1.05\times10^{-4}$ and $6.42\times10^{-5}$ for LG and LGNM approaches respectively. When we amplify the initial condition by 1000 times, the two models still generate good predictions, see Figure \ref{fig:gauss-source-scale} (b).  The relative $L^2$ errors are $3.89\times10^{-4}$ and $6.38\times10^{-5}$ for LG and LGNM approaches respectively. This implies that the model without the scale invariance constraint automatically learns this invariance property from data. Nevertheless, enforcing the scale invariance property will lead to better results when the magnitude of the problem initial conditions is very different from that used in the training dataset.
	\begin{figure}
	    \centering
	    \subfigure[original initial condition]{
	    \begin{minipage}[b]{0.46\textwidth}
	    \includegraphics[width=1\textwidth]{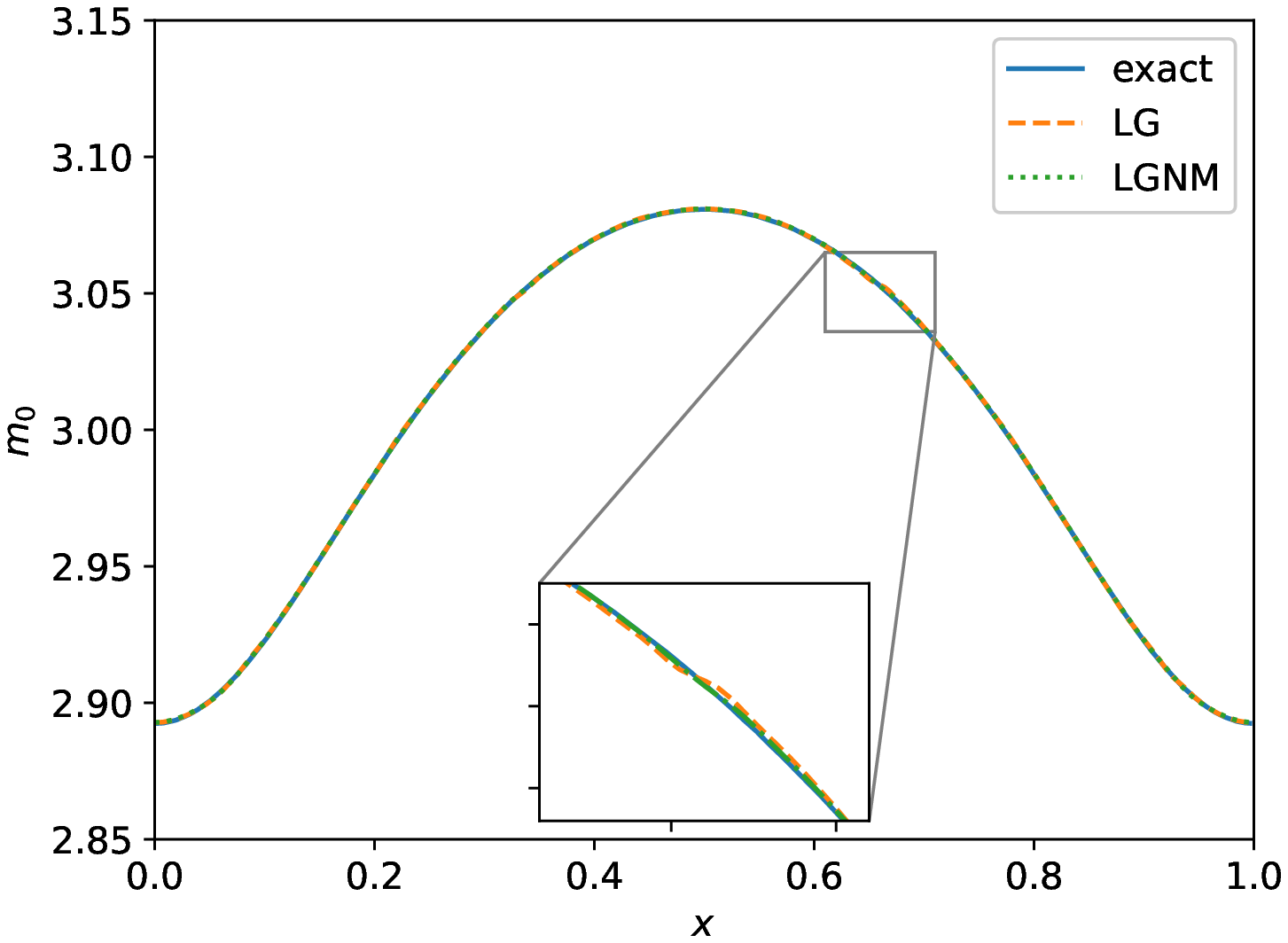}
	    \end{minipage}
	    }
	    \subfigure[original initial condition $\times$ 1000]{
	    \begin{minipage}[b]{0.46\textwidth}    
	    \includegraphics[width=1\textwidth]{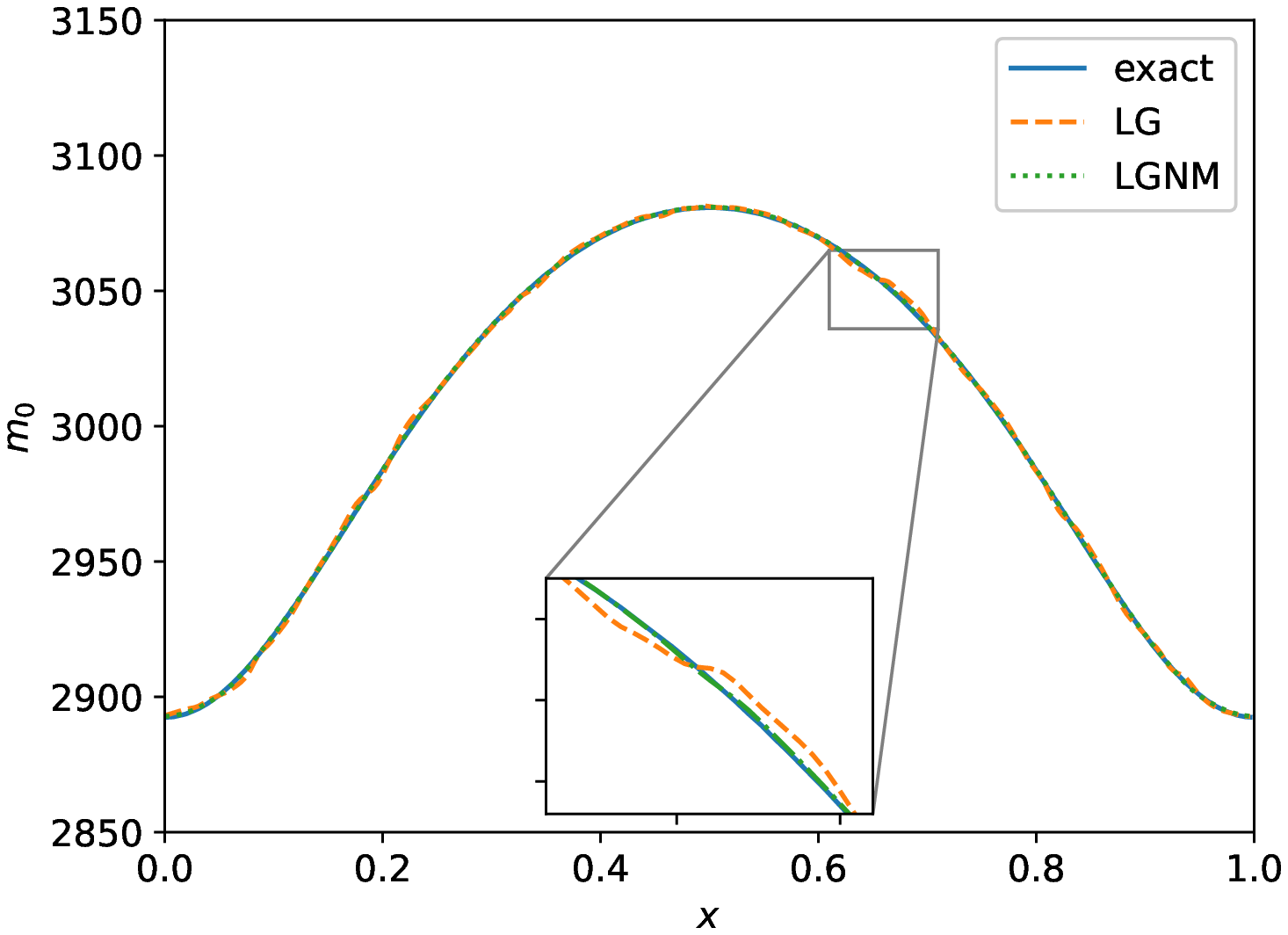}
	    \end{minipage}
	    }
	    \caption{Example \ref{ex:gauss-source}: Gaussian source problem. Numerical solutions of $m_0$ and $m_1$ at $t=0.5$ with $N=5$ in the optically thin regime ($\sigma_s=1$). Left: results from the original initial condition, equation   \eqref{eq:gauss-source-init}; right: results from the original initial condition $\times$ 1000,  equation \eqref{eq:gauss-source-init-amplify}.}
	    \label{fig:gauss-source-scale}
	\end{figure}

	{
	We also test the performance of the ML closure model with non-periodic boundary conditions. In the initial Gaussian distribution \eqref{eq:gauss-source-init}, we take $c_1=0.5$, $c_2=10^{-6}$, $x_0=0.6$ and $\theta=0.005$. The reflective boundary conditions are imposed on both the left and right boundaries. Assume that the grid points in the interior domain are $x_j=(j+\frac{1}{2})\Delta x$ with $j=0,\cdots,N_x-1$. To treat the reflective boundary conditions for the moment closure system, we set several ghost points outside of the domain: $x_j=(j+\frac{1}{2})\Delta x$ with $j=-3,-2,-1$ on the left and $j=N_x,N_x+1,N_x+2$ on the right. In each time step, the moments at the ghost points are updated by
	\begin{equation}
		m_k(x_j,t_n) = (-1)^k m_{k}(x_{-j-1},t_n), \quad j=-3,-2,-1
	\end{equation}
	and
	\begin{equation}
		m_k(x_j,t_n) = (-1)^k m_{k}(x_{2N_x-j-1},t_n), \quad j=N_x,N_x+1,N_x+2.
	\end{equation}

	The numerical results with $N=5$ in the optically thin regime ($\sigma_s=\sigma_t=1$) are shown in Figure \ref{fig:gauss-source-reflect-bc}. We observe that the $P_N$ closure has large deviations from the exact solution to the kinetic model. The $FP_N$ closure is more accurate than $P_N$ in $m_0$ and $m_1$, but has relatively larger error compared to that of the LGNM model, especially near the boundaries. Moreover, the LGNM model predicts the higher order moments $m_4$ and $m_5$ much more accurately than both the $P_N$ and $FP_N$ models. 
	\begin{figure}
	    \centering
	    \subfigure[$m_0$]{
	    \begin{minipage}[b]{0.46\textwidth}
	    \includegraphics[width=1\textwidth]{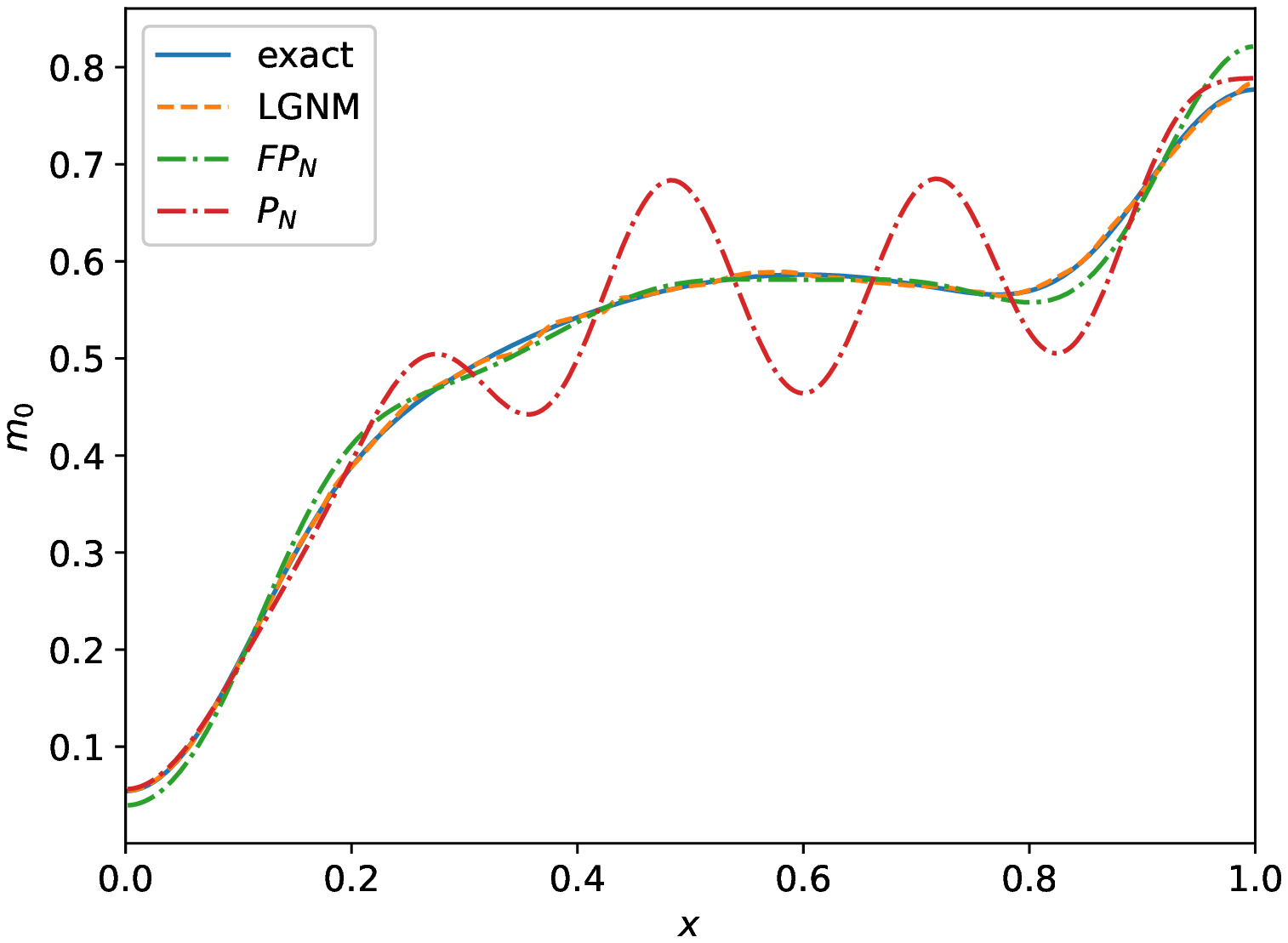}
	    \end{minipage}
	    }
	    \subfigure[$m_1$]{
	    \begin{minipage}[b]{0.46\textwidth}    
	    \includegraphics[width=1\textwidth]{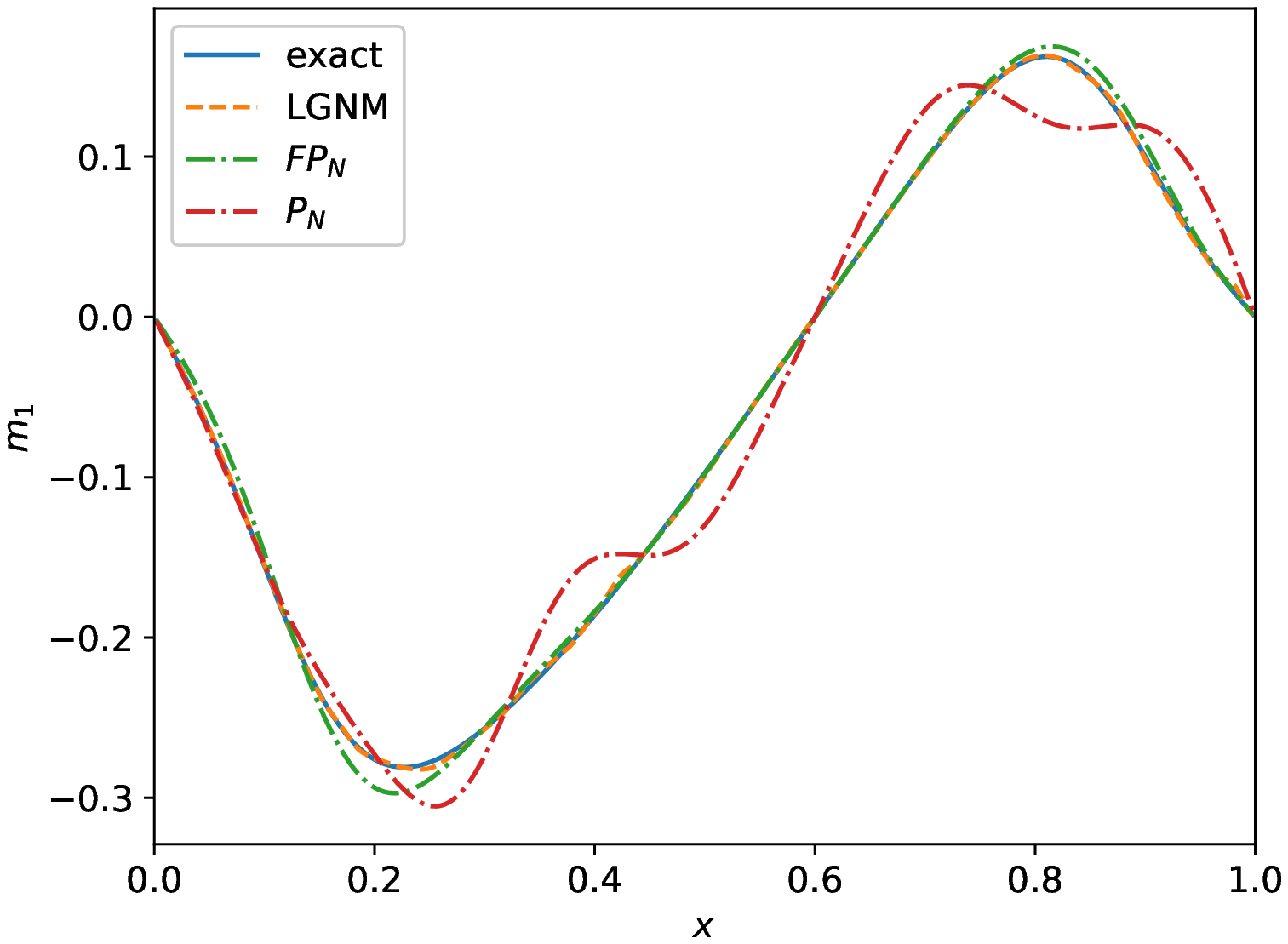}
	    \end{minipage}
	    }
	    \\
	    \subfigure[$m_4$]{
	    \begin{minipage}[b]{0.46\textwidth}
	    \includegraphics[width=1\textwidth]{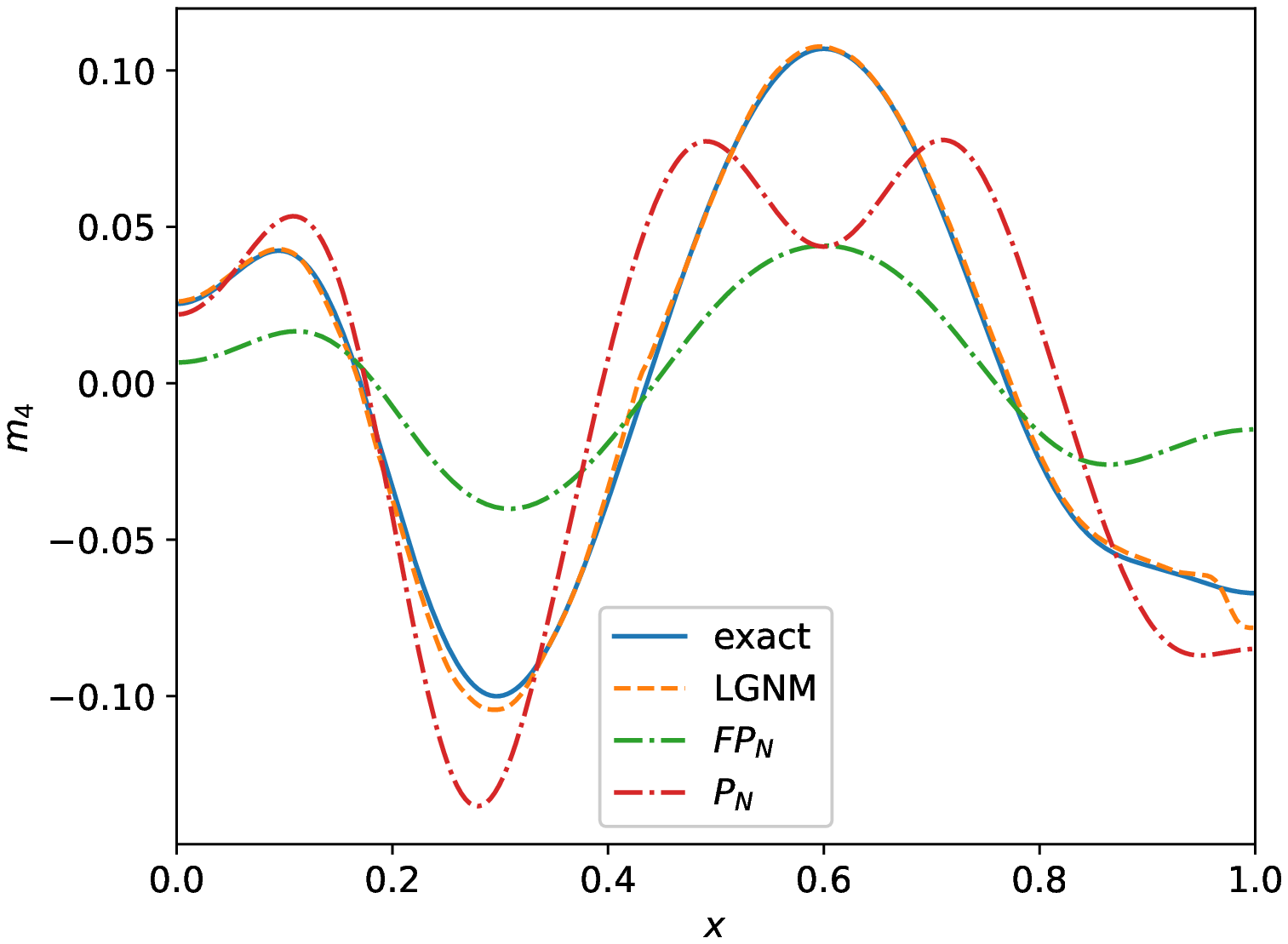}
	    \end{minipage}
	    }
	    \subfigure[$m_5$]{
	    \begin{minipage}[b]{0.46\textwidth}    
	    \includegraphics[width=1\textwidth]{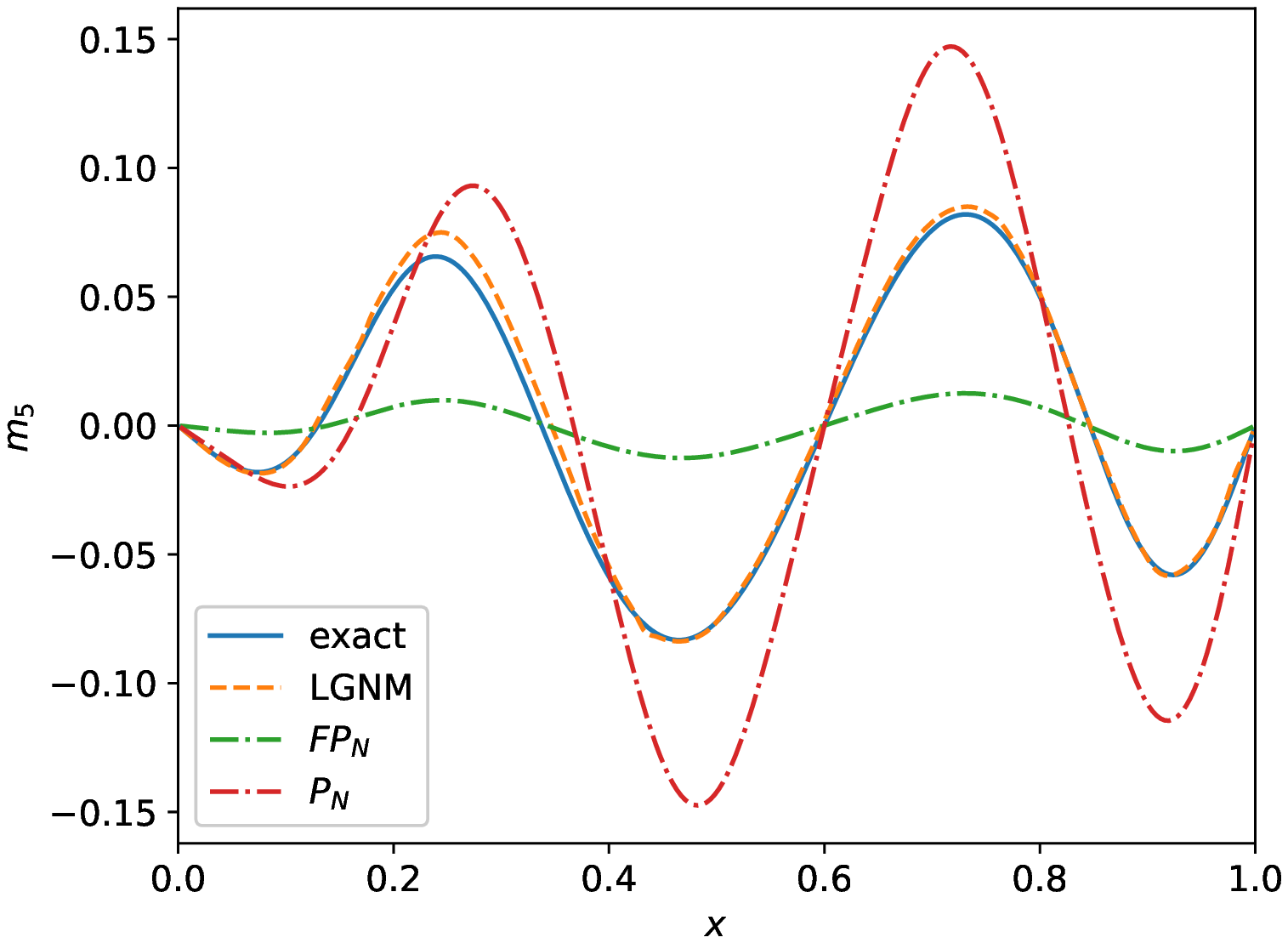}
	    \end{minipage}
	    }	    
	    \caption{{Example \ref{ex:gauss-source}: Gaussian source problem with reflective boundary conditions. Numerical solutions of $m_0$, $m_1$, $m_4$ and $m_5$ at $t=0.5$ with $N=5$ in the optically thin regime ($\sigma_s=\sigma_t=1$).}}
	    \label{fig:gauss-source-reflect-bc}
	\end{figure}

	We also remark that there exist a lot of work on inflow boundary conditions (including vacuum boundary condition as a special case), see e.g. \cite{hauck2011high,bunger2020stable,fan2020nonlinear2}. This remains an open problem, even for analytical moment models. This is not the focus of the current paper and we leave it to our future work.
	}
\end{exam}

{
\begin{exam}[higher wave number test]\label{ex:wave-number}
	To further test the generalizability of our ML closure model, we consider initial conditions across a range of wave numbers that extends outside of the wave
	numbers present in the training data:
	\begin{equation}\label{eq:init-wave-number}
		f_0(x,v) = 2 + \sin(2\pi k x + \phi).
	\end{equation}
	Here $k$ is the wave number of the initial data and $\phi$ is a random number from $[0, 2\pi)$. In the test, we take $k=1,2,\cdots,25$. This is a challenging test, since the magnitudes in the training data in \eqref{eq:init-fourier-series} decay with the wave number, so it is naturally difficult for the well-trained model to capture the correct behavior for the initial condition in \eqref{eq:init-wave-number} with non-decaying magnitudes. We test the optically thin regime, i.e., $\sigma_s=\sigma_t=1$, since the intermediate regime and the optically thick regime are relatively easy to capture for the moment closure model. We run the simulations to time $t=0.4$.

	Figure \ref{fig:wave-number-error} shows the relative $L^2$ error of $m_0$ when using our LGNM closure model with $N=5$. Overall, the method performs well for these waves that are outside of the training data. For low wave numbers ($1\le k\le 10$), the error stays in the magnitude of $10^{-4}$ to $10^{-3}$. For high wave numbers ($11\le k\le 25$), the error slightly increases and saturates in the magnitude of $10^{-3}$ to $10^{-2}$. This shows that our ML closure model generalizes well when the initial conditions are outside of the training data.
	\begin{figure}
	    \centering
	    \includegraphics[width=0.5\textwidth]{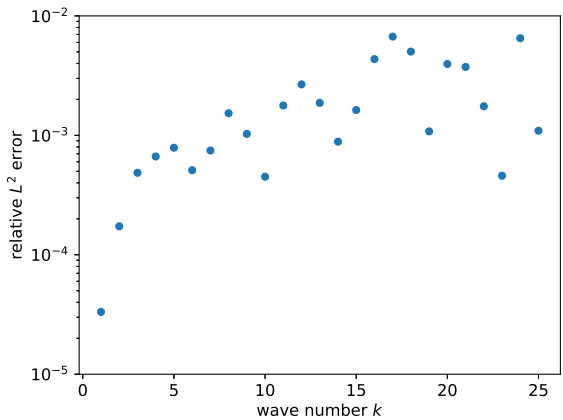}
	    \caption{{Example \ref{ex:wave-number}: higher wave number test. relative $L^2$ error of $m_0$ at $t=0.4$ and optically thin regime $\sigma_s=\sigma_t=1$. The number of moments in the system is $N=5$.}}
	    \label{fig:wave-number-error}
	\end{figure}
\end{exam}
}

\begin{exam}[two-material problem]\label{ex:two-material}
	The two-material problem models a domain with a discontinuous material cross sections \cite{larsen1989asymptotic}.
	In our problem setup, there exist two discontinuities $0<x_1<x_2<1$ in the domain, and $\sigma_s$ and $\sigma_a$ are piecewise constant functions:
	$$ 
	\sigma_s(x)=\left\{
	\begin{aligned}
	& \sigma_{s1}, \quad  ~ x_1 < x < x_2, \\
	& \sigma_{s2}, \quad  ~ 0\le x < x_1 ~ \textrm{or} ~ x_2\le x < 1.
	\end{aligned}
	\right.
	$$
	and
	$$ 
	\sigma_a(x)=\left\{
	\begin{aligned}
	& \sigma_{a1}, \quad  ~ x_1 < x < x_2, \\
	& \sigma_{a2}, \quad  ~ 0\le x < x_1 ~ \textrm{or} ~ x_2\le x < 1.
	\end{aligned}
	\right.
	$$
	Specifically, we take $x_1=0.3$, $x_2=0.7$, $\sigma_{s1}=1$, $\sigma_{s2}=10$ and $\sigma_{a1}=\sigma_{a2}=0$. 

	{	
	In Figure \ref{fig:two-material-compare-N5-FPN} and Figure \ref{fig:two-material-compare-N9-FPN}, we compare our LGNM moment closure with the $P_N$ closure and the $FP_N$ closure for $N=5$ and $N=9$, respectively. The gray background region is in the optically thin regime and the white background region is in the intermediate regime. Here, we tune the parameter $\nu$ in the $FP_N$ closure \eqref{eq:FPN-moment-system} in the range of $[1, 100]$ and find that $\nu=20$ is the optimal parameter for the best performance of $m_0$ in $FP_N$ closure. For both $N=5$ and $N=9$, we observe that the $P_N$ closure has large deviations from the kinetic solution in the optically thin portion of the domain, see the gray part in Figure \ref{fig:two-material-compare-N5-FPN} and Figure \ref{fig:two-material-compare-N9-FPN}. As a comparison, in the ML closure and the $FP_N$ closure, the low order moments $m_0$ and $m_1$ agree well with the kinetic model, see Figure \ref{fig:two-material-compare-N5-FPN} (a) and Figure \ref{fig:two-material-compare-N5-FPN} (b) for $N=5$ and Figure \ref{fig:two-material-compare-N9-FPN} (a) and Figure \ref{fig:two-material-compare-N9-FPN} (b) for $N=9$. For the high order moments, the ML closure behaves slightly better than the $FP_N$ closure in $m_4$ and $m_5$ for $N=5$, see Figure \ref{fig:two-material-compare-N5-FPN} (c) and Figure \ref{fig:two-material-compare-N5-FPN} (d). Moreover, for $N=9$, the ML closure behaves much better than the $FP_N$ closure in the high order moments $m_4$, $m_5$, $m_8$ and $m_9$. This illustrates the potential for better performance of our ML closure over the traditional approaches including the $P_N$ and $FP_N$ closures, especially for the case of capturing the higher order moments.
	}
	\begin{figure}
	    \centering
	    \subfigure[$m_0$]{
	    \begin{minipage}[b]{0.46\textwidth}
	    \includegraphics[width=1\textwidth]{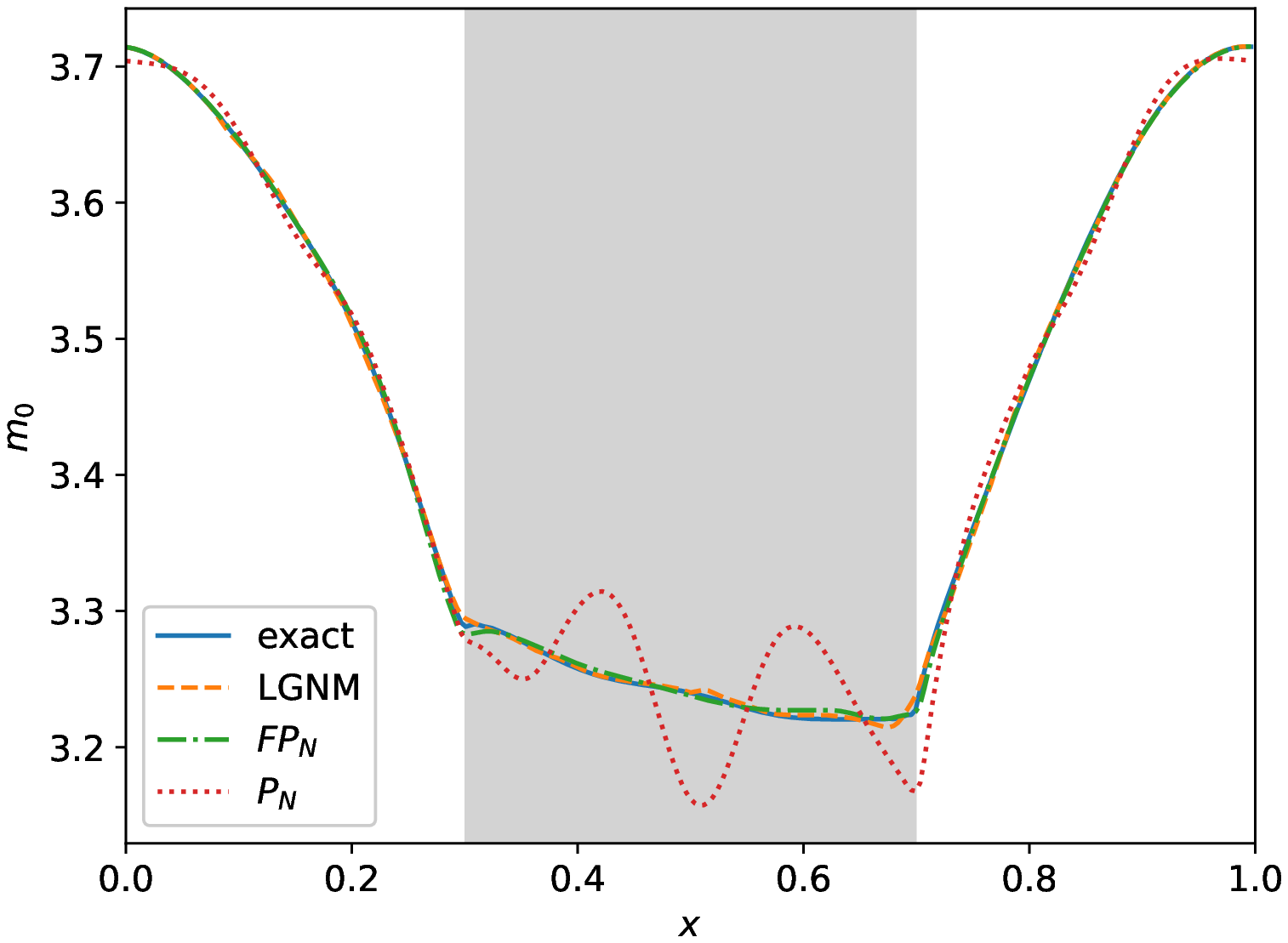}
	    \end{minipage}
	    }
	    \subfigure[$m_1$]{
	    \begin{minipage}[b]{0.46\textwidth}    
	    \includegraphics[width=1\textwidth]{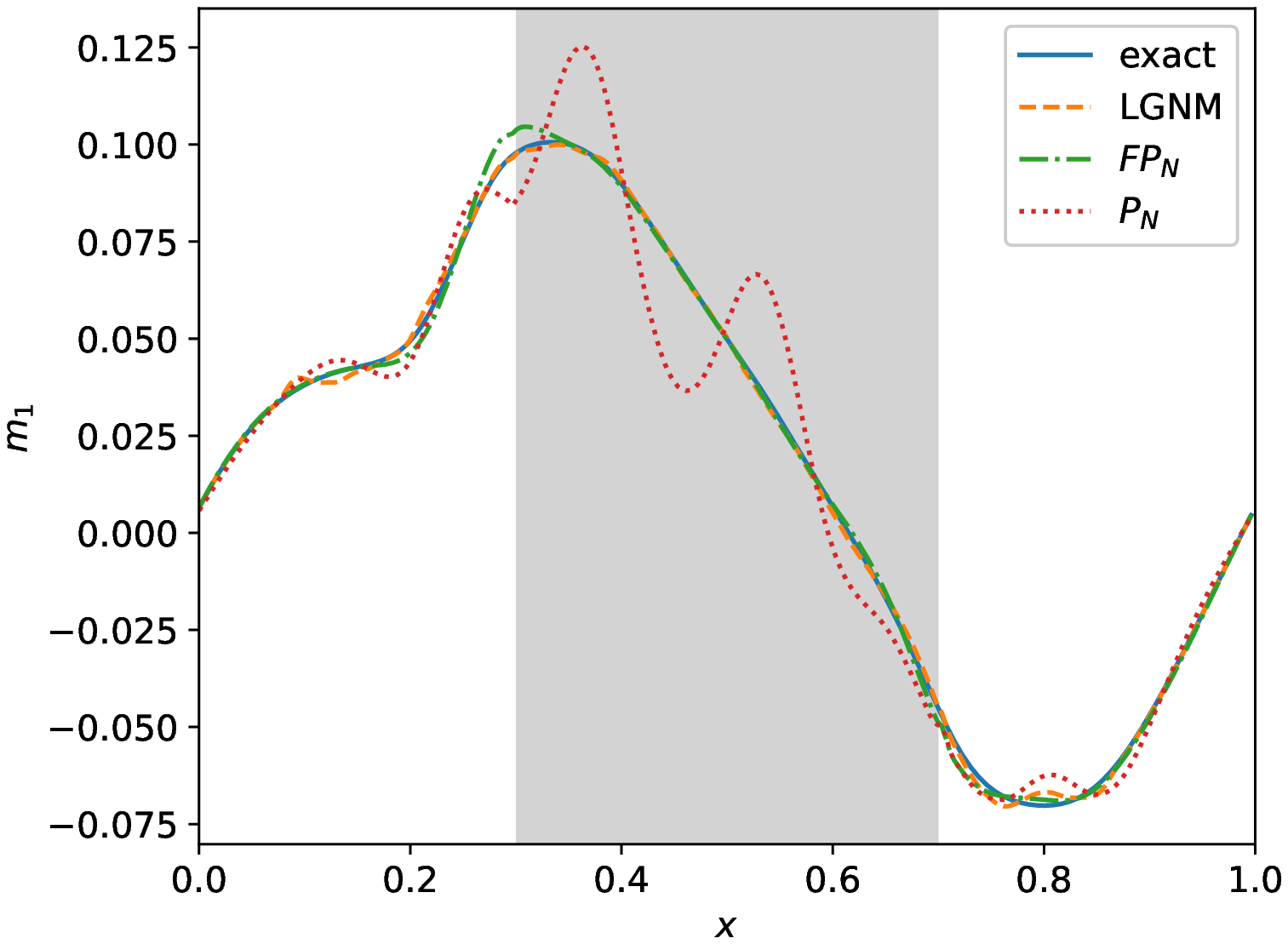}
	    \end{minipage}
	    } 
	    \\
	    \subfigure[$m_4$]{
	    \begin{minipage}[b]{0.46\textwidth}
	    \includegraphics[width=1\textwidth]{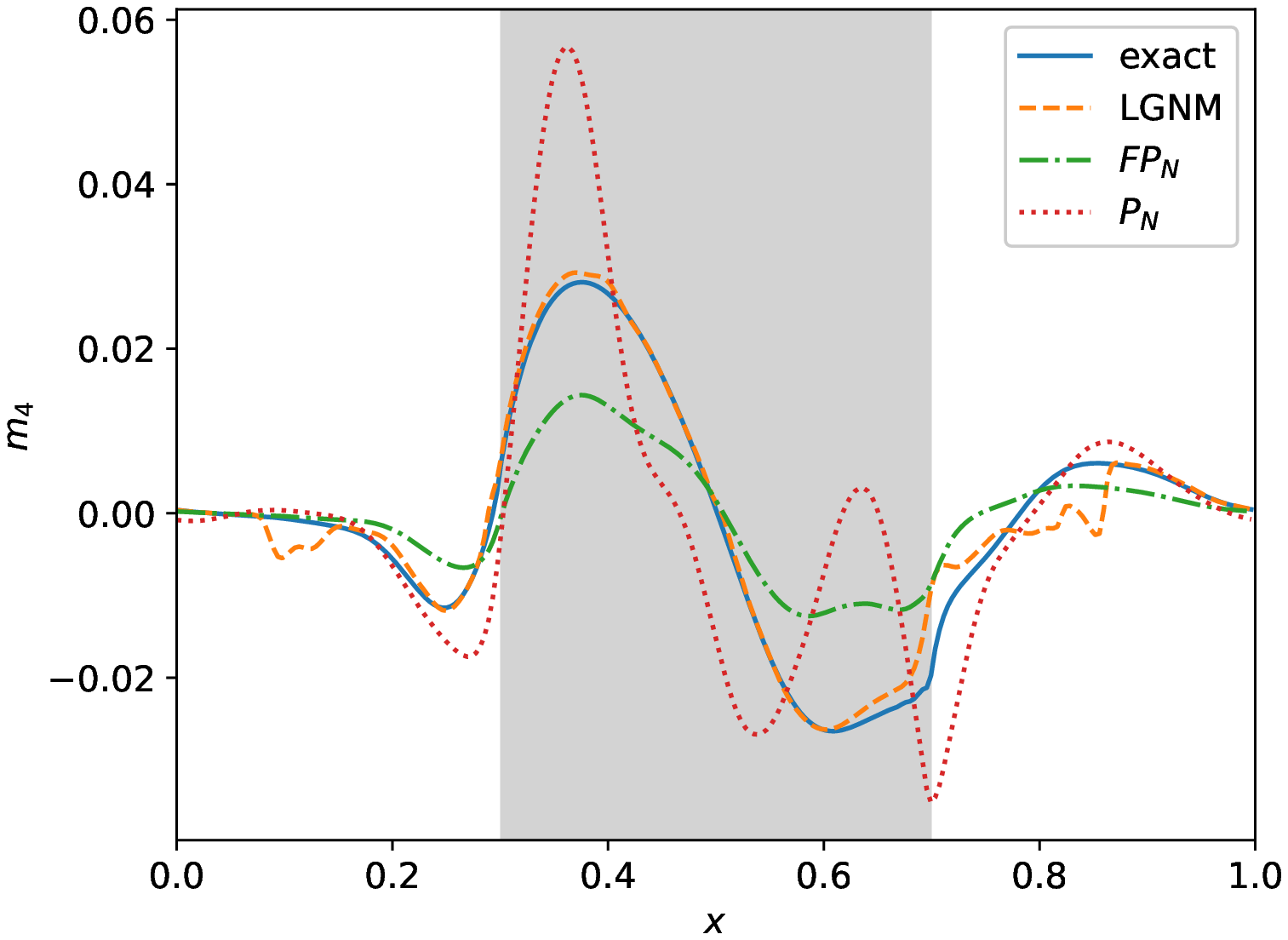}
	    \end{minipage}
	    }
	    \subfigure[$m_5$]{
	    \begin{minipage}[b]{0.46\textwidth}    
	    \includegraphics[width=1\textwidth]{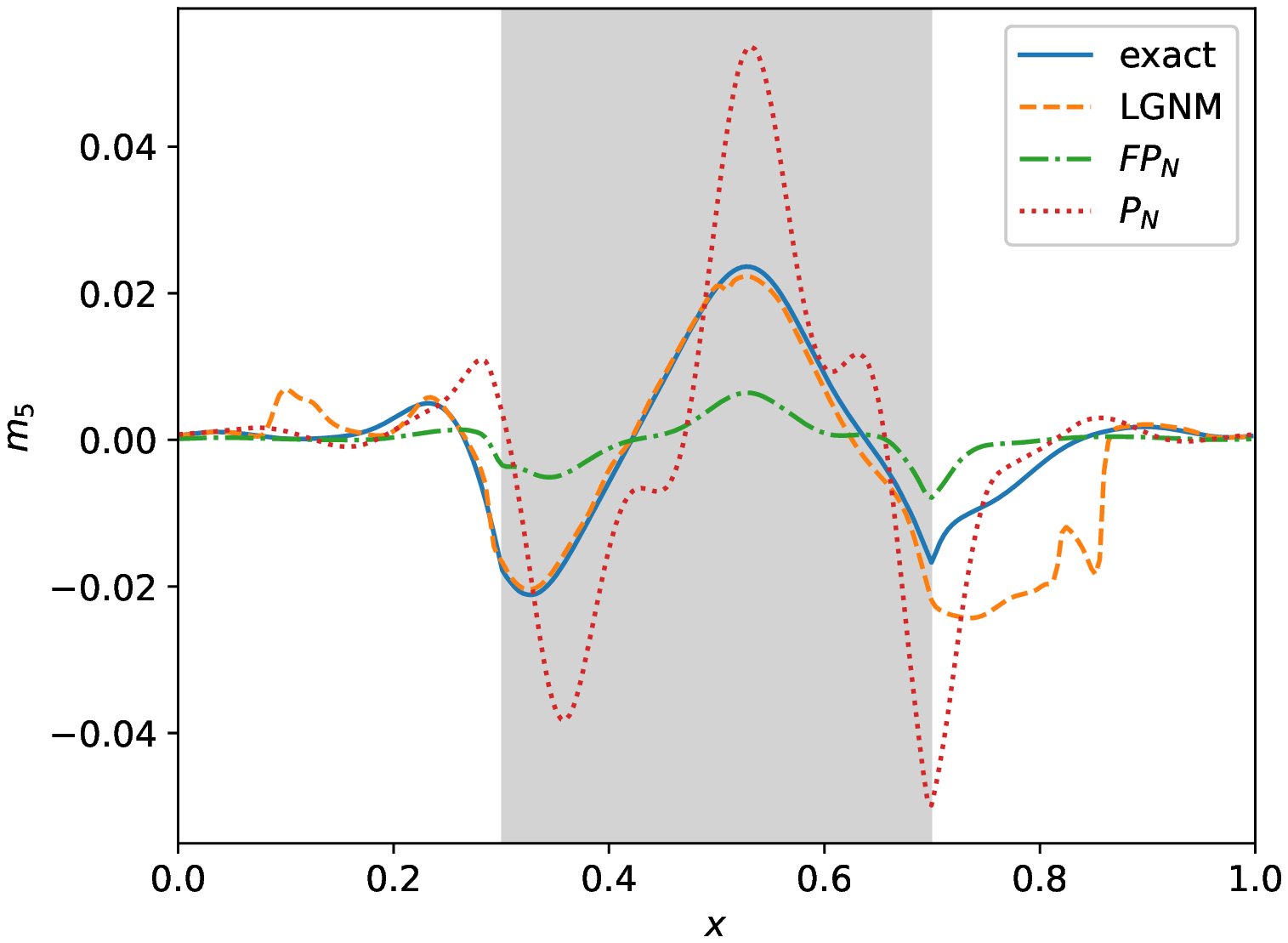}
	    \end{minipage}
	    }	    
	    \caption{{Example \ref{ex:two-material}: two-material problem. Numerical solutions of $m_0$, $m_1$, $m_4$ and $m_5$ at $t=0.4$ with $N=5$. We take $\nu=20$ in the $FP_N$ closure. The gray part in the middle is in the optically thin regime and the other part is in the intermediate regime.}}
	    \label{fig:two-material-compare-N5-FPN}
	\end{figure}

	\begin{figure}
	    \centering
	    \subfigure[$m_0$]{
	    \begin{minipage}[b]{0.46\textwidth}
	    \includegraphics[width=1\textwidth]{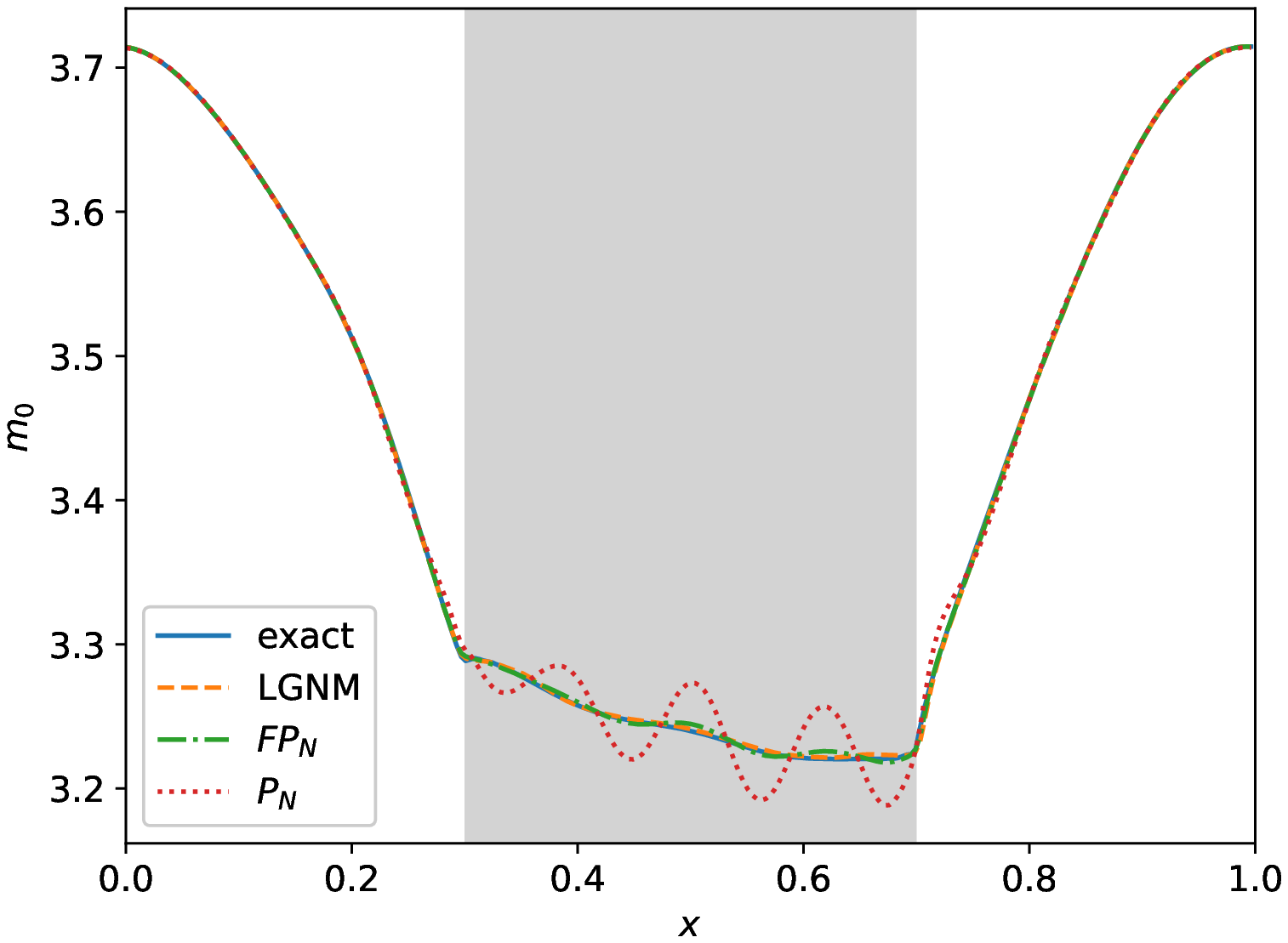}
	    \end{minipage}
	    }
	    \subfigure[$m_1$]{
	    \begin{minipage}[b]{0.46\textwidth}    
	    \includegraphics[width=1\textwidth]{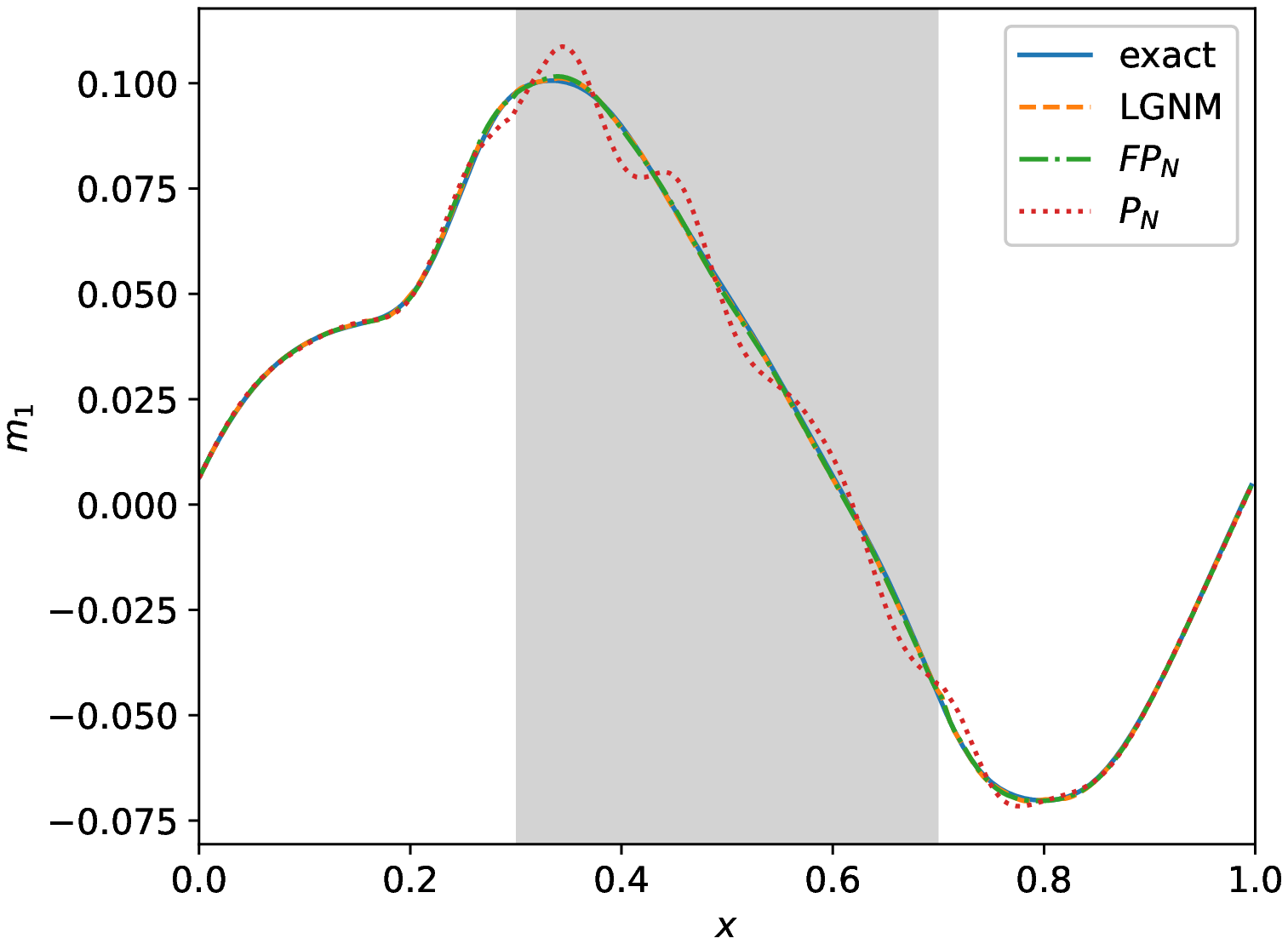}
	    \end{minipage}
	    } 
	    \\
	    \subfigure[$m_4$]{
	    \begin{minipage}[b]{0.46\textwidth}
	    \includegraphics[width=1\textwidth]{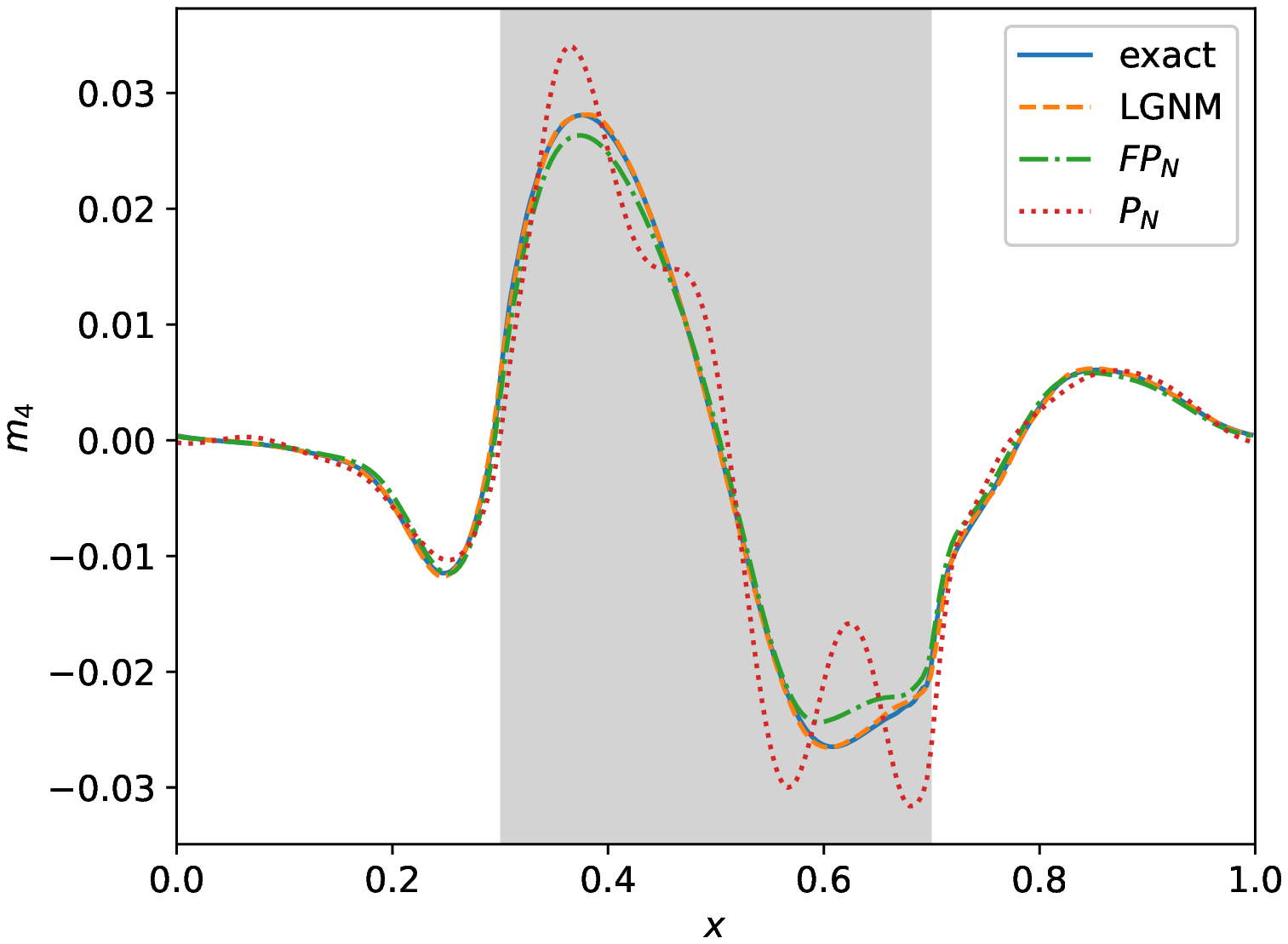}
	    \end{minipage}
	    }
	    \subfigure[$m_5$]{
	    \begin{minipage}[b]{0.46\textwidth}    
	    \includegraphics[width=1\textwidth]{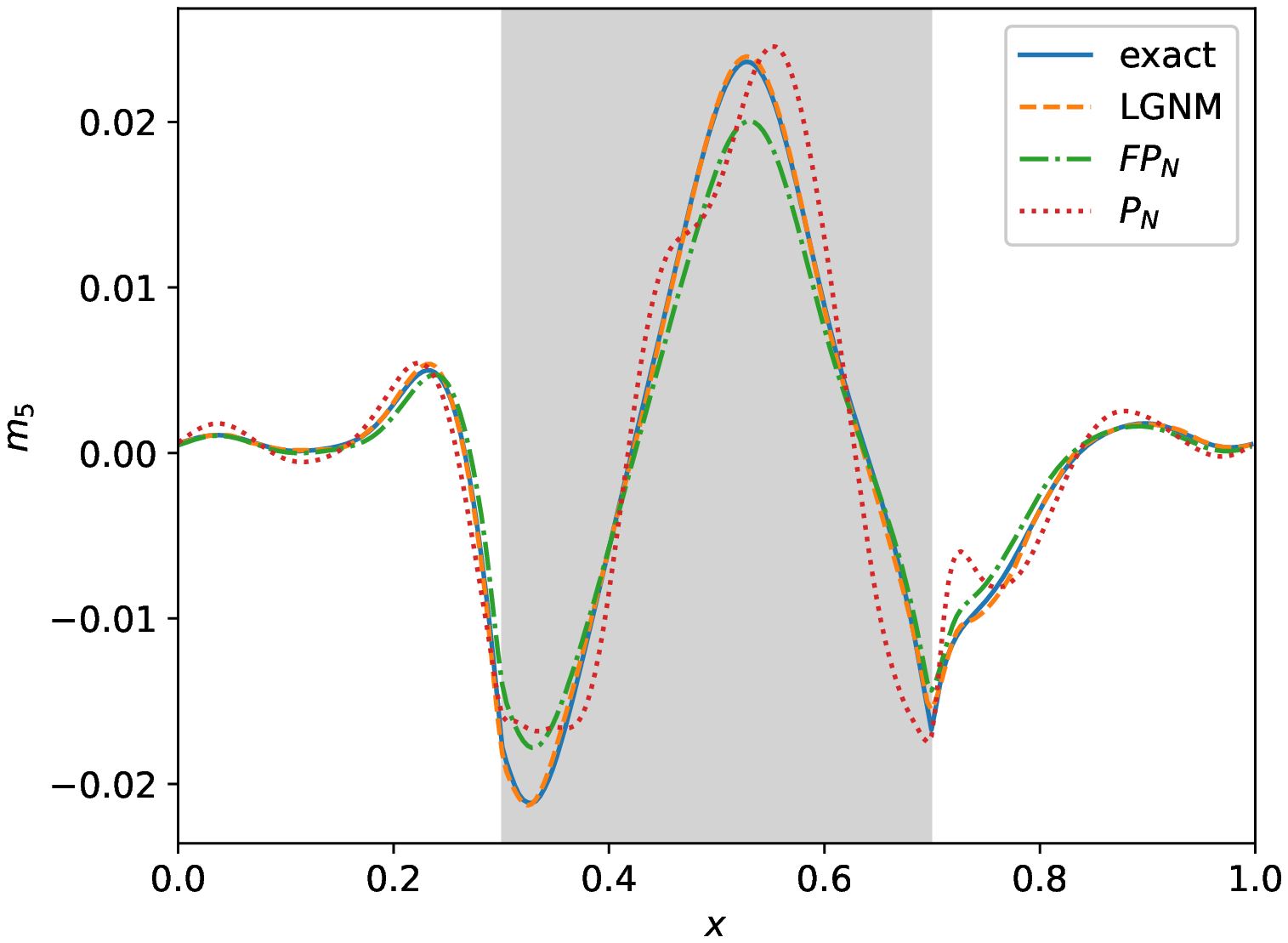}
	    \end{minipage}
	    }	    
	    \\
	    \subfigure[$m_8$]{
	    \begin{minipage}[b]{0.46\textwidth}
	    \includegraphics[width=1\textwidth]{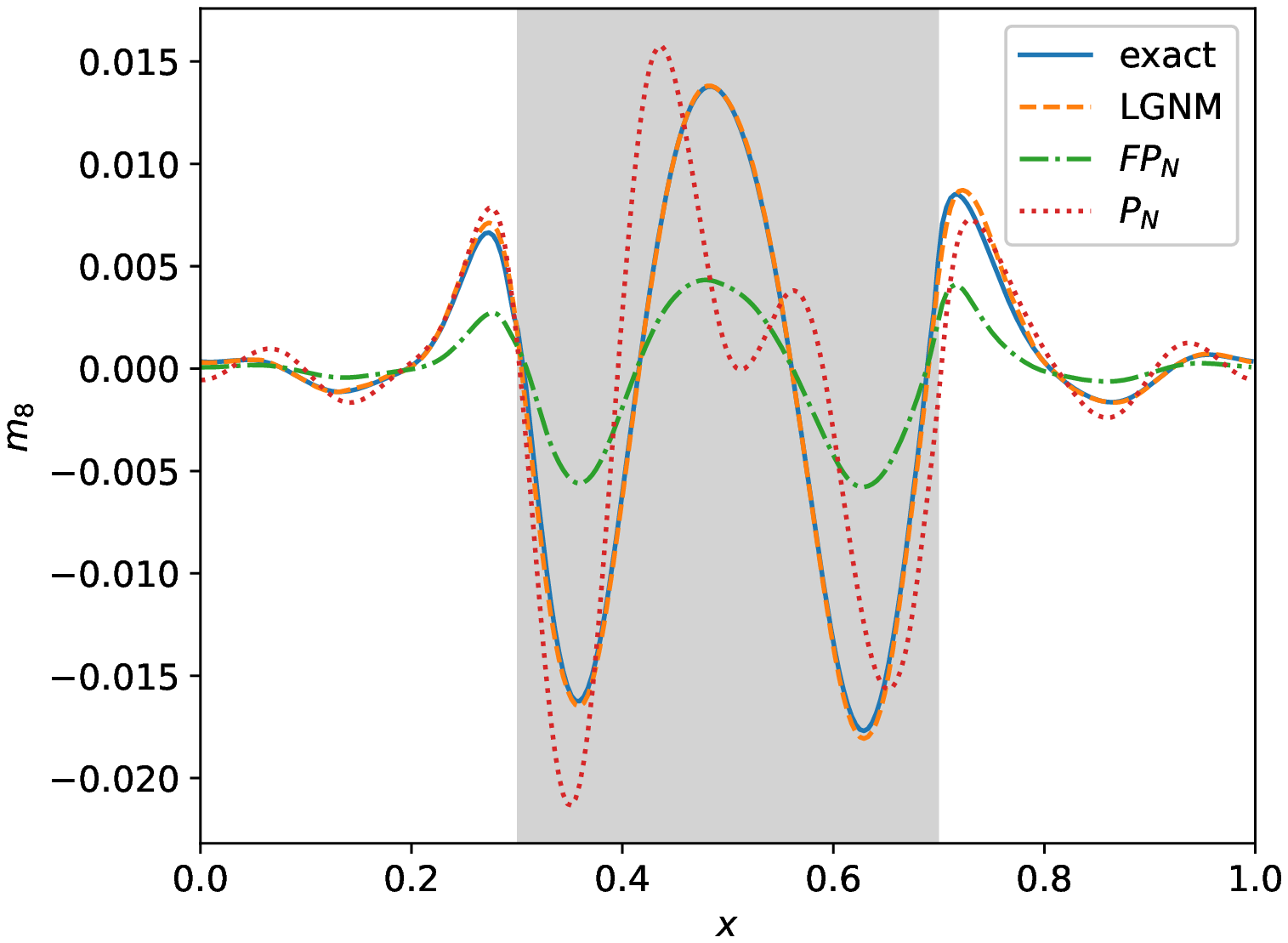}
	    \end{minipage}
	    }
	    \subfigure[$m_9$]{
	    \begin{minipage}[b]{0.46\textwidth}    
	    \includegraphics[width=1\textwidth]{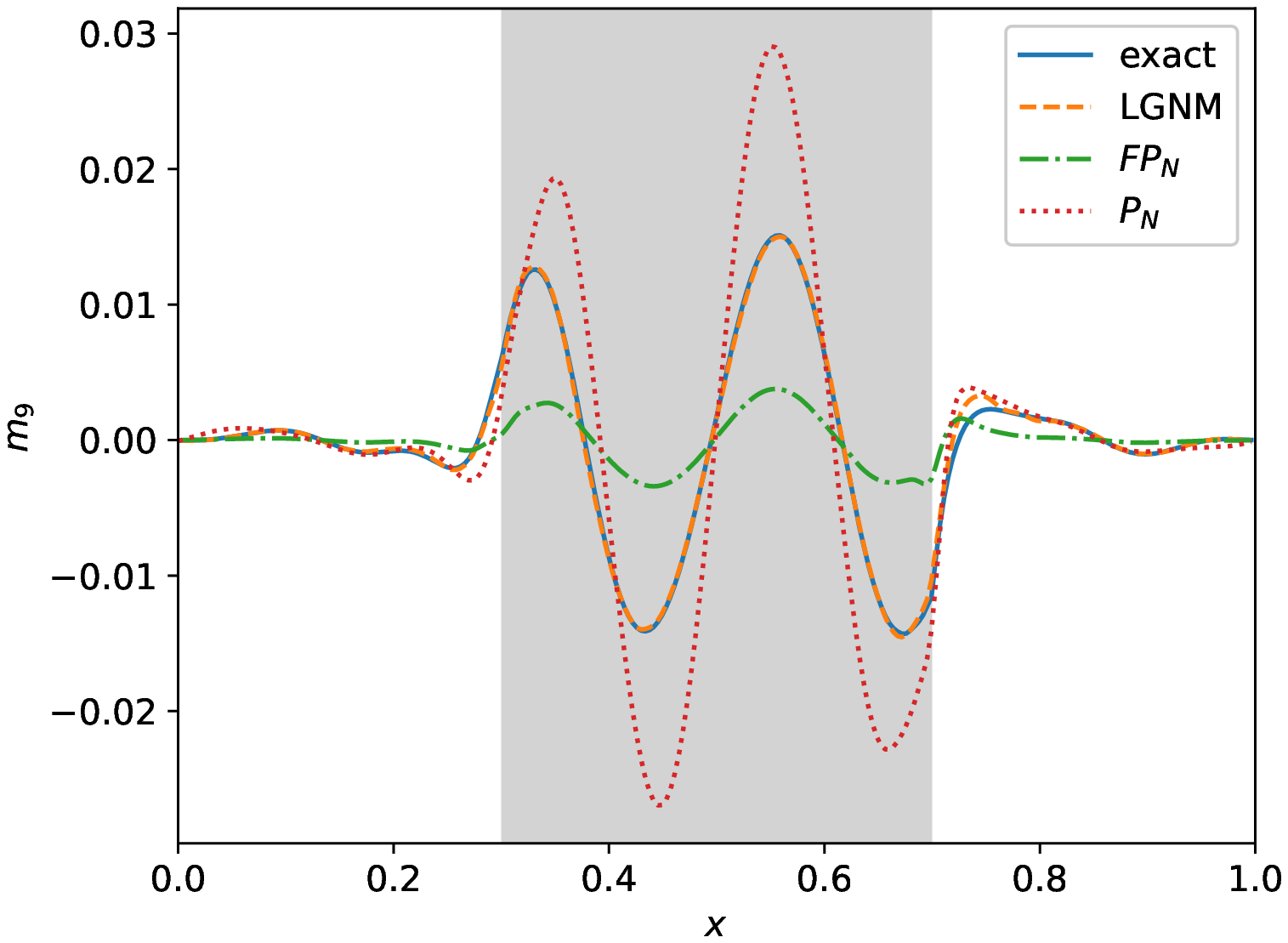}
	    \end{minipage}
	    }	    
	    \caption{{Example \ref{ex:two-material}: two-material problem. Numerical solutions of $m_0$, $m_1$, $m_4$, $m_5$, $m_8$ and $m_9$ at $t=0.4$ with $N=9$. We take $\nu=20$ in the $FP_N$ closure. The gray part in the middle is in the optically thin regime and the other part is in the intermediate regime.}}
	    \label{fig:two-material-compare-N9-FPN}
	\end{figure}

	{
	We further investigate the performance of the ML closure quantitatively. In Table \ref{tb:two-material-error}, we show the relative $L^2$ errors of $m_0$ and $m_N$ moments at $t=0.4$ with $N=5,6,7,8,9$. We first note that all three closure models show convergence with an increasing number of moments. Second, we note that for the zeroth order moment $m_0$, the $FP_N$ closure and the ML closure have similar performance, and both are more accurate than $P_N$ closure. Third, we note that for high order moments, $m_N$, the ML closure is much more accurate than both the $P_N$ and $FP_N$ closures.
	}
	\begin{table}[htbp]
		\centering
		\label{tb:two-material-error}
		\begin{tabular}{c|c|c|c|c|c|c}
			\hline
			 \multirow{2}[4]{*}{$N$} & \multicolumn{3}{c|}{relative $L^2$ error of $m_0$} & \multicolumn{3}{c}{relative $L^2$ error of $m_N$}  \\
			\cline{2-7} & $P_N$ & $FP_N$ &  LGNM & $P_N$ & $FP_N$ &  LGNM  \\
			\hline
			5 &	8.61e-03 & 9.55e-04	& 7.28e-04 & 1.14e+00 &	7.72e-01 & 6.06e-01 \\
			6 &	4.93e-03 & 6.40e-04	& 6.74e-04 & 1.01e+00 &	7.79e-01 & 7.81e-01 \\
			7 & 5.84e-03 & 6.83e-04 & 5.71e-04 & 9.36e-01 & 7.66e-01 & 1.13e-01 \\
			8 & 5.56e-03 & 6.78e-04 & 4.52e-04 & 9.03e-01 & 7.68e-01 & 8.92e-02 \\
			9 & 4.00e-03 & 5.25e-04 & 4.21e-04 & 8.24e-01 & 7.71e-01 & 5.06e-02 \\
			\hline    
		\end{tabular}
		\caption{{Example \ref{ex:two-material}: two-material problem. The relative $L^2$ errors of $m_0$ and $m_N$ at $t=0.4$ with $N=5,6,7,8,9$. We take $\nu=20$ in the $FP_N$ closure.}}
	\end{table}

\end{exam}

\section{Conclusion}\label{sec:conclusion}

In this work, we investigate the moment closure problem for the RTE in slab geometry and learn a closure relation from data. Instead of learning the moment itself, we use neural networks to directly learn its gradient. This new approach is consistent with the exact closure we derive for the free streaming limit and also provides a natural output normalization. Moreover, we incorporate the scale invariance of the closure model into the neural networks, which brings better generalization and performance, especially when applied to initial conditions that have their magnitude outside of the training data. A variety of benchmark tests, including the variable scattering problem, the Gaussian source problem and the two-material problem, were investigated. All tests show that the Learn Gradient ML closure model has both good accuracy as well as that the ML closure model using a simple training procedure has a strong generalization property, i.e., we did not need additional training to maintain accuracy even when applied to problems with discontinuities in the scattering cross section.

{We also remark that the methodology of learning the gradients can be generalized  to the multidimensional case. In this case, we would relate the spatial derivatives in each direction of the higher order moments to the derivatives of the lower order moments in all the spatial directions. We are currently working on this topic and hope to report the progress in the near future.}

Finally, we point out that hyperbolicity is an important property in moment closure models, which is difficult to enforce for traditional closure models \cite{cai2014globally,li2021direct} as well as ML models \cite{huang2020learning}. Our current ML closure model is not able to preserve hyperbolicity and thus have some numerical instabilities. How to incorporate hyperbolicity in the ML closure model is certainly an interesting topic and constitutes our ongoing work \cite{huang2021ml2,huang2021ml3}.

\section*{Acknowledgment}

We thank Michael M. Crockatt in Sandia National Laboratories for providing numerical solver for the radiative transfer equation. We acknowledge the High Performance Computing Center (HPCC) at Michigan State University for providing computational resources that have contributed to the research results reported within this paper. {We thank two anonymous reviewers for providing helpful comments on earlier draft of the manuscript.}

\bibliographystyle{abbrv}

\begin{thebibliography}{10}

\bibitem{alldredge2014adaptive}
G.~W. Alldredge, C.~D. Hauck, D.~P. OLeary, and A.~L. Tits.
\newblock Adaptive change of basis in entropy-based moment closures for linear
  kinetic equations.
\newblock {\em Journal of Computational Physics}, 258:489--508, 2014.

\bibitem{alldredge2012high}
G.~W. Alldredge, C.~D. Hauck, and A.~L. Tits.
\newblock High-order entropy-based closures for linear transport in slab
  geometry ii: A computational study of the optimization problem.
\newblock {\em SIAM Journal on Scientific Computing}, 34(4):B361--B391, 2012.

\bibitem{alldredge2016approximating}
G.~W. Alldredge, R.~Li, and W.~Li.
\newblock Approximating the $ {M}_2 $ method by the extended quadrature method
  of moments for radiative transfer in slab geometry.
\newblock {\em Kinetic \& Related Models}, 9(2):237, 2016.

\bibitem{bird1995molecular}
G.~Bird.
\newblock {\em Molecular {G}as {D}ynamics and the {D}irect {S}imulation of
  {G}as {F}lows}.
\newblock Clarendon Press, Oxford, United Kingdom, 1994.

\bibitem{bois2020neural}
L.~Bois, E.~Franck, L.~Navoret, and V.~Vigon.
\newblock A neural network closure for the {E}uler-{P}oisson system based on
  kinetic simulations.
\newblock {\em arXiv preprint arXiv:2011.06242}, 2020.

\bibitem{brunton2016discovering}
S.~L. Brunton, J.~L. Proctor, and J.~N. Kutz.
\newblock Discovering governing equations from data by sparse identification of
  nonlinear dynamical systems.
\newblock {\em Proceedings of the national academy of sciences},
  113(15):3932--3937, 2016.

\bibitem{buchan2015pod}
A.~G. Buchan, A.~Calloo, M.~G. Goffin, S.~Dargaville, F.~Fang, C.~C. Pain, and
  I.~M. Navon.
\newblock A {POD} reduced order model for resolving angular direction in
  neutron/photon transport problems.
\newblock {\em Journal of Computational Physics}, 296:138--157, 2015.

\bibitem{bunger2020stable}
J.~B{\"u}nger, N.~Sarna, and M.~Torrilhon.
\newblock Stable boundary conditions and discretization for pn equations.
\newblock {\em arXiv preprint arXiv:2004.02497}, 2020.

\bibitem{cai2014globally}
Z.~Cai, Y.~Fan, and R.~Li.
\newblock Globally hyperbolic regularization of {G}rad's moment system.
\newblock {\em Communications on pure and applied mathematics}, 67(3):464--518,
  2014.

\bibitem{chandrasekhar1944radiative}
S.~Chandrasekhar.
\newblock On the radiative equilibrium of a stellar atmosphere.
\newblock {\em The Astrophysical Journal}, 99:180, 1944.

\bibitem{crockatt2017arbitrary}
M.~M. Crockatt, A.~J. Christlieb, C.~K. Garrett, and C.~D. Hauck.
\newblock An arbitrary-order, fully implicit, hybrid kinetic solver for linear
  radiative transport using integral deferred correction.
\newblock {\em Journal of Computational Physics}, 346:212--241, 2017.

\bibitem{crockatt2019hybrid}
M.~M. Crockatt, A.~J. Christlieb, C.~K. Garrett, and C.~D. Hauck.
\newblock Hybrid methods for radiation transport using diagonally implicit
  runge--kutta and space--time discontinuous galerkin time integration.
\newblock {\em Journal of Computational Physics}, 376:455--477, 2019.

\bibitem{fan2020nonlinear}
Y.~Fan, R.~Li, and L.~Zheng.
\newblock A nonlinear hyperbolic model for radiative transfer equation in slab
  geometry.
\newblock {\em SIAM Journal on Applied Mathematics}, 80(6):2388--2419, 2020.

\bibitem{fan2020nonlinear2}
Y.~Fan, R.~Li, and L.~Zheng.
\newblock A nonlinear moment model for radiative transfer equation in slab
  geometry.
\newblock {\em Journal of Computational Physics}, 404:109128, 2020.

\bibitem{frank2012perturbed}
M.~Frank, C.~D. Hauck, and E.~Olbrant.
\newblock Perturbed, entropy-based closure for radiative transfer.
\newblock {\em arXiv preprint arXiv:1208.0772}, 2012.

\bibitem{guo2016sparse}
W.~Guo and Y.~Cheng.
\newblock A sparse grid discontinuous galerkin method for high-dimensional
  transport equations and its application to kinetic simulations.
\newblock {\em SIAM Journal on Scientific Computing}, 38(6):A3381--A3409, 2016.

\bibitem{guo2017adaptive}
W.~Guo and Y.~Cheng.
\newblock An adaptive multiresolution discontinuous galerkin method for
  time-dependent transport equations in multidimensions.
\newblock {\em SIAM Journal on Scientific Computing}, 39(6):A2962--A2992, 2017.

\bibitem{han2018solving}
J.~Han, A.~Jentzen, and W.~E.
\newblock Solving high-dimensional partial differential equations using deep
  learning.
\newblock {\em Proceedings of the National Academy of Sciences},
  115(34):8505--8510, 2018.

\bibitem{han2019uniformly}
J.~Han, C.~Ma, Z.~Ma, and W.~E.
\newblock Uniformly accurate machine learning-based hydrodynamic models for
  kinetic equations.
\newblock {\em Proceedings of the National Academy of Sciences},
  116(44):21983--21991, 2019.

\bibitem{hauck2010positive}
C.~Hauck and R.~McClarren.
\newblock Positive ${P_N}$ closures.
\newblock {\em SIAM Journal on Scientific Computing}, 32(5):2603--2626, 2010.

\bibitem{hauck2011high}
C.~D. Hauck.
\newblock High-order entropy-based closures for linear transport in slab
  geometry.
\newblock {\em Communications in Mathematical Sciences}, 9(1):187--205, 2011.

\bibitem{higham2019deep}
C.~F. Higham and D.~J. Higham.
\newblock Deep learning: An introduction for applied mathematicians.
\newblock {\em SIAM Review}, 61(4):860--891, 2019.

\bibitem{huang2021ml3}
J.~Huang, Y.~Cheng, A.~J. Christlieb, and L.~F. Roberts.
\newblock Machine learning moment closure models for the radiative transfer
  equation {III}: enforcing hyperbolicity and physical characteristic speeds.
\newblock {\em arXiv preprint arXiv:2109.00700}, 2021.

\bibitem{huang2021ml2}
J.~Huang, Y.~Cheng, A.~J. Christlieb, L.~F. Roberts, and W.-A. Yong.
\newblock Machine learning moment closure models for the radiative transfer
  equation {II}: enforcing global hyperbolicity in gradient based closures.
\newblock {\em arXiv preprint arXiv:2105.14410}, 2021.

\bibitem{huang2020learning}
J.~Huang, Z.~Ma, Y.~Zhou, and W.-A. Yong.
\newblock Learning thermodynamically stable and {G}alilean invariant partial
  differential equations for non-equilibrium flows.
\newblock {\em Journal of Non-Equilibrium Thermodynamics}, 2021.

\bibitem{jiang1996efficient}
G.-S. Jiang and C.-W. Shu.
\newblock Efficient implementation of weighted {ENO} schemes.
\newblock {\em Journal of Computational Physics}, 126(1):202--228, 1996.

\bibitem{kim2020fast}
Y.~Kim, Y.~Choi, D.~Widemann, and T.~Zohdi.
\newblock A fast and accurate physics-informed neural network reduced order
  model with shallow masked autoencoder.
\newblock {\em arXiv preprint arXiv:2009.11990}, 2020.

\bibitem{klose2002optical}
A.~D. Klose, U.~Netz, J.~Beuthan, and A.~H. Hielscher.
\newblock Optical tomography using the time-independent equation of radiative
  transfer—part 1: forward model.
\newblock {\em Journal of Quantitative Spectroscopy and Radiative Transfer},
  72(5):691--713, 2002.

\bibitem{koch2004evaluation}
R.~Koch and R.~Becker.
\newblock Evaluation of quadrature schemes for the discrete ordinates method.
\newblock {\em Journal of Quantitative Spectroscopy and Radiative Transfer},
  84(4):423--435, 2004.

\bibitem{koellermeier2021high}
J.~Koellermeier and M.~Castro.
\newblock High-order non-conservative simulation of hyperbolic moment models in
  partially-conservative form.
\newblock {\em East Asian Journal on Applied Mathematics}, 2021.

\bibitem{koellermeier2017numerical}
J.~Koellermeier and M.~Torrilhon.
\newblock Numerical study of partially conservative moment equations in kinetic
  theory.
\newblock {\em Communications in Computational Physics}, 21(4):981--1011, 2017.

\bibitem{laboure2016implicit}
V.~M. Laboure, R.~G. McClarren, and C.~D. Hauck.
\newblock Implicit filtered ${P}_{N}$ for high-energy density thermal radiation
  transport using discontinuous galerkin finite elements.
\newblock {\em Journal of Computational Physics}, 321:624--643, 2016.

\bibitem{larsen1989asymptotic}
E.~Larsen and J.~Morel.
\newblock Asymptotic solutions of numerical transport problems in optically
  thick, diffusive regimes ii.
\newblock {\em Journal of Computational Physics}, 83(1), 1989.

\bibitem{larsen2010advances}
E.~W. Larsen and J.~E. Morel.
\newblock Advances in discrete-ordinates methodology.
\newblock {\em Nuclear Computational Science}, pages 1--84, 2010.

\bibitem{lecun2015deep}
Y.~LeCun, Y.~Bengio, and G.~Hinton.
\newblock Deep learning.
\newblock {\em nature}, 521(7553):436--444, 2015.

\bibitem{lee2019deep}
K.~Lee and K.~Carlberg.
\newblock Deep conservation: A latent-dynamics model for exact satisfaction of
  physical conservation laws.
\newblock {\em arXiv preprint arXiv:1909.09754}, 2019.

\bibitem{lee2020model}
K.~Lee and K.~T. Carlberg.
\newblock Model reduction of dynamical systems on nonlinear manifolds using
  deep convolutional autoencoders.
\newblock {\em Journal of Computational Physics}, 404:108973, 2020.

\bibitem{levermore1984relating}
C.~Levermore.
\newblock Relating eddington factors to flux limiters.
\newblock {\em Journal of Quantitative Spectroscopy and Radiative Transfer},
  31(2):149--160, 1984.

\bibitem{levermore1996moment}
C.~D. Levermore.
\newblock Moment closure hierarchies for kinetic theories.
\newblock {\em Journal of statistical Physics}, 83(5):1021--1065, 1996.

\bibitem{li2021physics}
R.~Li, E.~Lee, and T.~Luo.
\newblock Physics-informed neural networks for solving multiscale mode-resolved
  phonon {B}oltzmann transport equation.
\newblock {\em arXiv preprint arXiv:2103.07983}, 2021.

\bibitem{li2021direct}
R.~Li, W.~Li, and L.~Zheng.
\newblock Direct flux gradient approximation to close moment model for kinetic
  equations.
\newblock {\em arXiv preprint arXiv:2102.07641}, 2021.

\bibitem{lou2020physics}
Q.~Lou, X.~Meng, and G.~E. Karniadakis.
\newblock Physics-informed neural networks for solving forward and inverse flow
  problems via the {B}oltzmann-{BGK} formulation.
\newblock {\em arXiv preprint arXiv:2010.09147}, 2020.

\bibitem{ma2020machine}
C.~Ma, B.~Zhu, X.-Q. Xu, and W.~Wang.
\newblock Machine learning surrogate models for {L}andau fluid closure.
\newblock {\em Physics of Plasmas}, 27(4):042502, 2020.

\bibitem{maulik2020neural}
R.~Maulik, N.~A. Garland, J.~W. Burby, X.-Z. Tang, and P.~Balaprakash.
\newblock Neural network representability of fully ionized plasma fluid model
  closures.
\newblock {\em Physics of Plasmas}, 27(7):072106, 2020.

\bibitem{mcclarren2010robust}
R.~G. McClarren and C.~D. Hauck.
\newblock Robust and accurate filtered spherical harmonics expansions for
  radiative transfer.
\newblock {\em Journal of Computational Physics}, 229(16):5597--5614, 2010.

\bibitem{mishra2020physics}
S.~Mishra and R.~Molinaro.
\newblock Physics informed neural networks for simulating radiative transfer.
\newblock {\em arXiv preprint arXiv:2009.13291}, 2020.

\bibitem{murchikova2017analytic}
E.~Murchikova, E.~Abdikamalov, and T.~Urbatsch.
\newblock Analytic closures for {M}1 neutrino transport.
\newblock {\em Monthly Notices of the Royal Astronomical Society},
  469(2):1725--1737, 2017.

\bibitem{paszke2019pytorch}
A.~Paszke, S.~Gross, F.~Massa, A.~Lerer, J.~Bradbury, G.~Chanan, T.~Killeen,
  Z.~Lin, N.~Gimelshein, L.~Antiga, A.~Desmaison, A.~Kopf, E.~Yang, Z.~DeVito,
  M.~Raison, A.~Tejani, S.~Chilamkurthy, B.~Steiner, L.~Fang, J.~Bai, and
  S.~Chintala.
\newblock Pytorch: An imperative style, high-performance deep learning library.
\newblock {\em arXiv preprint arXiv:1912.01703}, 2019.

\bibitem{peng2021reduced}
Z.~Peng, Y.~Chen, Y.~Cheng, and F.~Li.
\newblock A reduced basis method for radiative transfer equation.
\newblock {\em arXiv preprint arXiv:2103.07574}, 2021.

\bibitem{pomraning2005equations}
G.~C. Pomraning.
\newblock {\em The Equations of Radiation Hydrodynamics}.
\newblock Pergamon Press, Oxford, UK, 1973.

\bibitem{radice2013new}
D.~Radice, E.~Abdikamalov, L.~Rezzolla, and C.~D. Ott.
\newblock A new spherical harmonics scheme for multi-dimensional radiation
  transport {I}. static matter configurations.
\newblock {\em Journal of Computational Physics}, 242:648--669, 2013.

\bibitem{raissi2019physics}
M.~Raissi, P.~Perdikaris, and G.~E. Karniadakis.
\newblock Physics-informed neural networks: {A} deep learning framework for
  solving forward and inverse problems involving nonlinear partial differential
  equations.
\newblock {\em Journal of Computational Physics}, 378:686--707, 2019.

\bibitem{rim2020depth}
D.~Rim, L.~Venturi, J.~Bruna, and B.~Peherstorfer.
\newblock Depth separation for reduced deep networks in nonlinear model
  reduction: Distilling shock waves in nonlinear hyperbolic problems.
\newblock {\em arXiv preprint arXiv:2007.13977}, 2020.

\bibitem{scoggins2021machine}
J.~B. Scoggins, J.~Han, and M.~Massot.
\newblock Machine learning moment closures for accurate and efficient
  simulation of polydisperse evaporating sprays.
\newblock In {\em AIAA Scitech 2021 Forum}, page 1786, 2021.

\bibitem{shu1988efficient}
C.-W. Shu and S.~Osher.
\newblock Efficient implementation of essentially non-oscillatory
  shock-capturing schemes.
\newblock {\em Journal of computational physics}, 77(2):439--471, 1988.

\bibitem{spurr2001linearized}
R.~Spurr, T.~Kurosu, and K.~Chance.
\newblock A linearized discrete ordinate radiative transfer model for
  atmospheric remote-sensing retrieval.
\newblock {\em Journal of Quantitative Spectroscopy and Radiative Transfer},
  68(6):689--735, 2001.

\bibitem{tencer2016reduced}
J.~Tencer, K.~Carlberg, R.~Hogan, and M.~Larsen.
\newblock Reduced order modeling applied to the discrete ordinates method for
  radiation heat transfer in participating media.
\newblock In {\em ASME 2016 Heat Transfer Summer Conference collocated with the
  ASME 2016 Fluids Engineering Division Summer Meeting and the ASME 2016 14th
  International Conference on Nanochannels, Microchannels, and Minichannels}.
  American Society of Mechanical Engineers Digital Collection, 2016.

\bibitem{wang2020deep}
L.~Wang, X.~Xu, B.~Zhu, C.~Ma, and Y.-A. Lei.
\newblock Deep learning surrogate model for kinetic {L}andau-fluid closure with
  collision.
\newblock {\em AIP Advances}, 10(7):075108, 2020.

\bibitem{xiao2020using}
T.~Xiao and M.~Frank.
\newblock Using neural networks to accelerate the solution of the {B}oltzmann
  equation.
\newblock {\em arXiv preprint arXiv:2010.13649}, 2020.

\bibitem{zhang2018deep}
L.~Zhang, J.~Han, H.~Wang, R.~Car, and W.~E.
\newblock Deep potential molecular dynamics: a scalable model with the accuracy
  of quantum mechanics.
\newblock {\em Physical Review Letters}, 120(14):143001, 2018.

\bibitem{zhu2015conservation}
Y.~Zhu, L.~Hong, Z.~Yang, and W.-A. Yong.
\newblock Conservation-dissipation formalism of irreversible thermodynamics.
\newblock {\em Journal of Non-Equilibrium Thermodynamics}, 40(2):67--74, 2015.

\end{thebibliography}

\end{document}